\documentclass[11pt]{article}
\usepackage{amsfonts,mathrsfs,amssymb,amsthm}
\usepackage{amsmath,amscd}
\usepackage{mathptm,pslatex}
\usepackage{url}
\usepackage{cite}
\usepackage{exscale}
\usepackage{relsize}
\usepackage[all]{xy}
\usepackage{color}

\begin{document}

\newtheorem{Def}{Definition}[section]
\newtheorem{Bsp}[Def]{Example}
\newtheorem{Prop}[Def]{Proposition}
\newtheorem{Theo}[Def]{Theorem}
\newtheorem{Lem}[Def]{Lemma}
\newtheorem{Koro}[Def]{Corollary}
\theoremstyle{definition}
\newtheorem{Rem}[Def]{Remark}
\renewcommand{\labelenumi}{\textup{(\theenumi)}}

\newcommand{\rad}{{\rm rad}}
\newcommand{\op}{{\rm op}}
\newcommand{\opp}{{\rm opp}}
\newcommand{\sg}{{\rm sg}}
\newcommand{\dg}{{\rm dg}}
\newcommand{\Hom}{{\rm Hom}}
\newcommand{\Ho}{{\rm Ho}}
\newcommand{\id}{{\rm id}}
\newcommand{\im}{{\rm im}}
\newcommand{\per}{{\rm per}}
\newcommand{\Cbar}{\overline C}
\newcommand{\Chat}{\widehat C}
\newcommand{\Sdg}{\mathbf S_{\rm dg}}
\newcommand{\perdg}{\mathbf {per}_{\rm dg}}
\newcommand{\Binf}{B_\infty}
\newcommand{\ra}{\rightarrow}
\newcommand{\lra}{\longrightarrow}
\newcommand{\xra}[1]{\xrightarrow{#1}}
\newcommand{\xla}[1]{\xleftarrow{#1}}
\newcommand{\tensor}{\otimes}

{\Large \bf
\begin{center}
Keller's conjecture for split finite-dimensional algebras
\end{center}
}

\medskip
\centerline{\textbf{Changchang Xi and Jinbi Zhang$^*$}}

\renewcommand{\thefootnote}{\alph{footnote}}
\setcounter{footnote}{-1}
\footnote{$^*$Corresponding author's Email: zhangjb@ahu.edu.cn}
\renewcommand{\thefootnote}{\alph{footnote}}
\setcounter{footnote}{-1}
\renewcommand{\thefootnote}{\alph{footnote}}
\setcounter{footnote}{-1}

\begin{abstract}
We prove Keller's conjecture for every split finite-dimensional algebra over an arbitrary field.
\end{abstract}

\section{Introduction}
Let $A$ be a finite-dimensional associative unital $k$-algebra over a field $k$.
Let $A\text{-mod}$ denote the category of finitely generated left $A$-modules, and let $\mathbf D^b(A\text{-mod})$ denote its bounded derived category.
A complex in $\mathbf D^b(A\text{-mod})$ is called \emph{perfect} if it is isomorphic to a bounded complex of finitely generated projective $A$-modules.
Following \cite{Buchweitz1986,Orlov2004}, the \emph{singularity category} $\mathbf D_{\sg}(A)$ of $A$ is the Verdier quotient of $\mathbf D^b(A\text{-mod})$ by its full subcategory of perfect complexes.
The \emph{dg singularity category} $\Sdg(A)$ is the dg quotient of the bounded dg derived category of $A\text{-mod}$ by its full dg subcategory of perfect complexes.
Its zeroth cohomology category is $\mathbf D_{\sg}(A)$; see \cite[Section~3]{Drinfeld2004} and \cite[Subsection~2.1]{Keller2018}.
We denote by $C^*(\Sdg(A))$ the Hochschild cochain complex of $\Sdg(A)$.
By \cite[Subsection~4.2]{Keller2003}, the complex $C^*(\Sdg(A))$ has a natural $B_\infty$-algebra structure.
The \emph{Hochschild cohomology} of $\Sdg(A)$ is the graded algebra with components
$$
\mathrm{HH}^n(\Sdg(A),\Sdg(A))
:=H^n\bigl(C^*(\Sdg(A))\bigr),
\qquad n\in\mathbb Z.
$$

Let $A^e:=A\otimes_k A^{\op}$ be the enveloping algebra of $A$, and let $\Sigma$ denote the suspension functor of $\mathbf D_{\sg}(A^e)$.
Following \cite{Wang2015}, the \emph{singular Hochschild cohomology} of $A$ is the graded algebra with components
$$
\mathrm{HH}_{\sg}^n(A,A)
:=
\Hom_{\mathbf D_{\sg}(A^e)}(A,\Sigma^n A),
\qquad n\in\mathbb Z.
$$
Let $\Cbar_{\sg,L}^*(A)$ denote the left singular Hochschild cochain complex of $A$.
The cohomology of $\Cbar_{\sg,L}^*(A)$ is naturally isomorphic to $\mathrm{HH}_{\sg}^*(A,A)$ as a graded algebra \cite[Theorem~3.6 and Proposition~4.7]{Wang2021}.
Moreover, $\Cbar_{\sg,L}^*(A)$ carries a $B_\infty$-algebra structure \cite[Theorem~5.1]{Wang2021}.

We write $\Ho_k(\Binf)$ for the homotopy category of $B_\infty$-algebras  (see \cite[Section~2]{Keller2003}).
Suppose that $A/\rad(A)$ is separable over $k$.
Keller constructed, in \cite[Theorem~2.3]{Keller2018} and \cite{Keller2019}, a canonical isomorphism of graded algebras
$$
\eta_A:\mathrm{HH}_{\sg}^*(A^{\op},A^{\op})
\xrightarrow{\ \sim\ }
\mathrm{HH}^*(\Sdg(A),\Sdg(A)).
$$
In \cite[Conjecture~1.2]{Keller2018}, Keller conjectured that this
canonical isomorphism lifts to an isomorphism
$$
F_A:\Cbar_{\sg,L}^*(A^{\op})\longrightarrow C^*(\Sdg(A))
$$
in $\Ho_k(\Binf)$ whose induced map on cohomology is $\eta_A$.
We call the existence of an isomorphism between these two
$B_\infty$-algebras in $\Ho_k(\Binf)$, without requiring its induced
map on cohomology to be $\eta_A$, the \emph{weak form of Keller's conjecture}.

A finite-dimensional $k$-algebra $A$ is called \emph{split} if $A/\rad(A)$ is split semisimple, equivalently, if there is a $k$-algebra isomorphism
$
A/\rad(A)\simeq\prod_{i=1}^r M_{n_i}(k)
$
for some positive integers $n_1,\ldots,n_r$.
The weak form of Keller's conjecture holds for an algebra if and only if it holds for
each of its one-point (co)extensions
\cite[Theorem~14.4(1)--(2)]{ChenLiWang2025}.
If two algebras are singularly equivalent of Morita type with level,
then one satisfies the weak form of Keller's conjecture if and only if so does the other
\cite[Theorem~14.4(3)]{ChenLiWang2025}.
The weak form of Keller's conjecture has been verified for split algebras with radical
square zero \cite[Corollary~14.6]{ChenLiWang2025}.

In this article, we prove Keller's conjecture for all split finite-dimensional $k$-algebras.

\begin{Theo}\label{thm:main}
Let $A$ be a split finite-dimensional $k$-algebra. There is an isomorphism
$$
 F_A:\Cbar_{\sg,L}^*(A^{\op})\longrightarrow C^*(\Sdg(A))
 \quad\hbox{in }\Ho_k(\Binf)
$$
whose induced map on cohomology is Keller's canonical map $\eta_A$.
\end{Theo}

This paper is structured as follows.
In Section 2, we recall the necessary definitions and results.
In Section 3, we prove Theorem~\ref{thm:main}.

\section{Preliminaries}\label{sec:preliminaries}
In this section, we recall the notation, definitions and results
used in the proof of the main theorem.
Let $k$ be a field.
All algebras are unital $k$-algebras, all categories are $k$-categories and all functors are $k$-linear.
All tensor products are taken over $k$ unless a subscript is displayed.

\subsection{$B_\infty$-algebras}\label{pre:binfty}
For a graded vector space $V=\bigoplus_{i\in\mathbb Z}V^i$ and $v\in V^i$, we write $|v|=i$.
The suspension of $V$ is the graded vector space $sV$ defined by $(sV)^i=V^{i+1}$ for every $i\in\mathbb Z$.
For $v\in V^i$, we denote the corresponding element of $(sV)^{i-1}$ by $sv$.
Thus $|sv|=|v|-1$, and the map $s:V\to sV$, $v\mapsto sv$, has degree $-1$.
The inverse map $s^{-1}:sV\to V$, $sv\mapsto v$, has degree $1$.

For a complex $X$ of $k$-vector spaces and $n\in\mathbb Z$, put
$X[n]^i:=X^{i+n}$ and $d_{X[n]}:=(-1)^nd_X$.
Thus $sX=X[1]$ with $d_{sX}(sx)=-s(d_Xx)$.
For a homogeneous element $x\in X$, write $s^nx$ for its image in $X[n]$.
We also write $\Sigma^nX:=X[n]$ in the derived and singularity categories.

For a graded vector space $V$, put
$$
 T^c(sV):=\bigoplus_{r\geq0}(sV)^{\tensor r}
          =k\oplus sV\oplus(sV)^{\tensor2}\oplus\cdots.
$$
We identify $(sV)^{\tensor0}$ with $k$ and write $1:=1_k$.
The \emph{deconcatenation coproduct} $\Delta:T^c(sV)\to T^c(sV)\tensor T^c(sV)$ is defined by $\Delta(1):=1\tensor1$ and, for $r\geq1$ and $v_1,\ldots,v_r\in V$, by
$$
\begin{aligned}
 \Delta(sv_1\tensor\cdots\tensor sv_r)
 :={}&1\tensor(sv_1\tensor\cdots\tensor sv_r)\\
 &+\sum_{i=1}^{r-1}
 (sv_1\tensor\cdots\tensor sv_i)
 \tensor
 (sv_{i+1}\tensor\cdots\tensor sv_r)\\
 &+(sv_1\tensor\cdots\tensor sv_r)\tensor1.
\end{aligned}
$$

An element $z\in T^c(sV)$ is called \emph{primitive} if $\Delta(z)=1\tensor z+z\tensor1$.
The subspace of all primitive elements of $T^c(sV)$ is exactly $sV$.
Let $\pi:T^c(sV)\to sV$ be the canonical projection onto the summand $sV$ of $T^c(sV)$.
Then every primitive element $z$ satisfies $z=\pi(z)$.

Let $U,V,U',V'$ be graded $k$-vector spaces, and let
$\varphi:U\to V$ and $\psi:U'\to V'$ be homogeneous $k$-linear maps
of degrees $|\varphi|$ and $|\psi|$, respectively; thus
$\varphi(U^i)\subseteq V^{i+|\varphi|}$ and
$\psi((U')^i)\subseteq(V')^{i+|\psi|}$.
The tensor product $\varphi\tensor\psi:U\tensor U'\to V\tensor V'$ is defined by
$$
(\varphi\tensor\psi)(x\tensor y)
=
(-1)^{|\psi||x|}\varphi(x)\tensor\psi(y)
$$
for homogeneous $x\in U$ and $y\in U'$.

\begin{Def}\label{def:B-infinity}
\normalfont {\rm(\cite[Subsection~5.2]{GetzlerJones1994})} A \emph{$B_\infty$-algebra} is a graded vector space $V$ equipped with maps $M:T^c(sV)\tensor T^c(sV)\to T^c(sV)$ and $D:T^c(sV)\to T^c(sV)$ satisfying the following conditions.

{\rm(1)} The map $M$ has degree zero and defines an associative multiplication on $T^c(sV)$ with unit $1\in k=(sV)^{\tensor0}$.
Thus, for homogeneous $x,y,z\in T^c(sV)$, $M(M(x,y),z)=M(x,M(y,z))$ and $M(1,x)=x=M(x,1)$.
Moreover, $M$ is a morphism of coalgebras.
More precisely, let $\mathfrak t:T^c(sV)\tensor T^c(sV)\to T^c(sV)\tensor T^c(sV)$ be the graded interchange map defined by $\mathfrak t(x\tensor y):=(-1)^{|x||y|}y\tensor x$ for homogeneous $x,y\in T^c(sV)$.
Then $\Delta M =(M\tensor M)({\rm id}\tensor\mathfrak t\tensor{\rm id})(\Delta\tensor\Delta)$.

{\rm(2)} The map $D$ has degree one and satisfies the following conditions:

{\rm(i)} The map $D$ vanishes on the unit and is a coderivation of $\Delta$, that is, $D(1)=0$ and $\Delta D=(D\tensor{\rm id}+{\rm id}\tensor D)\Delta$.

{\rm(ii)} The map $D$ is a derivation of $M$, that is, for homogeneous $x,y\in T^c(sV)$,
$$
 D(M(x,y))
 =M(D(x),y)+(-1)^{|x|}M(x,D(y)).
$$

{\rm(iii)} One has $D^2=0$.

Equivalently, $(T^c(sV),\Delta,D,M)$ is a dg bialgebra whose underlying coalgebra is the tensor coalgebra $T^c(sV)$ and whose unit is $1\in k=(sV)^{\tensor0}$.
We say that the $B_\infty$-algebra $V$ is \emph{represented by} $(T^c(sV),\Delta,D,M)$.
Define $M^{\opp}:T^c(sV)\tensor T^c(sV)\to T^c(sV)$ by
$             M^{\opp}(x,y):=(-1)^{|x||y|}M(y,x)$
for homogeneous $x,y\in T^c(sV)$.
By \cite[Lemma~5.6 and Definition~5.7]{ChenLiWang2025},
the \emph{opposite $B_\infty$-algebra} $V^{\opp}$ is represented by
$(T^c(sV),\Delta,D,M^{\opp})$.
\end{Def}

\begin{Def}\label{def:brace-B-infinity}
\normalfont Let $V$ be a $B_\infty$-algebra represented by $(T^c(sV),\Delta,D,M)$.
For $r\geq1$ and $p,q\geq0$, put
$$
\begin{aligned}
 D_r&:=\pi D\big|_{(sV)^{\tensor r}}
      :(sV)^{\tensor r}\longrightarrow sV,\\
 M_{p,q}&:=\pi M\big|_{(sV)^{\tensor p}\tensor(sV)^{\tensor q}}
      :(sV)^{\tensor p}\tensor(sV)^{\tensor q}\longrightarrow sV.
\end{aligned}
$$
With the convention $V^{\tensor0}=(sV)^{\tensor0}=k,\; s^{\tensor0}:={\rm id}_k,$ define
$$
\begin{aligned}
 m_r
 &:=
 s^{-1}\circ D_r\circ s^{\tensor r}
 :V^{\tensor r}\longrightarrow V,
 &&\deg(m_r)=2-r,\\
 \mu_{p,q}
 &:=
 s^{-1}\circ M_{p,q}\circ
 (s^{\tensor p}\tensor s^{\tensor q})
 :V^{\tensor p}\tensor V^{\tensor q}\longrightarrow V,
 &&\deg(\mu_{p,q})=1-p-q.
\end{aligned}
$$
The coderivation $D$ is determined by the operations $m_r$, whereas the coalgebra multiplication $M$ is determined by the operations $\mu_{p,q}$.
Since $V^{\opp}$ has the same coderivation $D$, its operations $m_r$ agree with those of $V$.
In particular, $m_1=s^{-1}\pi Ds$ is the differential $d_V$ of the underlying complex of $V$.

Following \cite[Definition~5.12]{ChenLiWang2025}, the $B_\infty$-algebra is called a \emph{brace $B_\infty$-algebra} if $m_r=0$ for $r>2$ and $\mu_{p,q}=0$ for $p>1$.
Thus the operations $m_r$ reduce to the differential $m_1$ and the multiplication $m_2$, while the only possibly nonzero maps $\mu_{p,q}$ are the maps $\mu_{1,q}$ for $q\geq1$ and the unit components $\mu_{1,0}={\rm id}_V=\mu_{0,1}$.

Hence the brace $B_\infty$-algebra is determined by $m_1$, $m_2$, the operations $\mu_{1,q}$ for $q\geq1$, and the unit components.
Following \cite[Definition~1]{GerstenhaberVoronov1995}, for $q\geq1$ and homogeneous elements $a,b_1,\ldots,b_q\in V$, define the \emph{brace operation} by
$$
 a\{b_1,\ldots,b_q\}
 :=(-1)^{q|a|+(q-1)|b_1|+\cdots+|b_{q-1}|}
 \mu_{1,q}(a\tensor b_1\tensor\cdots\tensor b_q).
$$
\end{Def}

Thus the brace operations and the maps $\mu_{1,q}$ determine each other.
For a brace $B_\infty$-algebra $V$, $(V,m_1,m_2)$ is a dg algebra.
For homogeneous elements $f$ and $g$ of $V$, their \emph{Gerstenhaber bracket} is
$$
 [f,g]:=f\{g\}
 -(-1)^{(|f|-1)(|g|-1)}g\{f\}.
$$
The sign conventions above agree with Definition~5.12 and Lemma~5.18 of \cite{ChenLiWang2025}.
We will use the following consequence: two brace $B_\infty$-algebras on the same graded vector space with the same unit and the same brace operations have the same coalgebra multiplication $M$.
Indeed, the braces determine $\mu_{1,q}$, the conditions above make $\mu_{p,q}=0$ for $p>1$, and the unit determines the components with $p=0$ or $q=0$.
Hence the two maps obtained by composing the multiplications with $\pi:T^c(sV)\to sV$ are equal.
Since a coalgebra map to $T^c(sV)$ is determined by its composite with $\pi$, the two coalgebra multiplications are equal.

\begin{Def}\label{def:B-infinity-morphism}
\normalfont {\rm(\cite[Section~2]{Keller2003})} Let $V$ and $W$ be $B_\infty$-algebras represented by
$$
 (T^c(sV),\Delta_V,D_V,M_V)
 \qquad\text{and}\qquad
 (T^c(sW),\Delta_W,D_W,M_W),
$$
respectively.
A \emph{$B_\infty$-morphism} from $V$ to $W$ is a degree-zero map $F:T^c(sV)\to T^c(sW)$ satisfying
$
 F(1)=1,
 \Delta_WF=(F\tensor F)\Delta_V,
 FM_V=M_W(F\tensor F),
 FD_V=D_WF.
$
Thus $F$ is a morphism of dg bialgebras.
Let $\pi_W:T^c(sW)\to sW$ be the projection.
For every $r\geq1$, define
$$
 f_r:=s^{-1}\circ\pi_W\circ
 F\big|_{(sV)^{\tensor r}}\circ s^{\tensor r}
 :V^{\tensor r}\longrightarrow W.
$$
Then $f_r$ has degree $1-r$.
The map $F$ and the family $(f_r)_{r\geq1}$ determine each other.
We write the corresponding $B_\infty$-morphism as $f=(f_r)_{r\geq1}:V\to W$, call the maps $f_r$ its \emph{Taylor components}, and say that $F$ represents $f$.
Thus $f$ is a family of graded multilinear maps; only $f_1$ is a graded linear map $V\to W$.

The morphism $f$ is \emph{strict} if $f_r=0$ for every $r>1$, and it is a \emph{$B_\infty$-quasi-isomorphism} if $f_1$ is a quasi-isomorphism of complexes.
A \emph{strict $B_\infty$-quasi-isomorphism} has both properties.
A $B_\infty$-morphism $f:V\to W$ is called a \emph{$B_\infty$-isomorphism} if there is a $B_\infty$-morphism $g:W\to V$ such that $gf={\rm id}_V$ and $fg={\rm id}_W$.
\end{Def}

The $B_\infty$-algebras and their $B_\infty$-morphisms form a category, denoted by $\Binf$.
We write $\Ho_k(\Binf)$ for the localization of $\Binf$ with respect to the class of $B_\infty$-quasi-isomorphisms (see \cite[Section~2]{Keller2003}).
Hence every $B_\infty$-quasi-isomorphism becomes an isomorphism in $\Ho_k(\Binf)$.

\subsection{Hochschild and singular Hochschild cochains}\label{pre:cochains}
For a $k$-algebra $A$, let $X,Y$ be complexes of left $A$-modules,
with differentials $d_X^i:X^i\to X^{i+1}$ and
$d_Y^i:Y^i\to Y^{i+1}$. Put
$\Hom_A^n(X,Y):=\prod_i\Hom_A(X^i,Y^{i+n})$, with differential
$d_{\Hom}(f):=d_Yf-(-1)^nfd_X$ for $f\in\Hom_A^n(X,Y)$.
The complex $\operatorname{End}_A(X):=\Hom_A(X,X)$ is a dg algebra
with multiplication given by composition.
For a complex $X$, put
$Z^n(X):=\ker d_X^n$, $B^n(X):=\operatorname{im}d_X^{n-1}$
and $H^n(X):=Z^n(X)/B^n(X)$.
For a chain map $i:X\to Y$, put
$\operatorname{Cone}(i)^n:=Y^n\oplus X^{n+1}$, with differential
$d(y,x):=(d_Yy+i(x),-d_Xx)$ for $y\in Y^n$ and $x\in X^{n+1}$.
The associated triangle is
$$
 X\xrightarrow{i}Y\xrightarrow{j}\operatorname{Cone}(i)
 \xrightarrow{p}X[1],
$$
where $j^n(y):=(y,0)$ and $p^n(y,x):=x$ for
$n\in\mathbb Z$, $y\in Y^n$ and $x\in X^{n+1}$.
In particular, for a degreewise split exact sequence of complexes
of left $A$-modules,
\begin{equation}\label{pre:cone}
\begin{gathered}
0\longrightarrow X\longrightarrow Y\longrightarrow Z\longrightarrow0,
\qquad
Y=X\oplus Z,\quad d_Y=\begin{pmatrix}d_X&b\\0&d_Z\end{pmatrix},\\
X\longrightarrow Y\longrightarrow Z\xrightarrow{-b}X[1]
\end{gathered}
\end{equation}
is its triangle in the derived category, where $b\in\Hom_A^1(Z,X)$
\cite[Sections~1.5 and~10.1]{Weibel1994}.

For a dg algebra $D$, let $D^{\op}$ denote its \emph{dg opposite}.
It has the same differential as $D$, and its multiplication is
$x\cdot_{\op}y:=(-1)^{|x||y|}yx$ for homogeneous $x,y\in D$;
see \cite[Section~2.1]{ChenLiWang2025}.

Put $D^e:=D\otimes D^{\op}$. For an ordinary algebra $A$ concentrated
in degree zero, this construction gives its ordinary opposite $A^{\op}$.

For a dg algebra $D$, let $\mathbf D(D)$ denote the derived category
of left dg $D$-modules. A right dg $D$-module $U$ is regarded as a left
dg $D^{\op}$-module by
$$
 a^{\op}u:=(-1)^{|a||u|}ua
 \qquad(a\in D, u\in U\text{ homogeneous}).
$$
For right dg $D$-modules $U,V$, the morphism complex is therefore
written $\Hom_{D^{\op}}(U,V)$.
A dg $D$-module $P$ is K-projective if
$\Hom_D(P,X)$ is acyclic for every acyclic dg $D$-module $X$.
A dg right $D$-module $P$ is K-flat if $P\otimes_DX$ is acyclic
for every acyclic dg left $D$-module $X$; the left version is defined
in the same way. We use the property-(P) criterion of
\cite[Section~3.1]{Keller1994}: an exhaustive filtration, split on
graded modules, whose quotients are sums of shifts of summands of
the free dg module makes its union K-projective.
Such a filtration will be called a semifree filtration.
In particular, for an ordinary algebra $A$, bounded-above complexes
of projective $A$-modules are K-projective. Derived tensor products are denoted by
$\otimes_D^{\mathbf L}$.
For a class $\mathcal X$ of objects in a triangulated category,
$\operatorname{thick}(\mathcal X)$ denotes the smallest full subcategory
containing $\mathcal X$ and closed under isomorphisms, shifts, cones and
direct summands. When the ambient category has arbitrary coproducts,
$\operatorname{Loc}(\mathcal X)$ denotes the smallest such subcategory
also closed under arbitrary coproducts.
An object $P$ is compact if $\Hom(P,-)$ commutes with arbitrary
coproducts. In $\mathbf D(D)$ the compact objects form
$\operatorname{thick}(D)$ \cite[Section~5.3]{Keller1994}.

Filtered colimits of $k$-vector spaces are exact
\cite[Lemma~2.6.14 and Theorem~2.6.15]{Weibel1994}.
Consequently, for a filtered direct system $(X_p,u_{p,q})$ of
complexes of $k$-vector spaces, the canonical map
$$
 \varinjlim_p H^n(X_p)\xrightarrow{\sim}
 H^n\!\left(\varinjlim_pX_p\right)
 \qquad(n\in\mathbb Z)
$$
is an isomorphism.

Let $\mathcal A$ be a small dg category.
We write $C^*(\mathcal A)$ for its Hochschild cochain complex; see \cite[Subsection~4.2]{Keller2003}.
We regard a dg algebra $D$ as a dg category with one object and use $C^*(D)$ for its Hochschild cochain complex.
Put $\operatorname{per}(D):=\operatorname{thick}_{\mathbf D(D)}(D)$,
and call its objects perfect dg $D$-modules. We write $\perdg(D)$
for a dg enhancement by K-projective representatives
\cite[Sections~5 and~7.3]{Keller1994}.
Recall that a dg functor is a \emph{quasi-equivalence} if its maps on
morphism complexes are quasi-isomorphisms and its induced functor
on $H^0$ is an equivalence \cite[Subsection~2.3]{Drinfeld2004}.
A dg functor is a dg Morita equivalence if derived extension of
scalars induces an equivalence of the derived categories of dg
modules \cite[Section~4.6]{Keller2006}.
For a small dg category $\mathcal A$, its enveloping dg category is
$\mathcal A^e:=\mathcal A\otimes_k\mathcal A^{\op}$.
Its objects are pairs $(U,V)$ of objects of $\mathcal A$, and
$$
 \mathcal A^e((U,V),(U',V'))
 :=\mathcal A(U,U')\otimes_k\mathcal A(V',V).
$$
Composition uses the tensor sign convention and dg opposite.
Let $\mathbf D(\mathcal A^e)$ denote the derived category of left
dg $\mathcal A^e$-modules, or dg $\mathcal A$-bimodules.
We write $X(U,V)$ for the value of such a module at $(V,U)$;
thus $X(U,V)$ is contravariant in $U$ and covariant in $V$.
The diagonal bimodule is $I_{\mathcal A}(U,V):=\mathcal A(U,V)$.
A dg quasi-functor between small dg categories is represented by
a dg bimodule whose value at each source object is quasi-isomorphic
to a representable module over the target
\cite[Section~4.2]{Keller2006}.

\begin{Lem}\label{lem:perfect-restriction}
\normalfont {\rm(\cite[Theorem~4.6(c)]{Keller2003})} For a dg algebra $D$, the restriction map
$$
 C^*(\perdg(D^{\op}))\longrightarrow C^*(D)
$$
is an isomorphism in $\Ho_k(\Binf)$.
\end{Lem}

\begin{Lem}\label{lem:hochschild-quasi-equivalence}
\normalfont {\rm(\cite[Theorem~4.6(b)]{Keller2003})}
The Hochschild cochain $B_\infty$-algebras of quasi-equivalent
small dg categories are isomorphic in $\Ho_k(\Binf)$.
\end{Lem}

For a finite-dimensional algebra $A$, put $\overline A_k:=A/(k\cdot1_A)$.
For a complex $M$ of $A$-bimodules, let $\Cbar^*(A,M)$ denote the
normalized Hochschild cochain complex with coefficients in $M$
\cite[Section~6.1]{ChenLiWang2025}.
For $p\geq0$, define the graded $A$-$A$-bimodules of right and left noncommutative differential $p$-forms by
$$
 \Omega_{\mathrm{nc},R}^p(A)
 :=(s\overline A_k)^{\tensor p}\tensor A,
 \qquad
 \Omega_{\mathrm{nc},L}^p(A)
 :=A\tensor(s\overline A_k)^{\tensor p}.
$$
The canonical $A$-$A$-bimodule structures are defined in \cite[Section~8.1]{ChenLiWang2025}.

\begin{Def}
\normalfont {\rm(\cite[Section~8.1]{ChenLiWang2025})} The \emph{normalized right singular Hochschild cochain complex} of $A$ is
$$
 \Cbar_{\sg,R}^*(A)
 :=
 \varinjlim_{p\geq0}
 \Cbar^*\bigl(A,\Omega_{\mathrm{nc},R}^p(A)\bigr),
$$
where the transition maps are
$$
 \begin{aligned}
 \theta_{p,R}:
 \Cbar^*\bigl(A,\Omega_{\mathrm{nc},R}^p(A)\bigr)
 &\longrightarrow
 \Cbar^*\bigl(A,\Omega_{\mathrm{nc},R}^{p+1}(A)\bigr),\\
 f&\longmapsto {\rm id}_{s\overline A_k}\tensor f.
 \end{aligned}
$$
\end{Def}

\begin{Def}
\normalfont {\rm(\cite[Definition~3.2]{Wang2021})} The \emph{normalized left singular Hochschild cochain complex} of $A$ is
$$
 \Cbar_{\sg,L}^*(A)
 :=
 \varinjlim_{p\geq0}
 \Cbar^*\bigl(A,\Omega_{\mathrm{nc},L}^p(A)\bigr),
$$
where the transition maps are
$$
 \begin{aligned}
 \theta_{p,L}:
 \Cbar^*\bigl(A,\Omega_{\mathrm{nc},L}^p(A)\bigr)
 &\longrightarrow
 \Cbar^*\bigl(A,\Omega_{\mathrm{nc},L}^{p+1}(A)\bigr),\\
 f&\longmapsto f\tensor{\rm id}_{s\overline A_k}.
 \end{aligned}
$$
\end{Def}

By \cite[Theorem~8.8]{ChenLiWang2025} and \cite[Theorem~5.1]{Wang2021}, respectively, the Hochschild differentials, cup products and brace operations on the terms of the right and left direct systems make $\Cbar_{\sg,R}^*(A)$ and $\Cbar_{\sg,L}^*(A)$ brace $B_\infty$-algebras.

A split finite-dimensional algebra $B$ is basic if
$B/\rad(B)\simeq k^r$.

\begin{Def}\label{def:relative-singular}
\normalfont
{\rm(\cite[Sections~6.2 and~8.3]{ChenLiWang2025})} Let $A=E\oplus J$ be a split basic finite-dimensional $k$-algebra, where $J=\rad(A)$ and $E=\bigoplus_{i=1}^n ke_i$ for pairwise orthogonal idempotents $e_1,\ldots,e_n$ with $\sum_{i=1}^n e_i=1$.
Its \emph{normalized $E$-relative bar resolution} is
$$
 \overline{\operatorname{Bar}}_E(A)
 :=
 \bigoplus_{r\geq0}
 A\tensor_E(sJ)^{\tensor_E r}\tensor_EA,
 \qquad
 \overline{\operatorname{Bar}}_E^{-r}(A)
 :=A\tensor_E(sJ)^{\tensor_E r}\tensor_EA.
$$
For $r\geq1$, its differential is defined by
$$
\begin{aligned}
d_{\mathrm{bar}}
 (a_0\tensor sj_1\tensor\cdots\tensor sj_r\tensor a_{r+1})
:=&a_0j_1\tensor sj_2\tensor\cdots\tensor sj_r\tensor a_{r+1}\\
+&\sum_{i=1}^{r-1}(-1)^i
 a_0\tensor sj_1\tensor\cdots\tensor s(j_ij_{i+1})\tensor\cdots
 \tensor sj_r\tensor a_{r+1}\\
+&(-1)^r a_0\tensor sj_1\tensor\cdots\tensor sj_{r-1}
 \tensor j_ra_{r+1}.
\end{aligned}
$$
The map to the diagonal bimodule is
$$
 \varepsilon:\overline{\operatorname{Bar}}_E^0(A)
 =A\tensor_EA\longrightarrow A,
 \qquad a_0\tensor a_1\longmapsto a_0a_1.
$$
For $p\geq0$, define the \emph{$E$-relative right noncommutative differential $p$-forms} by
$$
 \Omega_{\mathrm{nc},R,E}^p(A)
 :=\operatorname{coker}\bigl(
 d_{\mathrm{bar}}:
 \overline{\operatorname{Bar}}_E^{-p-1}(A)\longrightarrow
 \overline{\operatorname{Bar}}_E^{-p}(A)\bigr).
$$
There is an identification of $E$-$A$-bimodules
$$
 \Omega_{\mathrm{nc},R,E}^p(A)
 \simeq (sJ)^{\tensor_Ep}\tensor_EA.
$$
For $p=0$, $\Omega_{\mathrm{nc},R,E}^0(A)$ is the regular $A$-$A$-bimodule $A$.
Let $p\geq1$, $a_0,a_{p+1},b\in A$, and $x_1,\ldots,x_p\in J$.
Write $\pi_J:E\oplus J\to J$ for the projection.
The right $A$-action is defined by
$$(sx_1\tensor\cdots\tensor sx_p\tensor a_{p+1})
 b:={}
 sx_1\tensor\cdots\tensor sx_p\tensor a_{p+1}b,
$$
and the left $A$-action is defined by
$$
\begin{aligned}
a_0\blacktriangleright
 (sx_1\tensor\cdots\tensor sx_p\tensor a_{p+1})
:=&s\pi_J(a_0x_1)\tensor sx_2\tensor\cdots\tensor sx_p
       \tensor a_{p+1}\\
+&\sum_{i=1}^{p-1}(-1)^i
 s\pi_J(a_0)\tensor sx_1\tensor\cdots
 \tensor s(x_ix_{i+1})\tensor\cdots\tensor sx_p
 \tensor a_{p+1}\\
+&(-1)^p s\pi_J(a_0)\tensor sx_1\tensor\cdots\tensor sx_{p-1}
       \tensor x_pa_{p+1}.
\end{aligned}
$$
An empty sum is zero, and every product in the formulas defining the two actions is taken in $A$.

For $m+p\geq0$, put
$$
 \Cbar_E^m\bigl(A,\Omega_{\mathrm{nc},R,E}^p(A)\bigr)
 :=\Hom_{E-E}\bigl(
 (sJ)^{\tensor_E(m+p)},
 (sJ)^{\tensor_Ep}\tensor_EA\bigr),
$$
and put this space equal to zero when $m+p<0$.
For every $m\in\mathbb Z$, the $E$-relative Hochschild differential has degree one:
$$
 \begin{aligned}
 d_{A,p}:\Cbar_E^m\bigl(A,\Omega_{\mathrm{nc},R,E}^p(A)\bigr)
 &\longrightarrow
 \Cbar_E^{m+1}\bigl(A,\Omega_{\mathrm{nc},R,E}^p(A)\bigr),\\
 f&\longmapsto d_{A,p}f.
 \end{aligned}
$$
Let $f\in\Cbar_E^m(A,\Omega_{\mathrm{nc},R,E}^p(A))$, put $n:=m+p$, and let $x_1,\ldots,x_{n+1}\in J$.
Then
$$
\begin{aligned}
(d_{A,p}f)(sx_1\tensor\cdots\tensor sx_{n+1})
:=&-(-1)^m x_1\blacktriangleright
 f(sx_2\tensor\cdots\tensor sx_{n+1})\\
+&\sum_{i=1}^{n}(-1)^{m-i+1}
 f(sx_1\tensor\cdots\tensor s(x_ix_{i+1})
   \tensor\cdots\tensor sx_{n+1})\\
+&(-1)^p f(sx_1\tensor\cdots\tensor sx_n)x_{n+1}.
\end{aligned}
$$
The transition maps are
$$
 \begin{aligned}
 \theta_{p,R,E}:
 \Cbar_E^*\bigl(A,\Omega_{\mathrm{nc},R,E}^p(A)\bigr)
 &\longrightarrow
 \Cbar_E^*\bigl(A,\Omega_{\mathrm{nc},R,E}^{p+1}(A)\bigr),\\
 f&\longmapsto {\rm id}_{sJ}\tensor_E f.
 \end{aligned}
$$
They satisfy
$$
 d_{A,p+1}\theta_{p,R,E}
 =\theta_{p,R,E}d_{A,p}.
$$
The \emph{$E$-relative right singular Hochschild cochain complex} is
$$
 \Cbar_{\sg,R,E}^*(A)
 :=\varinjlim_{p\geq0}
 \Cbar_E^*\bigl(A,\Omega_{\mathrm{nc},R,E}^p(A)\bigr).
$$
Its differential is
$$
 d_A([f]):=[d_{A,p}(f)]
 \qquad
 \bigl(f\in
 \Cbar_E^*\bigl(A,\Omega_{\mathrm{nc},R,E}^p(A)\bigr)\bigr).
$$
Its cup product and braces are induced by the formulas in \cite[Section~8]{ChenLiWang2025}.
\end{Def}

By \cite[Theorem~8.8 and Section~8.3]{ChenLiWang2025}, the cup product and braces make $\Cbar_{\sg,R,E}^*(A)$ a brace $B_\infty$-algebra.
In the conventions of \cite[Section~6.1]{ChenLiWang2025}, its operation $m_1$ is the differential $d_A$, and its operation $m_2$ is the cup product.
For homogeneous cochains $f$ and $g$, we write $f\smile g:=m_2(f,g)$.

The same definitions apply to a finite-dimensional algebra
$A=E\oplus I$, with $I$ in place of $J$; here $E=\bigoplus_i ke_i$
and $I$ is a two-sided ideal. For $E=k1$, use $A/k1$ in the bar
slots and omit the relative subscript from the forms and cochains.
The cokernel
$\Omega_{\mathrm{nc},R,E}^p(A)$ is concentrated in degree $-p$;
its shift $\Omega_{\mathrm{nc},R,E}^p(A)[-p]$ is the ordinary
$p$th syzygy of the diagonal bimodule. In particular,
$$
 \Cbar_E^*\bigl(A,\Omega_{\mathrm{nc},R,E}^p(A)\bigr)
 =\Hom_{A^e}\bigl(\overline{\operatorname{Bar}}_E(A),
                       \Omega_{\mathrm{nc},R,E}^p(A)\bigr).
$$
The left form coordinates of the same cokernel are
$\Omega_{\mathrm{nc},L,E}^p(A):=A\otimes_E(sI)^{\otimes_Ep}$.
The transition $f\mapsto f\otimes_E\id_{sI}$ defines
$\Cbar_{\sg,L,E}^*(A)$. The cup products and braces on both
relative complexes are those of \cite[Section~8.3]{ChenLiWang2025}.
For a homogeneous right-form cochain $f$ of degree $n$ and input
length $m=n+p$, the tensor convention gives
\begin{equation}\label{pre:right-transition}
 (\theta_{p,R,E}f)(sa_0,\ldots,sa_m)
 =(-1)^nsa_0\otimes_Ef(sa_1,\ldots,sa_m).
\end{equation}
For dg algebras, the relative bar differential also includes the
internal tensor differential \cite[Section~6.2]{ChenLiWang2025}.

\begin{Lem}\label{lem:left-right-singular}
\normalfont {\rm(\cite[Theorem~5.10, Proposition~8.10 and
Section~8.3]{ChenLiWang2025})}
Let $A=E\oplus I$ be a finite-dimensional algebra, where
$E=\bigoplus_i ke_i$ is a unital subalgebra and $I$ is a two-sided ideal.
There is a $B_\infty$-isomorphism
$$
 \tau_E:\Cbar_{\sg,L,E}^*(A)
 \longrightarrow\bigl(\Cbar_{\sg,R,E}^*(A^{\op})\bigr)^{\opp}.
$$
For input length $m$ and form length $p$, its first component
reverses the input and form words with sign
$(-1)^{\binom{m+1}{2}-\binom{p+1}{2}}$.
The absolute version applies to every finite-dimensional algebra $A$:
$$
 \tau_k:\Cbar_{\sg,L}^*(A)
 \xrightarrow{\sim}\bigl(\Cbar_{\sg,R}^*(A^{\op})\bigr)^{\opp}.
$$
The morphisms $\tau_E$ and $\tau_k$ are obtained by composing the
swap map of \cite[Proposition~8.10]{ChenLiWang2025} with the
$B_\infty$-isomorphism of \cite[Theorem~5.10]{ChenLiWang2025}.
In particular, $\tau_E$ denotes the full $B_\infty$-morphism,
not only its first component.
\end{Lem}

\begin{Lem}\label{src:syzygy-stabilization}
\normalfont {\rm(\cite[Lemma~2.4]{Keller2018};
\cite[Section~9.3]{ChenLiWang2025})}
Let $A$ be a finite-dimensional algebra, and let $M$ be a finite
left $A$-module. Choose a resolution by finitely generated
projective modules and write its syzygy sequences as
$$
 0\longrightarrow\Omega^{p+1}M\longrightarrow P_p
 \longrightarrow\Omega^pM\longrightarrow0,\qquad \Omega^0M=M.
$$
Put $W_p:=\Omega^pM[p]$, and let
$\delta_p:\Omega^pM\to\Omega^{p+1}M[1]$ be the connecting
morphism of the unshifted sequence, with convention \eqref{pre:cone}.
Use the connecting morphism of the sequence after shifting all
three terms by $p$:
$$
\partial_p:=(-1)^p\delta_p[p]:W_p\longrightarrow W_{p+1},
 \qquad t_p:=\partial_{p-1}\cdots\partial_0,\quad t_0:=\id_M.
$$
If $\mathsf q:\mathbf D^b(A\text{-mod})\to\mathbf D_{\sg}(A)$
is the quotient functor, then for every
$X\in\mathbf D^b(A\text{-mod})$ and $n\in\mathbb Z$,
$$
 \varinjlim_p\Hom_{\mathbf D^b(A\text{-mod})}(X,W_p[n])
 \xrightarrow{\sim}\Hom_{\mathbf D_{\sg}(A)}(X,M[n]),
 \qquad [f]\longmapsto
 \mathsf q(t_p[n])^{-1}\mathsf q(f).
$$
The transition is postcomposition with $\partial_p[n]$.
\end{Lem}

For a chain map $i:X\to Y$ of complexes of left $A$-modules,
the sign in Lemma~\ref{src:syzygy-stabilization} follows from the chain isomorphism
$$
 \operatorname{Cone}(i[p])\longrightarrow\operatorname{Cone}(i)[p],
 \qquad(y,x)\longmapsto(y,(-1)^px).
$$
Here $y\in Y^{n+p}$ and $x\in X^{n+p+1}$ in degree $n$.
We next identify the map induced by $\partial_p[n]$ with the
singular-cochain transition. Let $A=E\oplus I$
have the decomposition specified after Definition~\ref{def:relative-singular}, and put
$M_p:=\Omega_{\mathrm{nc},R,E}^p(A)[-p]$ and use the form sequence
$$
 0\longrightarrow M_{p+1}\xrightarrow{j_p}A\otimes_EM_p
 \xrightarrow{\mathrm{act}}M_p\longrightarrow0,
 \qquad j_p(a\otimes z):=a\otimes z-1\otimes az.
$$
Here $a\in I$ and $z\in M_p$. For a closed degree-$n$ right-form
cochain $f$ at stage $p$, let
$\widetilde f$ be the degree-$n$ $A$-bimodule map on the bar resolution
whose value on normalized inputs is $1\otimes f$, in
$(A\otimes_EM_p)[p]$. Let $m:=n+p$ and $a_0,\ldots,a_m\in I$.
The first bar face and $df=0$ give
$$
 \begin{aligned}
 (d\widetilde f)(sa_0,\ldots,sa_m)
 &=(-1)^{n+1}j_p\bigl(a_0\otimes
                     f(sa_1,\ldots,sa_m)\bigr),\\
 -j_p^{-1}d\widetilde f&=\theta_{p,R,E}f.
 \end{aligned}
$$
The second equality is read in the shifted form coordinates of
\eqref{pre:right-transition}. Hence the transition on cohomology
is postcomposition with $\partial_p[n]$.

We record the compatibility with reversal of bar words. For an
$A$-bimodule $X$, let $X^{\op}$ be the $A^{\op}$-bimodule with
$a^{\op}xb^{\op}:=bxa$ for $a,b\in A$ and $x\in X$.
This exact functor on bimodules acts degreewise on complexes and
commutes with shifts. Write
$P_E(A):=\overline{\operatorname{Bar}}_E(A)$ for this calculation.
For $r\geq0$, reversal defines a bimodule isomorphism
$$
\begin{aligned}
 \mathfrak r_r:(P_E(A)^{-r})^{\op}&\longrightarrow P_E(A^{\op})^{-r},\\
 a_0\otimes s\bar a_1\otimes\cdots\otimes s\bar a_r\otimes a_{r+1}
 &\longmapsto
 (-1)^{r(r+1)/2}a_{r+1}^{\op}\otimes s\overline{a_r^{\op}}
 \otimes\cdots\otimes s\overline{a_1^{\op}}\otimes a_0^{\op},
\end{aligned}
$$
where $a_0,\ldots,a_{r+1}\in A$ and all tensor products are over $E$.
For $r\geq1$ and $0\leq i\leq r$, the $i$th bar face becomes
the $(r-i)$th bar face, and
$$
 \frac{r(r+1)}2+r-i-\frac{(r-1)r}2-i=2(r-i).
$$
Hence $d\mathfrak r_r=\mathfrak r_{r-1}d$ for $r\geq1$; in degree zero
the multiplication maps also commute with $\mathfrak r_0$.
For $p\geq0$, put
$M_{E,p}(A):=\operatorname{Coker}(d:P_E(A)^{-p-1}\to P_E(A)^{-p})$,
regarded as an unshifted bimodule. The chain map $\mathfrak r$
induces an isomorphism
$\overline{\mathfrak r}_p:M_{E,p}(A)^{\op}\to M_{E,p}(A^{\op})$.
There is a commutative diagram of exact sequences of
$A^{\op}$-bimodules
$$
\xymatrix@C=1pc{
0\ar[r]&M_{E,p+1}(A)^{\op}\ar[r]\ar[d]_{\overline{\mathfrak r}_{p+1}}
 &(P_E(A)^{-p})^{\op}\ar[r]\ar[d]_{\mathfrak r_p}
 &M_{E,p}(A)^{\op}\ar[r]\ar[d]_{\overline{\mathfrak r}_p}&0\\
0\ar[r]&M_{E,p+1}(A^{\op})\ar[r]
 &P_E(A^{\op})^{-p}\ar[r]&M_{E,p}(A^{\op})\ar[r]&0.
}
$$
Write $\partial_{E,p}^A,t_{E,p}^A$ for the maps of
Lemma~\ref{src:syzygy-stabilization} for $P_E(A)$.
The cone maps induced by this diagram commute with projection
to the shifted kernel. Applying the same factor $(-1)^p$ to
the two connecting maps gives
\begin{equation}\label{pre:reversal-syzygy}
\begin{aligned}
 \overline{\mathfrak r}_{p+1}[p+1](\partial_{E,p}^A)^{\op}
 &=\partial_{E,p}^{A^{\op}}\overline{\mathfrak r}_p[p],\\
 \overline{\mathfrak r}_p[p](t_{E,p}^A)^{\op}
 &=t_{E,p}^{A^{\op}}.
\end{aligned}
\end{equation}
Here $\overline{\mathfrak r}_0$ identifies the two diagonal
bimodules. On a cochain with $m$ inputs and $p$ form factors,
conjugating by input reversal and coefficient reversal has sign
$$
 (-1)^{m(m+1)/2+p(p+1)/2}
 =(-1)^{\binom{m+1}{2}-\binom{p+1}{2}},
$$
which is the first component in Lemma~\ref{lem:left-right-singular}.
The formulas also apply for $E=k$ with normalized classes in
$A/k1$. Thus reversal sends both the numerator and the signed
syzygy denominator to their opposite bimodule morphisms.
For left forms it therefore gives the same identification with
postcomposition by $\partial_p[n]$.
A closed degree-$n$ singular
cochain at stage $p$ therefore represents
$\mathsf q(t_p[n])^{-1}\mathsf q([f])$ in
$\mathbf D_{\sg}(A^e)$. This is the identification used for
Keller's canonical map.

For a finite-dimensional algebra $A$, denote the resulting graded
algebra isomorphism by
$$
 \mathrm{can}_{A,k}:H^*(\Cbar_{\sg,L}^*(A^{\op}))
       \xrightarrow{\sim}\mathrm{HH}_{\sg}^*(A^{\op},A^{\op}).
$$
If $A$ has a split semisimple subalgebra $E$ as in
Lemma~\ref{lem:left-right-singular}, denote its relative version by
$$
 \mathrm{can}_{A,E}:H^*(\Cbar_{\sg,L,E}^*(A^{\op}))
       \xrightarrow{\sim}\mathrm{HH}_{\sg}^*(A^{\op},A^{\op}).
$$
In either case, a closed cochain $f$ of total degree $n$ at form
stage $p$ is sent to $\mathsf q(t_p[n])^{-1}\mathsf q([f])$.
Here $\mathsf q$ is the quotient functor for $(A^{\op})^e$.

\begin{Def}\label{def:e-relative hcc}
\normalfont
{\rm(\cite[Section~6.2]{ChenLiWang2025})} Let $D$ be a dg algebra, and let
$E=\bigoplus_{i=1}^n ke_i\subseteq D^0$,
where $e_1,\ldots,e_n$ are pairwise orthogonal idempotents
with $\sum_{i=1}^n e_i=1_D$ and $d_D(e_i)=0$ for every $i$.
Put $\overline D:=D/E$, $(s\overline D)^{\tensor_E0}:=E$ and
$T_E(s\overline D):=\bigoplus_{m\geq0}(s\overline D)^{\tensor_E m}$.
The \emph{normalized $E$-relative Hochschild cochain complex} of $D$ is
$$
 \Cbar_E^*(D)
 :=\Hom_{E-E}\bigl(T_E(s\overline D),D\bigr),
$$
equipped with the normalized $E$-relative Hochschild differential.
For $n\in\mathbb Z$ and $m\geq0$, put
$$
 \Cbar_E^{n,m}(D)
 :=\bigl\{f\in\Hom_{E-E}
 \bigl((s\overline D)^{\tensor_E m},D\bigr)\mid
 |f(x)|=|x|+n\text{ for every homogeneous }
 x\in(s\overline D)^{\tensor_E m}\bigr\}.
 $$
We call $m$ the tensor degree.
For $m\geq0$, put
$$
 \Cbar_E^{*,m}(D)
 :=\bigoplus_{n\in\mathbb Z}\Cbar_E^{n,m}(D).
$$
The decomposition of $T_E(s\overline D)$ into its tensor powers induces a natural isomorphism of $k$-vector spaces
$$
 \Cbar_E^n(D)
 \cong\prod_{m\geq0}\Cbar_E^{n,m}(D).
$$
\end{Def}

By \cite[Section~6]{ChenLiWang2025}, $\Cbar_E^*(D)$ is a brace $B_\infty$-algebra.

\begin{Lem}\label{lem:relative-hochschild-inclusion}
\normalfont
{\rm(\cite[Lemma~6.4]{ChenLiWang2025})} Let $D$ and $E$ be as in Definition~\ref{def:e-relative hcc}.
Then the natural injection
$$
 \Cbar_E^*(D)\longrightarrow C^*(D)
$$
is a strict $B_\infty$-quasi-isomorphism.
\end{Lem}

\subsection{Radical quivers and dg Leavitt algebras}

\begin{Def}\label{def:radical-quiver-leavitt}
\normalfont
Let $A=E\oplus J$ be split basic, where $J:=\rad(A)$ and
$E:=ke_1\oplus\cdots\oplus ke_n$, and choose a $k$-basis $\mathcal B_{ji}$ of each $e_jJe_i$.
Following \cite[Section~3]{Schaps1988}, the \emph{radical quiver} $Q$ has $Q_0:=\{1,\ldots,n\}$ and one arrow $i\to j$ for every element of $\mathcal B_{ji}$.
Write $s,t:Q_1\to Q_0$ for the source and target maps.
For each $\alpha\in Q_1$, denote the corresponding basis element of $J$ by $j_\alpha$.
Then $kQ_1\to J$, $\alpha\longmapsto j_\alpha$, is an $E$-$E$-bimodule isomorphism.
For $r\geq2$, let
$$
 Q_r:=\left\{
 \alpha_r\cdots\alpha_1\ \middle|\
 \alpha_1,\ldots,\alpha_r\in Q_1,
 \ t(\alpha_i)=s(\alpha_{i+1})\text{ for }1\leq i<r
 \right\}
$$
be the set of formal paths of length $r$.
Thus $Q_2$ consists of the formal paths $\beta_2\beta_1$ formed by two composable arrows.
The formal path $\beta_2\beta_1\in Q_2$ is not the element $j_{\beta_2}j_{\beta_1}\in J^2$.
The element $j_{\beta_2}j_{\beta_1}$ is the product in the algebra $A$ and will be expressed in the chosen basis of $J$.

Assume that $Q$ has no sinks.
Following \cite[Definition~1.3]{AbramsArandaPino2005} and \cite[Section~4]{ChenLiWang2025}, define its \emph{Leavitt path algebra} $L(Q)$ as follows.
For every arrow $\alpha:i\to j$ of $Q$, adjoin an arrow $\alpha^*:j\to i$, called its \emph{ghost arrow}, and denote the resulting quiver by $\widehat Q$.
The arrows of $Q$ are called \emph{real arrows}.
We read path multiplication from right to left.
Thus $e_{t(\alpha)}\alpha=\alpha=\alpha e_{s(\alpha)}$ and $e_{s(\alpha)}\alpha^*=\alpha^*=\alpha^*e_{t(\alpha)}$.
The algebra $L(Q)$ is the quotient of the path algebra $k\widehat Q$ by the relations
$$
 \alpha\beta^*=\delta_{\alpha,\beta}e_{t(\alpha)}
 \quad\bigl(\alpha,\beta\in Q_1, s(\alpha)=s(\beta)\bigr),
 \qquad
 e_i=\sum_{\substack{\alpha\in Q_1\\s(\alpha)=i}}
 \alpha^*\alpha
 \quad(i\in Q_0).
$$
Here $e_i$ is the trivial path at $i$, and $1=\sum_{i\in Q_0}e_i$.
The assumption that $Q$ has no sinks ensures that the second sum is nonempty at every vertex.
Finally, $L(Q)$ is graded by $|e_i|:=0$, $|\alpha|:=1$ and $|\alpha^*|:=-1$ for $i\in Q_0$ and $\alpha\in Q_1$.
\end{Def}

Let $Q$ be a finite quiver without sinks.
Put $E:=kQ_0$ and $\overline{L(Q)}:=L(Q)/E$.
Regard $L(Q)$ as a dg algebra with zero differential, and let $d_0$ be the Hochschild differential on $\Cbar_E^*(L(Q))$ determined by the graded-algebra multiplication of $L(Q)$.
Let $\Chat^*(L(Q))$ be the $B_\infty$-algebra defined in \cite[Section~11]{ChenLiWang2025}, whose underlying graded vector space is
$$
 \Chat^*(L(Q))
 =\left(\bigoplus_{i\in Q_0}e_iL(Q)e_i\right)
   \oplus
   \left(\bigoplus_{i\in Q_0}s^{-1}e_iL(Q)e_i\right).
$$

\begin{Lem}\label{lem:clw-morphisms}
Put $A_Q:=E\oplus kQ_1$ with $(kQ_1)^2=0$.
\begin{enumerate}
\item {\rm(\cite[Theorem~10.4, Proposition~11.4 and Theorem~13.1]{ChenLiWang2025})} There are $B_\infty$-morphisms
$$
 \Cbar_{\sg,R,E}^*(A_Q)
 \xleftarrow{\ \kappa\ }
 \Cbar_{\sg,R}^*(Q)
 \xrightarrow{\ \rho\ }
 \Chat^*(L(Q))
 \xrightarrow{\ (\Phi_1,\Phi_2,\ldots)\ }
 \Cbar_E^*(L(Q))^{\opp}.
$$
The maps $\kappa$ and $\rho$ are strict $B_\infty$-isomorphisms, and $(\Phi_1,\Phi_2,\ldots)$ is a $B_\infty$-quasi-isomorphism.
Here $\Cbar_{\sg,R}^*(Q)$ is the right singular Hochschild cochain
$B_\infty$-algebra of \cite[Chapter~10]{ChenLiWang2025}.
At form stage $p$, its homogeneous generators are pairs
$(\gamma,\beta)$ with $\gamma\in Q_m$, $\beta\in Q_p$, and
$s(\gamma)=s(\beta)$, $t(\gamma)=t(\beta)$, together with
$s^{-1}(\gamma,\beta)$ for $\gamma\in Q_m$, $\beta\in Q_{p+1}$
with the same endpoints. Both kinds of generators have degree
$m-p$. The transition sends $(\gamma,\beta)$ to
$\sum_{s(\alpha)=t(\gamma)}(\alpha\gamma,\alpha\beta)$ and
commutes with $s^{-1}$.

\item {\rm(\cite[Remarks~12.1 and~12.3]{ChenLiWang2025})} For $\alpha\in Q_1$ and $p\in Q_2$, write $\alpha\parallel p$ if $s(\alpha)=s(p)$ and $t(\alpha)=t(p)$.
For $p=\beta_2\beta_1\in Q_2$ and $\alpha\in Q_1$ with $\alpha\parallel p$, let
$$
 f_{\alpha,p}:(skQ_1)^{\tensor_E2}\longrightarrow kQ_1
$$
be the $E$-$E$-linear map which sends $s\beta_2\tensor_Es\beta_1$ to $\alpha$ and vanishes on every other tensor of arrows.
Under the maps in {\rm(1)},
$$
 \rho\kappa^{-1}(f_{\alpha,p})=-s^{-1}\alpha^*p.
$$
Moreover, for every $\gamma\in Q_1$,
$$
 \Phi_1(-s^{-1}\alpha^*p)(s\gamma)
 =\gamma\alpha^*p
 =\delta_{\gamma,\alpha}p.
$$
\end{enumerate}
\end{Lem}

\begin{Lem}\label{lem:clw-homotopy-data}
\begin{enumerate}
\item {\rm(\cite[(12.2)--(12.5) and Lemma~12.4]{ChenLiWang2025})} There are chain maps $\Phi_1,\Psi_1$ and a degree-$-1$ map $H$
$$
 \xymatrix{
 \Chat^*(L(Q)) \ar@<0.55ex>[r]^-{\Phi_1}&
 \Cbar_E^*(L(Q))\ar@<0.55ex>[l]^-{\Psi_1}
 }
 \qquad
 H:\Cbar_E^*(L(Q))\longrightarrow\Cbar_E^*(L(Q))[-1]
$$
such that
$$
 \Psi_1\Phi_1={\rm id}_{\Chat^*(L(Q))},\qquad
 {\rm id}_{\Cbar_E^*(L(Q))}-\Phi_1\Psi_1=d_0H+Hd_0.
$$
Moreover,
$$
 \operatorname{im}(\Phi_1)\subseteq
 \Cbar_E^{*,0}(L(Q))\oplus\Cbar_E^{*,1}(L(Q)),
 \qquad
 \Psi_1|_{\prod_{r\geq2}\Cbar_E^{n,r}(L(Q))}=0
 \quad(n\in\mathbb Z).
$$
For $f\in\Cbar_E^{*,1}(L(Q))$, the map $\Psi_1$ is given by
$$
 \Psi_1(f)
 =-\sum_{\alpha\in Q_1}s^{-1}\alpha^*f(s\overline\alpha).
$$
\item {\rm(\cite[Section~4, Subsection~7.2 and (12.4)--(12.5)]{ChenLiWang2025})} There is a degree $-1$ graded $E$-derivation
$$
 D:L(Q)\longrightarrow
 \bigoplus_{i\in Q_0}L(Q)e_i\tensor sk\tensor e_iL(Q)
$$
determined by
$$
 D(e_i)=0,\qquad
 D(\alpha)=-\alpha\tensor s1_k\tensor e_{s(\alpha)},\qquad
 D(\alpha^*)=-e_{s(\alpha)}\tensor s1_k\tensor\alpha^*,
$$
and the graded Leibniz identity
$D(xy)=D(x)y+(-1)^{|x|}xD(y)$.
Define the map
$$
 \iota:
 \bigoplus_{i\in Q_0}L(Q)e_i\tensor sk\tensor e_iL(Q)
 \longrightarrow L(Q)\tensor_Es\overline {L(Q)}\tensor_EL(Q)
$$
for $x\in L(Q)e_i$ and $y\in e_iL(Q)$ by
$$
 \iota(x\tensor s1_k\tensor y)
 =-\sum_{\substack{\gamma\in Q_1\\s(\gamma)=i}}
 x\gamma^*\tensor_Es\overline\gamma\tensor_Ey.
$$
Let
$$
 \pi:L(Q)\tensor_Es\overline {L(Q)}\tensor_EL(Q)
 \longrightarrow\bigoplus_{i\in Q_0}L(Q)e_i\tensor sk\tensor e_iL(Q),
 \qquad
 \pi(a\tensor_Es\overline z\tensor_Eb):=aD(z)b.
$$
Let $q:L(Q)\to\overline {L(Q)}=L(Q)/E$ be the quotient map.
Define the degree-$-1$ map
$$
 \overline{\iota\pi}
 :=\bigl((s\circ q)\tensor{\rm id}\tensor{\rm id}\bigr)
   \circ\iota\circ\pi:
 L(Q)\tensor_Es\overline {L(Q)}\tensor_EL(Q)
 \longrightarrow s\overline {L(Q)}\tensor_Es\overline {L(Q)}\tensor_EL(Q).
$$
Thus, if
$$
 \iota\pi(a\tensor_Es\overline z\tensor_Eb)
 =\sum_c a_c\tensor_Es\overline z_c\tensor_Eb_c,
$$
then
$$
 \overline{\iota\pi}(a\tensor_Es\overline z\tensor_Eb)
 =\sum_c s\overline{a_c}\tensor_Es\overline z_c\tensor_Eb_c.
$$
For $n\in\mathbb Z$ and $r\geq2$, the restriction of the homotopy $H$ in {\rm(1)} is a map
$$
 H:\Cbar_E^{n,r}(L(Q))\longrightarrow
   \Cbar_E^{n-1,r-1}(L(Q)).
$$
For a homogeneous $f\in\Cbar_E^{n,r}(L(Q))$ and homogeneous $a_1,\ldots,a_{r-1}\in L(Q)$, it is given by
$$
 \begin{aligned}
 &H(f)(s\overline a_1\tensor_E\cdots\tensor_Es\overline a_{r-1})\\
 &\quad=(-1)^{1+|f|+\sum_{j=1}^{r-2}(|a_j|-1)}
 f\bigl(s\overline a_1\tensor_E\cdots\tensor_Es\overline a_{r-2}
 \tensor_E\overline{\iota\pi}
 (1\tensor_Es\overline a_{r-1}\tensor_E1)\bigr).
 \end{aligned}
$$
In this formula, if $\overline{\iota\pi}(1\tensor_Es\overline a_{r-1}\tensor_E1) =\sum_c u_c\tensor_Ev_c\tensor_Ey_c$, then
$$
 f\bigl(s\overline a_1\tensor_E\cdots\tensor_Es\overline a_{r-2}
 \tensor_E\overline{\iota\pi}
 (1\tensor_Es\overline a_{r-1}\tensor_E1)\bigr)
 :=\sum_c f(s\overline a_1\tensor_E\cdots\tensor_Es\overline a_{r-2}
 \tensor_Eu_c\tensor_Ev_c)y_c.
$$
On the remaining tensor degrees, put
$$
 H|_{\Cbar_E^{n,0}(L(Q))\oplus\Cbar_E^{n,1}(L(Q))}=0.
$$
Thus, under the isomorphism
$$
 \Cbar_E^n(L(Q))
 \cong\prod_{r\geq0}\Cbar_E^{n,r}(L(Q)),
$$
the map $H:\Cbar_E^n(L(Q))\to\Cbar_E^{n-1}(L(Q))$ is given by
$$
 H(f_0,f_1,f_2,f_3,\ldots)
 =(0,H(f_2),H(f_3),H(f_4),\ldots).
$$

For $\alpha\in Q_1$, the formulas for $D$ and $\iota$ give
$\pi(1\tensor_Es\overline\alpha\tensor_E1)=D(\alpha)$
and
$$
\begin{aligned}
 \iota D(\alpha)
 &=\sum_{\substack{\gamma\in Q_1\\s(\gamma)=s(\alpha)}}
   \alpha\gamma^*\tensor_Es\overline\gamma
                   \tensor_Ee_{s(\alpha)}\\
 &=e_{t(\alpha)}\tensor_Es\overline\alpha
                   \tensor_Ee_{s(\alpha)}.
\end{aligned}
$$
Since $\overline{e_{t(\alpha)}}=0$ in $\overline {L(Q)}=L(Q)/E$, it follows that
$$
 \overline{\iota\pi}
 (1\tensor_Es\overline\alpha\tensor_E1)=0.
$$
Consequently, if $r\geq2$, $f\in\Cbar_E^{n,r}(L(Q))$, $x_1,\ldots,x_{r-2}\in s\overline {L(Q)}$ are homogeneous and $\alpha\in Q_1$, then
$$
 H(f)(x_1\tensor_E\cdots\tensor_E x_{r-2}
             \tensor_Es\overline\alpha)=0.
$$
When $r=2$, the initial tensor is omitted.
\end{enumerate}
\end{Lem}

\medskip
Throughout the paper, $L(Q)$ carries the grading fixed in Definition~\ref{def:radical-quiver-leavitt}, namely $|\alpha|=1$ and $|\alpha^*|=-1$.
To state the results of Chen--Wang, we write $L_{\mathrm{CW}}(Q)$ for the same underlying Leavitt path algebra equipped with their grading
$$
 |\alpha|_{\mathrm{CW}}=-1,
 \qquad
 |\alpha^*|_{\mathrm{CW}}=1
 \qquad(\alpha\in Q_1).
$$

Let $A=E\oplus J$ be a split basic finite-dimensional $k$-algebra whose radical quiver has no sinks, where $E:=\bigoplus_i ke_i$.
Choose a $k$-basis of each $e_jJe_i$, write their union as
$\{j_\alpha\}$, and let $Q$ be the corresponding radical quiver.
The $E$-$E$-bimodule isomorphism
$$
 kQ_1\longrightarrow J,
 \qquad \alpha\longmapsto j_\alpha,
$$
defines an associative $E$-$E$-bilinear map
$\nu:kQ_1\tensor_EkQ_1\longrightarrow kQ_1$.
Namely, the image of $\nu(\beta_2\tensor_E\beta_1)$ under the isomorphism $kQ_1\to J$, $\alpha\mapsto j_\alpha$, is $j_{\beta_2}j_{\beta_1}$.
For $\alpha\in Q_1$ and $p\in Q_2$, write $\alpha\parallel p$ if $s(\alpha)=s(p)$ and $t(\alpha)=t(p)$.
For every $p=\beta_2\beta_1\in Q_2$, write
$$
 j_{\beta_2}j_{\beta_1}
 =\sum_{\substack{\alpha\in Q_1\\\alpha\parallel p}}
   \lambda_{p,\alpha}j_\alpha.
$$

\begin{Lem}\label{lem:chen-wang}
\begin{enumerate}
\item {\rm(\cite[(4.2)--(4.4)]{ChenWang2024})} The differential $\partial_\nu$ on $L_{\mathrm{CW}}(Q)$ is the degree-one derivation determined by
$\partial_\nu(e_i)=0 \qquad(i\in Q_0)$
and
$$
\begin{aligned}
 \partial_\nu(\alpha^*)
 &=\sum_{\substack{\beta_2\beta_1\in Q_2\\
                   \alpha\parallel\beta_2\beta_1}}
      \lambda_{\beta_2\beta_1,\alpha}\beta_1^*\beta_2^*,\\
 \partial_\nu(\alpha)
 &=\sum_{\substack{\beta,\gamma\in Q_1\\
                   \beta\alpha\in Q_2\\
                   \gamma\parallel\beta\alpha}}
      \lambda_{\beta\alpha,\gamma}\beta^*\gamma
      \qquad(\alpha\in Q_1).
\end{aligned}
$$

\item {\rm(\cite[Section~2, Proposition~4.1(2), Proposition~10.2 and Theorem~10.5]{ChenWang2024})} There is an isomorphism in the homotopy category of dg categories
$$
 \Sdg(A)\simeq
 \perdg(L_{\mathrm{CW}}(Q),\partial_\nu).
$$
More precisely, there are two quasi-equivalences with a common source that induce this isomorphism.
\end{enumerate}
\end{Lem}

\section{Proof of Theorem~\ref{thm:main}}\label{sec:proof}
In Subsection~\ref{proof:isomorphism},
we prove the weak form of Keller's conjecture for split
finite-dimensional algebras. 
In Subsection~\ref{proof:canonical}, we construct $F_A$ and prove
that it induces Keller's canonical map $\eta_A$ on cohomology.

\subsection{The weak form of Keller's conjecture}\label{proof:isomorphism}
\begin{Def}\label{def:twist}
\normalfont
{\rm(\cite[Definition~1.3]{CalaqueRossi2011})}
Let $V$ be a $B_\infty$-algebra represented by $(T^c(sV),\Delta,D,M)$.
\begin{enumerate}
\item  An element $a\in V$ is called a \emph{Maurer--Cartan element} if $sa$ has degree one and
$$
D(sa)+M(sa,sa)=0.
$$

\item Let $a$ be a Maurer--Cartan element of $V$.
For homogeneous $x\in T^c(sV)$, define the map
$$
\operatorname{ad}_{sa}:T^c(sV)\longrightarrow T^c(sV),
\qquad
x\longmapsto M(sa,x)-(-1)^{|x|}M(x,sa).
$$
The \emph{twist of $V$ by $a$} is the $B_\infty$-algebra represented by
$$
\bigl(T^c(sV),\Delta,D+\operatorname{ad}_{sa},M\bigr),
$$
and is denoted by $\operatorname{Tw}_a(V)$.
Let $\pi:T^c(sV)\to sV$ be the projection.
Since $d_V=m_1=s^{-1}\pi Ds$, the differential of the underlying complex of $\operatorname{Tw}_a(V)$ is
$$
 m_1^a=m_1+\epsilon_a,
 \qquad
 \epsilon_a:=s^{-1}\pi\operatorname{ad}_{sa}s.
$$
The map $\epsilon_a$ has degree one.
For a homogeneous element $x\in V$,
$$
 \epsilon_a(x)
 =s^{-1}\pi\bigl(
 M(sa,sx)-(-1)^{|x|-1}M(sx,sa)
 \bigr).
$$
\end{enumerate}
\end{Def}

\begin{Lem}\label{lem:primitive-twist}
Let $V$ and $W$ be $B_\infty$-algebras, and let $a\in V$ be a Maurer--Cartan element.
\begin{enumerate}
\item {\rm(Compare \cite[Definition~1.3]{CalaqueRossi2011})} $\operatorname{Tw}_a(V)$ is a $B_\infty$-algebra.

\item Let $f:V\to W$ be a $B_\infty$-morphism represented by $F:T^c(sV)\to T^c(sW)$.
Then $f_1(a)$ is a Maurer--Cartan element of $W$, and the same map $F$ represents a $B_\infty$-morphism $\operatorname{Tw}_a(V)\to \operatorname{Tw}_{f_1(a)}(W)$.
\end{enumerate}
\end{Lem}

{\it Proof.}
Write the dg bialgebras representing $V$ and $W$ as
$$
 (T^c(sV),\Delta_V,D_V,M_V)
 \quad\text{and}\quad
 (T^c(sW),\Delta_W,D_W,M_W),
$$
respectively.

{\rm(1)} Calaque--Rossi formulate the twisting construction without suspension and over a field of characteristic zero.
We verify the assertion for $T^c(sV)$ in the conventions used here.
The argument uses the dg-bialgebra identities.

Put $D_a:=D_V+\operatorname{ad}_{sa}$.
The tensor coalgebra, the coproduct $\Delta_V$ and the multiplication $M_V$ are unchanged in $\operatorname{Tw}_a(V)$.
By Definition~\ref{def:B-infinity}{\rm(2)}, it is therefore enough to prove that
$$
 \deg(D_a)=1,
 \qquad D_a(1)=0,
 \qquad
 D_aM_V=M_V(D_a\tensor{\rm id}+{\rm id}\tensor D_a),
$$
$$
 \Delta_VD_a=(D_a\tensor{\rm id}+{\rm id}\tensor D_a)\Delta_V,
 \qquad D_a^2=0.
$$
We verify the degree, unit, derivation, coderivation and square-zero identities in the order listed above.

Since $sa$ has degree one, $\operatorname{ad}_{sa}$ and $D_a$ have degree one.
Moreover,
$$
 \operatorname{ad}_{sa}(1)
 =M_V(sa,1)-M_V(1,sa)=0.
$$
Since $D_V(1)=0$, the equality $\operatorname{ad}_{sa}(1)=0$ gives $D_a(1)=0$.

We next verify the derivation identity $D_aM_V=M_V(D_a\tensor{\rm id}+{\rm id}\tensor D_a)$.
Since $D_V$ is a derivation of $M_V$ and $D_a=D_V+\operatorname{ad}_{sa}$, it is enough to prove that $\operatorname{ad}_{sa}M_V=M_V(\operatorname{ad}_{sa}\tensor{\rm id}+{\rm id}\tensor\operatorname{ad}_{sa})$.
Let $x,y\in T^c(sV)$ be homogeneous.
By the definition of $\operatorname{ad}_{sa}$ and the associativity of $M_V$, we have
$$
\begin{aligned}
(\operatorname{ad}_{sa}M_V)(x\tensor y)
={}&\operatorname{ad}_{sa}(M_V(x,y))\\
={}&M_V(sa,M_V(x,y))
   -(-1)^{|x|+|y|}M_V(M_V(x,y),sa)\\
={}&M_V(M_V(sa,x),y)
   -(-1)^{|x|+|y|}M_V(x,M_V(y,sa)).
\end{aligned}
$$
On the other hand, the Koszul convention gives
$$
\begin{aligned}
&\bigl[
M_V(\operatorname{ad}_{sa}\tensor{\rm id}
    +{\rm id}\tensor\operatorname{ad}_{sa})
\bigr](x\tensor y)\\
={}&M_V(\operatorname{ad}_{sa}(x),y)
   +(-1)^{|x|}M_V(x,\operatorname{ad}_{sa}(y))\\
={}&M_V(M_V(sa,x),y)
   -(-1)^{|x|}M_V(M_V(x,sa),y)\\
&\quad
   +(-1)^{|x|}M_V(x,M_V(sa,y))
   -(-1)^{|x|+|y|}M_V(x,M_V(y,sa))\\
={}&M_V(M_V(sa,x),y)
   -(-1)^{|x|+|y|}M_V(x,M_V(y,sa)),
\end{aligned}
$$
where the middle two terms cancel by the associativity of $M_V$.
Therefore $\operatorname{ad}_{sa}M_V = M_V(\operatorname{ad}_{sa}\tensor{\rm id} +{\rm id}\tensor\operatorname{ad}_{sa}),$ and hence $D_a$ is a derivation of $M_V$.

We now verify the coderivation identity $\Delta_VD_a =(D_a\tensor{\rm id}+{\rm id}\tensor D_a)\Delta_V$.
Since $D_V$ is a coderivation of $\Delta_V$ and $D_a=D_V+\operatorname{ad}_{sa}$, it is enough to prove that $ \Delta_V\operatorname{ad}_{sa} = (\operatorname{ad}_{sa}\tensor{\rm id} +{\rm id}\tensor\operatorname{ad}_{sa})\Delta_V$.
Since $sa\in(sV)^{\tensor1}$, the definition of the deconcatenation coproduct gives $\Delta_V(sa)=1\tensor sa+sa\tensor1$.
Thus $sa$ is primitive.
Let $x\in T^c(sV)$ be homogeneous, and write $\Delta_V(x)=\sum x_{(1)}\tensor x_{(2)}$, with homogeneous summands.
Since $\Delta_V$ has degree zero, every summand satisfies $|x|=|x_{(1)}|+|x_{(2)}|$.
Using the fact that $M_V$ is a morphism of coalgebras, we obtain
$$
\begin{aligned}
(\Delta_V\operatorname{ad}_{sa})(x)
={}&\Delta_V\bigl(
M_V(sa,x)-(-1)^{|x|}M_V(x,sa)
\bigr)\\
={}&(M_V\tensor M_V)
({\rm id}\tensor\mathfrak t\tensor{\rm id})
\bigl(\Delta_V(sa)\tensor\Delta_V(x)\bigr)
-(-1)^{|x|}
(M_V\tensor M_V)
({\rm id}\tensor\mathfrak t\tensor{\rm id})
\bigl(\Delta_V(x)\tensor\Delta_V(sa)\bigr)\\
={}&\sum\bigl(
M_V(sa,x_{(1)})\tensor x_{(2)}
+(-1)^{|x_{(1)}|}
 x_{(1)}\tensor M_V(sa,x_{(2)})\\
&-(-1)^{|x|}
 x_{(1)}\tensor M_V(x_{(2)},sa)
-(-1)^{|x|+|x_{(2)}|}
 M_V(x_{(1)},sa)\tensor x_{(2)}
\bigr)\\
={}&\sum\bigl(
\operatorname{ad}_{sa}(x_{(1)})\tensor x_{(2)}
+(-1)^{|x_{(1)}|}
 x_{(1)}\tensor\operatorname{ad}_{sa}(x_{(2)})
\bigr)\\
={}&
\bigl[
(\operatorname{ad}_{sa}\tensor{\rm id}
+{\rm id}\tensor\operatorname{ad}_{sa})\Delta_V
\bigr](x).
\end{aligned}
$$
Therefore $\Delta_V\operatorname{ad}_{sa} =(\operatorname{ad}_{sa}\tensor{\rm id} +{\rm id}\tensor\operatorname{ad}_{sa})\Delta_V$, and hence $D_a$ is a coderivation of $\Delta_V$.

It remains to prove that $D_a^2=0$.
Let $x\in T^c(sV)$ be homogeneous.
Since $D_a=D_V+\operatorname{ad}_{sa}$ and $D_V^2=0$, we have
$D_a^2(x) = (D_V\operatorname{ad}_{sa} +\operatorname{ad}_{sa}D_V)(x) +\operatorname{ad}_{sa}^{\,2}(x)$.
Using the definition of $\operatorname{ad}_{sa}$ and the derivation rule for $D_V$, we obtain
$$
\begin{aligned}
&(D_V\operatorname{ad}_{sa}
 +\operatorname{ad}_{sa}D_V)(x)\\
={}&D_V\bigl(
M_V(sa,x)-(-1)^{|x|}M_V(x,sa)
\bigr)
 +M_V(sa,D_V(x))
 -(-1)^{|x|+1}M_V(D_V(x),sa)\\
={}&M_V(D_V(sa),x)-M_V(sa,D_V(x))
 -(-1)^{|x|}M_V(D_V(x),sa)-M_V(x,D_V(sa))\\
&\quad
 +M_V(sa,D_V(x))
 +(-1)^{|x|}M_V(D_V(x),sa)\\
={}&M_V(D_V(sa),x)-M_V(x,D_V(sa)).
\end{aligned}
$$
A direct expansion using the associativity of $M_V$ also gives
$$
\begin{aligned}
\operatorname{ad}_{sa}^{\,2}(x)
={}&M_V(sa,M_V(sa,x))
 -(-1)^{|x|}M_V(sa,M_V(x,sa))\\
&\quad
 +(-1)^{|x|}M_V(M_V(sa,x),sa)
 -M_V(M_V(x,sa),sa)\\
={}&M_V(M_V(sa,sa),x)-M_V(x,M_V(sa,sa)),
\end{aligned}
$$
where the middle two terms cancel by associativity.
Consequently,
$$
\begin{aligned}
D_a^2(x)
={}&M_V(D_V(sa),x)-M_V(x,D_V(sa))+M_V(M_V(sa,sa),x)-M_V(x,M_V(sa,sa))\\
={}&M_V\bigl(D_V(sa)+M_V(sa,sa),x\bigr)-M_V\bigl(x,D_V(sa)+M_V(sa,sa)\bigr)\\
={}&M_V(0,x)-M_V(x,0)\\
={}&0,
\end{aligned}
$$
where the penultimate equality follows from the Maurer--Cartan equation.
Thus $D_a^2=0$.
Therefore $D_a$ satisfies the three conditions imposed on the differential in Definition~\ref{def:B-infinity}.
Therefore $\operatorname{Tw}_a(V)$ is a $B_\infty$-algebra.

{\rm(2)} The coproduct, multiplication and unit are unchanged by the two twists.
Since $F$ already preserves the coproduct, multiplication and unit, we only have to prove that $f_1(a)$ is a Maurer--Cartan element and that
$$
 F(D_V+\operatorname{ad}_{sa})
 =(D_W+\operatorname{ad}_{sf_1(a)})F.
$$

We first prove that $f_1(a)$ is a Maurer--Cartan element of $W$.
By Definition~\ref{def:twist}, it is enough to verify that $sf_1(a)$ has degree one and that $D_W(sf_1(a))+M_W(sf_1(a),sf_1(a))=0$.
Since $F$ is a coalgebra map and $sa$ is primitive, $F(sa)$ is primitive.
The primitive elements of the tensor coalgebra $T^c(sW)$ are precisely the elements of $sW$.
Hence $F(sa)=\pi_WF(sa)$.
Taking $r=1$ in the definition of the Taylor components gives $\pi_WF(sa)=sf_1(a)$, and therefore $F(sa)=sf_1(a)$.
Since $F$ has degree zero and $sa$ has degree one, it follows that $sf_1(a)$ has degree one.
Since $F$ commutes with the differentials and multiplications, we have
$$
\begin{aligned}
&D_W(sf_1(a))+M_W(sf_1(a),sf_1(a))\\
={}&D_WF(sa)+M_W(F(sa),F(sa))\\
={}&F(D_V(sa))+F(M_V(sa,sa))\\
={}&F\bigl(D_V(sa)+M_V(sa,sa)\bigr)\\
={}&0,
\end{aligned}
$$
where the last equality follows from the Maurer--Cartan equation for $a$.
Therefore $f_1(a)$ is a Maurer--Cartan element of $W$.

It remains to prove that $F(D_V+\operatorname{ad}_{sa})=(D_W+\operatorname{ad}_{sf_1(a)})F$.
Since $FD_V=D_WF$, it is enough to prove that $F\operatorname{ad}_{sa}=\operatorname{ad}_{sf_1(a)}F$.
Let $x\in T^c(sV)$ be homogeneous.
Since $F$ preserves the multiplications and $F(sa)=sf_1(a)$, we have
$$
\begin{aligned}
(F\operatorname{ad}_{sa})(x)
={}&F\bigl(
M_V(sa,x)-(-1)^{|x|}M_V(x,sa)
\bigr)\\
={}&M_W(F(sa),F(x))
   -(-1)^{|x|}M_W(F(x),F(sa))\\
={}&M_W(sf_1(a),F(x))
   -(-1)^{|F(x)|}M_W(F(x),sf_1(a))\\
={}&\operatorname{ad}_{sf_1(a)}(F(x))\\
={}&(\operatorname{ad}_{sf_1(a)}F)(x),
\end{aligned}
$$
where $|F(x)|=|x|$ because $F$ has degree zero.
Therefore $F\operatorname{ad}_{sa}=\operatorname{ad}_{sf_1(a)}F$.
Combining $F\operatorname{ad}_{sa}=\operatorname{ad}_{sf_1(a)}F$ with $FD_V=D_WF$, we obtain
$$
\begin{aligned}
 F(D_V+\operatorname{ad}_{sa})
 &=FD_V+F\operatorname{ad}_{sa}\\
 &=D_WF+\operatorname{ad}_{sf_1(a)}F\\
 &=(D_W+\operatorname{ad}_{sf_1(a)})F.
\end{aligned}
$$
Thus $F$ preserves the twisted differentials.
Since it already preserves the coproduct, multiplication and unit, it is a dg-bialgebra map between the two twists and represents the asserted $B_\infty$-morphism.
Its underlying coalgebra map is unchanged, so its Taylor components remain $(f_1,f_2,\ldots)$.
This proves~{\rm(2)}. $\square$

\medskip
Let $C$ be a split basic finite-dimensional $k$-algebra and put $J_C:=\rad(C)$.
Since $C/J_C\simeq k^n$ as $k$-algebras, we may choose pairwise orthogonal primitive idempotents $e_1,\ldots,e_n$ whose images form the standard primitive idempotents of $C/J_C$.
Set $E_C:=ke_1\oplus\cdots\oplus ke_n$.
Then $E_C\simeq C/J_C$ as $k$-algebras, and $C=E_C\oplus J_C$ as $k$-vector spaces.
Define the \emph{radical-square-zero algebra associated with $C$} on the vector space $E_C\oplus J_C$ by
$$
 (e+x)\cdot_{\widetilde C}(e'+y)
 :=ee'+ey+xe'
 \qquad(e,e'\in E_C,\ x,y\in J_C),
$$
where the three products on the right are taken in $C$.
Thus the multiplication of $\widetilde C$ agrees with that of $C$ on $E_C\tensor E_C$, $E_C\tensor J_C$ and $J_C\tensor E_C$, while $x\cdot_{\widetilde C}y=0$ for $x,y\in J_C$.
Let $m_C:C\tensor_k C\to C$ denote the multiplication of $C$, and let $m_{\widetilde C}:\widetilde C\tensor_k \widetilde C\to \widetilde C$ denote the multiplication of $\widetilde C$.
Let $\pi_J:E_C\oplus J_C\to J_C$ be the projection and define
$$
 \mu_0:C\tensor_k C\longrightarrow C,
 \qquad \mu_0(x,y):=\pi_J(x)\pi_J(y),
$$
where the product on the right is taken in $C$.
Thus $\mu_0(x,y)=xy$ for $x,y\in J_C$, while $\mu_0(x,y)=0$ if $x\in E_C$ or $y\in E_C$.
Via their common vector-space decomposition $E_C\oplus J_C$, we identify $C$ and $\widetilde C$.
Under the identification $C=E_C\oplus J_C=\widetilde C$, one has $m_C=m_{\widetilde C}+\mu_0$.

The map $\mu_0$ defines an $E_C$-$E_C$-linear cochain
$$
 \mu:(sJ_C)^{\tensor_{E_C}2}\longrightarrow C,
 \qquad \mu(sx\tensor sy):=\mu_0(x,y).
$$
Since $J_C$ and $C$ are concentrated in degree zero and the suspension map has degree $-1$, the map $\mu$ has degree two.
Under the vector-space identification $C=\widetilde C=E_C\oplus J_C$, we have
$$
 \mu\in\Cbar_{E_C}^2(\widetilde C)
 =\Hom_{E_C-E_C}\bigl((sJ_C)^{\tensor_{E_C}2},C\bigr).
$$
Through the canonical map from the term $p=0$ to the colimit, we regard $\mu$ as an element of $\Cbar_{\sg,R,E_C}^2(\widetilde C)$.

\begin{Lem}\label{lem:singular-braces}
After forgetting the differentials, $\Cbar_{\sg,R,E_C}^*(C)$ and $\Cbar_{\sg,R,E_C}^*(\widetilde C)$ have the same underlying graded vector space
$$
 V:=\varinjlim_{p\geq0}V_p^*,
 \qquad
 \theta_p:V_p^*\longrightarrow V_{p+1}^*,
 \qquad
 \theta_p(f):={\rm id}_{sJ_C}\tensor_{E_C}f,
$$
where
$$
 V_p^m:=
 \Hom_{E_C-E_C}\bigl(
 (sJ_C)^{\tensor_{E_C}(m+p)},
 (sJ_C)^{\tensor_{E_C}p}
 \tensor_{E_C}(E_C\oplus J_C)
 \bigr)
$$
when $m+p\geq0$, and $V_p^m:=0$ otherwise.
On the common colimit $V$, their brace operations agree:
$$
 f\{g_1,\ldots,g_q\}_C
 =
 f\{g_1,\ldots,g_q\}_{\widetilde C}
$$
for all homogeneous $f,g_1,\ldots,g_q\in V$ and $q\geq1$.
\end{Lem}
{\it Proof.}
The multiplications of $C$ and $\widetilde C$ have the same restrictions to $E_C\tensor E_C$, $E_C\tensor J_C$ and $J_C\tensor E_C$.
Hence the left and right actions of $E_C$ on the common vector space $E_C\oplus J_C$ are the same.
Consequently, the spaces $V_p^m$ are the same for the two algebras.
Their transition maps are also the same, because both are given by $ \theta_p(f)={\rm id}_{sJ_C}\tensor_{E_C}f$.
Thus the two colimits have the same underlying graded vector space $V$.

By the explicit brace formula of \cite[Definition~8.4 and Lemma~8.5]{ChenLiWang2025}, every summand of $f\{g_1,\ldots,g_q\}$ is obtained by tensoring and composing $f,g_1,\ldots,g_q$, copies of ${\rm id}_{sJ_C}$, and the projection $ E_C\oplus J_C\to sJ_C$, $e+x\longmapsto sx$.
The multiplication map $J_C\tensor_{E_C}J_C\to J_C$ does not occur in the brace compositions.
Hence replacing the multiplication of $C$ on $J_C\tensor_{E_C}J_C$ by zero does not change any brace summand or its Koszul sign.
Choose representatives $f\in V_{p_0}^{m_0}$ and
$g_i\in V_{p_i}^{m_i}$ for $1\leq i\leq q$. The finitely many
brace summands may be moved to a common output stage $r$.
For each insertion, the two representatives agree after transition
to stage $r$, since their defining compositions and tensor signs
agree. Summing over the insertions proves the equality in $V$.
Independence of the chosen representatives follows from
\cite[Remark~8.6]{ChenLiWang2025}.
\hfill$\square$

\begin{Prop}\label{prop:source-twist}
Under the identification of the underlying graded vector spaces in Lemma~\ref{lem:singular-braces}, one has
$$
 \Cbar_{\sg,R,E_C}^*(C)
 =\operatorname{Tw}_\mu\bigl(
   \Cbar_{\sg,R,E_C}^*(\widetilde C)\bigr)
$$
as $B_\infty$-algebras.
\end{Prop}

{\it Proof.}
By Lemma~\ref{lem:singular-braces}, we identify the underlying graded vector spaces by writing
$$
 V
 :=\Cbar_{\sg,R,E_C}^*(C)
 =\Cbar_{\sg,R,E_C}^*(\widetilde C).
$$
Under the identification $V=\Cbar_{\sg,R,E_C}^*(C)=\Cbar_{\sg,R,E_C}^*(\widetilde C)$ of the underlying graded vector spaces, the two $B_\infty$-algebras have the same brace operations.
Hence they have the same tensor coalgebra $T^c(sV)$ and the same deconcatenation coproduct $\Delta$.
Since they are brace $B_\infty$-algebras, their common brace operations and unit determine the same map $\pi M$, and hence the same coalgebra multiplication $M$; see the discussion following Definition~\ref{def:brace-B-infinity}.

Write $D_C$ and $D_{\widetilde C}$ for the two coderivations.
By Definition~\ref{def:twist} and Lemma~\ref{lem:primitive-twist}{\rm(1)}, it remains to prove that $\mu$ is a Maurer--Cartan element of $\Cbar_{\sg,R,E_C}^*(\widetilde C)$ and that
$$
                 D_C=D_{\widetilde C}+\operatorname{ad}_{s\mu}.
$$

{\rm(1)} We first prove that $\mu$ is a Maurer--Cartan element of $\Cbar_{\sg,R,E_C}^*(\widetilde C)$.
Since $|\mu|=2$, we have $|s\mu|=1$.
It remains to prove $D_{\widetilde C}(s\mu)+M(s\mu,s\mu)=0$.
It suffices to show that $D_{\widetilde C}(s\mu)=0$ and $M(s\mu,s\mu)=0$.
Let $\pi:T^c(sV)\to sV$ be the projection onto the summand of tensor length one.
Every primitive element $z\in T^c(sV)$ satisfies $z=\pi(z)$.
Thus a primitive element with zero image under $\pi$ is zero.
It is therefore enough to prove that $D_{\widetilde C}(s\mu)$ and $M(s\mu,s\mu)$ are primitive and that $\pi D_{\widetilde C}(s\mu)=0$ and $\pi M(s\mu,s\mu)=0$.

Since $s\mu\in sV$, we have $\Delta(s\mu)=1\tensor s\mu+s\mu\tensor1$.
As $D_{\widetilde C}$ is a coderivation and $D_{\widetilde C}(1)=0$,
$$
\begin{aligned}
 \Delta D_{\widetilde C}(s\mu)
 &=(D_{\widetilde C}\tensor{\rm id}
   +{\rm id}\tensor D_{\widetilde C})\Delta(s\mu)\\
 &=(D_{\widetilde C}\tensor{\rm id}
   +{\rm id}\tensor D_{\widetilde C})
   (1\tensor s\mu+s\mu\tensor1)\\
 &=D_{\widetilde C}(1)\tensor s\mu
   +1\tensor D_{\widetilde C}(s\mu)+D_{\widetilde C}(s\mu)\tensor1
   +(-1)^{|s\mu|}s\mu\tensor D_{\widetilde C}(1)\\
 &=1\tensor D_{\widetilde C}(s\mu)
   +D_{\widetilde C}(s\mu)\tensor1.
\end{aligned}
$$
Thus $D_{\widetilde C}(s\mu)$ is primitive.
Since $M$ is a coalgebra map and $|s\mu|=1$, we compute
$$
\begin{aligned}
 \Delta M(s\mu,s\mu)
 ={}&(M\tensor M)
 ({\rm id}\tensor\mathfrak t\tensor{\rm id})
 \bigl(\Delta(s\mu)\tensor\Delta(s\mu)\bigr)\\
 ={}&(M\tensor M)
 ({\rm id}\tensor\mathfrak t\tensor{\rm id})
 \bigl((1\tensor s\mu+s\mu\tensor1)
 \tensor(1\tensor s\mu+s\mu\tensor1)\bigr)\\
 ={}&(M\tensor M)\bigl(
 1\tensor1\tensor s\mu\tensor s\mu
 -1\tensor s\mu\tensor s\mu\tensor1
 +s\mu\tensor1\tensor1\tensor s\mu
 +s\mu\tensor s\mu\tensor1\tensor1\bigr)\\
 ={}&1\tensor M(s\mu,s\mu)
   -M(1,s\mu)\tensor M(s\mu,1)+M(s\mu,1)\tensor M(1,s\mu)
   +M(s\mu,s\mu)\tensor1\\
 ={}&1\tensor M(s\mu,s\mu)-s\mu\tensor s\mu
   +s\mu\tensor s\mu+M(s\mu,s\mu)\tensor1\\
 ={}&1\tensor M(s\mu,s\mu)+M(s\mu,s\mu)\tensor1.
\end{aligned}
$$
Thus $M(s\mu,s\mu)$ is also primitive.
By the definitions of the Taylor components and $m_1=d_{\widetilde C}$,
$$
 \pi D_{\widetilde C}(s\mu)
 =D_1(s\mu)
 =s\bigl(m_1(\mu)\bigr)
 =s(d_{\widetilde C}\mu).
$$
Since $|\mu|=2$, the Koszul convention gives $(s\tensor s)(\mu\tensor\mu)=s\mu\tensor s\mu$, and the definition of the brace operation gives $\mu\{\mu\}=(-1)^{|\mu|}\mu_{1,1}(\mu,\mu) =\mu_{1,1}(\mu,\mu)$.
Hence
$$
 \pi M(s\mu,s\mu)
 =M_{1,1}(s\mu,s\mu)
 =s\bigl(\mu_{1,1}(\mu,\mu)\bigr)
 =s(\mu\{\mu\}).
$$

It remains to prove $d_{\widetilde C}\mu=0$ and $\mu\{\mu\}=0$.
Let $x_1,x_2,x_3\in J_C$.
The formula for the normalized relative Hochschild differential in degree two gives
$$
\begin{aligned}
 &(d_{\widetilde C}\mu)
   (sx_1\tensor sx_2\tensor sx_3)\\
 ={}&-x_1\cdot_{\widetilde C}\mu(sx_2\tensor sx_3)
 +\mu\bigl(s(x_1\cdot_{\widetilde C}x_2)\tensor sx_3\bigr)\\
 &-\mu\bigl(sx_1\tensor s(x_2\cdot_{\widetilde C}x_3)\bigr)
 +\mu(sx_1\tensor sx_2)\cdot_{\widetilde C}x_3\\
 ={}&-x_1\cdot_{\widetilde C}(x_2x_3)
 +(x_1x_2)\cdot_{\widetilde C}x_3\\
 ={}&0.
\end{aligned}
$$
Here $x_ix_j$ in the values of $\mu$ denotes multiplication in $C$ and belongs to $J_C$, whereas every product denoted by $\cdot_{\widetilde C}$ is taken in $\widetilde C$ and vanishes on $J_C\tensor J_C$.
Hence $d_{\widetilde C}\mu=0$.

By the right singular Hochschild brace formula \cite[Definition~8.2, (8.3) and Lemma~8.5]{ChenLiWang2025}, we have
$$
\begin{aligned}
 &\mu\{\mu\}(sx_1\tensor sx_2\tensor sx_3)\\
 ={}&\mu\bigl(s(x_1x_2)\tensor sx_3\bigr)
     -\mu\bigl(sx_1\tensor s(x_2x_3)\bigr)\\
 ={}&(x_1x_2)x_3-x_1(x_2x_3)\\
 ={}&0,
\end{aligned}
$$
where all products in the displayed calculation are taken in $C$ and the last equality follows from its associativity.
Thus $\mu\{\mu\}=0$.

The two primitive elements therefore have zero image under $\pi$, so $D_{\widetilde C}(s\mu)=0$ and $M(s\mu,s\mu)=0$.
Consequently, $D_{\widetilde C}(s\mu)+M(s\mu,s\mu)=0$, and $\mu$ is a Maurer--Cartan element of $\Cbar_{\sg,R,E_C}^*(\widetilde C)$.

{\rm(2)} We now prove $D_C=D_{\widetilde C}+\operatorname{ad}_{s\mu}$.
A coderivation on $T^c(sV)$ is determined by its composite with $\pi:T^c(sV)\to sV$.
For $r\geq1$, put
$$
 D_{C,r}:=\pi D_C\big|_{(sV)^{\tensor r}},\qquad
 D_{\widetilde C,r}:=\pi D_{\widetilde C}
                         \big|_{(sV)^{\tensor r}}.
$$
Thus the required equality of coderivations is equivalent to
$$
 D_{C,r}-D_{\widetilde C,r}
 =\pi\operatorname{ad}_{s\mu}
   \big|_{(sV)^{\tensor r}}
 \qquad(r\geq1).
$$
We prove the equality of the Taylor components separately for $r=1$, $r=2$ and $r\geq3$.

For $r=1$, the required equality is
\begin{equation}\label{eq:source-first-component-goal}
 (D_{C,1}-D_{\widetilde C,1})(s[f])
 =\pi\operatorname{ad}_{s\mu}(s[f])
\end{equation}
for every homogeneous $[f]\in V$.

We now calculate the right-hand side of \eqref{eq:source-first-component-goal}.
For homogeneous $u,v\in V$, the Koszul convention and Definition~\ref{def:brace-B-infinity} give
$$
\begin{aligned}
 \mu_{1,1}(u,v)
 &=s^{-1}M_{1,1}(s\tensor s)(u\tensor v)\\
 &=(-1)^{|u|}s^{-1}M_{1,1}(su,sv),\\
 u\{v\}&=(-1)^{|u|}\mu_{1,1}(u,v)
          =s^{-1}M_{1,1}(su,sv).
\end{aligned}
$$
Hence $\pi M(su,sv)=M_{1,1}(su,sv)=s(u\{v\})$.
Since $|s[f]|=|f|-1$, the definition of $\operatorname{ad}_{s\mu}$ now gives
$$
\begin{aligned}
 \pi\operatorname{ad}_{s\mu}(s[f])
 &=\pi M(s\mu,s[f])-(-1)^{|f|-1}\pi M(s[f],s\mu)\\
 &=s(\mu\{[f]\})-(-1)^{|f|-1}s([f]\{\mu\})\\
 &=s\bigl(\mu\{[f]\}-(-1)^{|f|-1}[f]\{\mu\}\bigr).
\end{aligned}
$$

We now calculate the left-hand side of \eqref{eq:source-first-component-goal}.
By Definition~\ref{def:brace-B-infinity}, $m_1=s^{-1}D_1s$.
By Definition~\ref{def:relative-singular}, the first operations
associated with $D_C$ and $D_{\widetilde C}$ are $d_C$ and
$d_{\widetilde C}$, respectively. Hence $D_{C,1}s=sd_C$ and
$D_{\widetilde C,1}s=sd_{\widetilde C}$.
Consequently, the left-hand side of \eqref{eq:source-first-component-goal} is
$$
 (D_{C,1}-D_{\widetilde C,1})(s[f])
 =s\bigl(d_C([f])\bigr)-s\bigl(d_{\widetilde C}([f])\bigr)
 =s\bigl((d_C-d_{\widetilde C})([f])\bigr).
$$
We shall prove that
\begin{equation}\label{eq:source-first-component-change}
 (d_C-d_{\widetilde C})([f])
 =\mu\{[f]\}-(-1)^{|f|-1}[f]\{\mu\}.
\end{equation}
Choose $p\geq0$ and a homogeneous representative
$$
 f:(sJ_C)^{\tensor_{E_C}a}\longrightarrow
 (sJ_C)^{\tensor_{E_C}p}\tensor_{E_C}C,
 \qquad a:=|f|+p\geq0,
$$
of $[f]$.
By the definition of the differential on the direct limit,
$$
\begin{aligned}
 (d_C-d_{\widetilde C})([f])
=d_C([f])-d_{\widetilde C}([f])=[d_{C,p}(f)]-[d_{\widetilde C,p}(f)]=[d_{C,p}(f)-d_{\widetilde C,p}(f)].
\end{aligned}
$$
Let $x_1,\ldots,x_{a+1}\in J_C$.
Write each homogeneous summand of $f(sx_2\tensor\cdots\tensor sx_{a+1})$ as $sy_1\tensor\cdots\tensor sy_p\tensor c$, where $y_1,\ldots,y_p\in J_C$ and $c\in C$, and write each homogeneous summand of $f(sx_1\tensor\cdots\tensor sx_a)$ as $sz_1\tensor\cdots\tensor sz_p\tensor d$, where $z_1,\ldots,z_p\in J_C$ and $d\in C$.
All formulas below are extended linearly over the homogeneous summands above.

Let $p\geq1$.
By Definition~\ref{def:relative-singular}, subtracting the two relative Hochschild differential formulas first gives
$$
\begin{aligned}
 &\bigl((d_{C,p}-d_{\widetilde C,p})(f)\bigr)
 (sx_1\tensor\cdots\tensor sx_{a+1})\\
=&-(-1)^{|f|}\Bigl(
 x_1\blacktriangleright_C
 (sy_1\tensor\cdots\tensor sy_p\tensor c)
 -x_1\blacktriangleright_{\widetilde C}
 (sy_1\tensor\cdots\tensor sy_p\tensor c)\Bigr)\\
+&\sum_{i=1}^{a}(-1)^{|f|-i+1}\Bigl(
 f\bigl(sx_1\tensor\cdots\tensor
 s(x_i\cdot_Cx_{i+1})\tensor\cdots\tensor sx_{a+1}\bigr)
 -f\bigl(sx_1\tensor\cdots\tensor
 s(x_i\cdot_{\widetilde C}x_{i+1})\tensor\cdots
 \tensor sx_{a+1}\bigr)\Bigr)\\
+&(-1)^p\Bigl(
 (sz_1\tensor\cdots\tensor sz_p\tensor d)\cdot_Cx_{a+1}
 -(sz_1\tensor\cdots\tensor sz_p\tensor d)
   \cdot_{\widetilde C}x_{a+1}\Bigr).
\end{aligned}
$$
Using the two bimodule actions in Definition~\ref{def:relative-singular} and $m_C-m_{\widetilde C}=\mu_0$, the difference of the two relative Hochschild differentials becomes
$$
\begin{aligned}
 &\bigl((d_{C,p}-d_{\widetilde C,p})(f)\bigr)
 (sx_1\tensor\cdots\tensor sx_{a+1})\\
={}&-(-1)^{|f|}
 s\mu_0(x_1,y_1)\tensor sy_2\tensor\cdots\tensor sy_p\tensor c\\
&-(-1)^{|f|}\sum_{l=1}^{p-1}(-1)^l
 sx_1\tensor sy_1\tensor\cdots\tensor
 s\mu_0(y_l,y_{l+1})\tensor\cdots\tensor sy_p\tensor c\\
&-(-1)^{|f|+p}
 sx_1\tensor sy_1\tensor\cdots\tensor sy_{p-1}
 \tensor\mu_0(y_p,c)\\
&+\sum_{i=1}^{a}(-1)^{|f|-i+1}
 f\bigl(sx_1\tensor\cdots\tensor
 s\mu_0(x_i,x_{i+1})\tensor\cdots\tensor sx_{a+1}\bigr)\\
&+(-1)^p sz_1\tensor\cdots\tensor sz_p
 \tensor\mu_0(d,x_{a+1}).
\end{aligned}
$$
The single-brace formula of \cite[Definition~8.2 and (8.3)]{ChenLiWang2025} gives
$$
\begin{aligned}
 &\mu\{f\}(sx_1\tensor\cdots\tensor sx_{a+1})\\
 ={}&-(-1)^{|f|+p}
 sx_1\tensor sy_1\tensor\cdots\tensor sy_{p-1}
 \tensor\mu_0(y_p,c)\\
 &+(-1)^p sz_1\tensor\cdots\tensor sz_p
 \tensor\mu_0(d,x_{a+1}).
\end{aligned}
$$
The algebraic formula in \cite[Lemma~8.5]{ChenLiWang2025} gives
$$
\begin{aligned}
 &f\{\mu\}(sx_1\tensor\cdots\tensor sx_{a+1})\\
 ={}&\sum_{i=1}^{a}(-1)^{i-1}
 f\bigl(sx_1\tensor\cdots\tensor
 s\mu_0(x_i,x_{i+1})\tensor\cdots\tensor sx_{a+1}\bigr)\\
 &-s\mu_0(x_1,y_1)\tensor sy_2\tensor\cdots\tensor sy_p\tensor c\\
 &-\sum_{l=1}^{p-1}(-1)^l
 sx_1\tensor sy_1\tensor\cdots\tensor
 s\mu_0(y_l,y_{l+1})\tensor\cdots\tensor sy_p\tensor c.
\end{aligned}
$$
Comparing the three displayed formulas gives
$$
 d_{C,p}(f)-d_{\widetilde C,p}(f)
 =\mu\{f\}-(-1)^{|f|-1}f\{\mu\}.
$$

For $p=0$, Definition~\ref{def:relative-singular} and the single-brace formulas of \cite[Definition~8.2 and (8.3)]{ChenLiWang2025} give
$$
\begin{aligned}
 &\bigl((d_{C,0}-d_{\widetilde C,0})(f)\bigr)
 (sx_1\tensor\cdots\tensor sx_{a+1})\\
={}&-(-1)^{|f|}
 \mu_0\bigl(x_1,f(sx_2\tensor\cdots\tensor sx_{a+1})\bigr)\\
&+\sum_{i=1}^{a}(-1)^{|f|-i+1}
 f\bigl(sx_1\tensor\cdots\tensor
 s\mu_0(x_i,x_{i+1})\tensor\cdots\tensor sx_{a+1}\bigr)\\
&+\mu_0\bigl(f(sx_1\tensor\cdots\tensor sx_a),x_{a+1}\bigr),\\
 &\mu\{f\}(sx_1\tensor\cdots\tensor sx_{a+1})\\
={}&-(-1)^{|f|}
 \mu_0\bigl(x_1,f(sx_2\tensor\cdots\tensor sx_{a+1})\bigr)
 +\mu_0\bigl(f(sx_1\tensor\cdots\tensor sx_a),x_{a+1}\bigr),\\
 &f\{\mu\}(sx_1\tensor\cdots\tensor sx_{a+1})\\
={}&\sum_{i=1}^{a}(-1)^{i-1}
 f\bigl(sx_1\tensor\cdots\tensor
 s\mu_0(x_i,x_{i+1})\tensor\cdots\tensor sx_{a+1}\bigr).
\end{aligned}
$$
Comparing the three displayed formulas gives the same equality for $p=0$.
Taking the images of the finite-stage formulas in $V$ gives
$$
\begin{aligned}
 (d_C-d_{\widetilde C})([f])
 =[d_{C,p}(f)-d_{\widetilde C,p}(f)]=\mu\{[f]\}-(-1)^{|f|-1}[f]\{\mu\}.
\end{aligned}
$$
Here the second equality follows from \cite[Remark~8.6]{ChenLiWang2025}.
Hence \eqref{eq:source-first-component-change}.
It follows that
$$
 (D_{C,1}-D_{\widetilde C,1})(s[f])
 =s\bigl(\mu\{[f]\}-(-1)^{|f|-1}[f]\{\mu\}\bigr).
$$
The right-hand side is equal to the value of $\pi\operatorname{ad}_{s\mu}(s[f])$ computed above.
Thus \eqref{eq:source-first-component-goal} holds.

For $r=2$, the required equality is
$$
 (D_{C,2}-D_{\widetilde C,2})(s[f]\tensor s[g])
 =\pi\operatorname{ad}_{s\mu}(s[f]\tensor s[g])
$$
for all homogeneous $[f],[g]\in V$.
We first calculate the right-hand side.
By the definition of $\operatorname{ad}_{s\mu}$,
$$
\begin{aligned}
 &\pi\operatorname{ad}_{s\mu}(s[f]\tensor s[g])\\
=&M_{1,2}(s\mu\tensor s[f]\tensor s[g])-(-1)^{|s[f]|+|s[g]|}
M_{2,1}(s[f]\tensor s[g]\tensor s\mu).
\end{aligned}
$$
Since $M_{2,1}=0$, Definition~\ref{def:brace-B-infinity} gives
$$
 \pi\operatorname{ad}_{s\mu}(s[f]\tensor s[g])
 =M_{1,2}(s\mu\tensor s[f]\tensor s[g]).
$$
Since $|\mu|=2$, the definition of the brace operation gives
$$
 \mu\{[f],[g]\}
 =(-1)^{|f|}\mu_{1,2}(\mu,[f],[g]).
$$
Moreover, the definition of $\mu_{1,2}$ and the Koszul convention give
$$
\begin{aligned}
 \mu_{1,2}(\mu,[f],[g])
 & =s^{-1}M_{1,2}
 (s\tensor s^{\tensor2})(\mu\tensor[f]\tensor[g])\\
 & =(-1)^{|f|}s^{-1}M_{1,2}
 (s\mu\tensor s[f]\tensor s[g]).
\end{aligned}
$$
The two factors $(-1)^{|f|}$ cancel.
Hence
$$
 \pi\operatorname{ad}_{s\mu}(s[f]\tensor s[g])
 =s\bigl(\mu\{[f],[g]\}\bigr).
$$

We now calculate the left-hand side.
By Definition~\ref{def:brace-B-infinity},
$$
 m_{C,2}=s^{-1}D_{C,2}(s\tensor s).
$$
Since $m_{C,2}$ is the cup product and
$$
 (s\tensor s)([f]\tensor[g])
 =(-1)^{|f|}s[f]\tensor s[g],
$$
we obtain
$$
\begin{aligned}
 {}[f]\smile_C[g]
 &=m_{C,2}([f],[g])=s^{-1}D_{C,2}(s\tensor s)([f]\tensor[g])=(-1)^{|f|}s^{-1}D_{C,2}(s[f]\tensor s[g]).
\end{aligned}
$$
Therefore
$$
 D_{C,2}(s[f]\tensor s[g])
 =(-1)^{|f|}s([f]\smile_C[g]).
$$
The same calculation for $\widetilde C$ gives
$$
 D_{\widetilde C,2}(s[f]\tensor s[g])
 =(-1)^{|f|}s([f]\smile_{\widetilde C}[g]).
$$
Subtracting the equalities for $D_{C,2}$ and $D_{\widetilde C,2}$, we obtain
$$
 (D_{C,2}-D_{\widetilde C,2})(s[f]\tensor s[g])
 =(-1)^{|f|}s\bigl(
 [f]\smile_C[g]-[f]\smile_{\widetilde C}[g]\bigr).
$$
Choose homogeneous representatives
$$
 f:(sJ_C)^{\tensor_{E_C}a}\longrightarrow
      (sJ_C)^{\tensor_{E_C}p}\tensor_{E_C}C,
 \qquad
 g:(sJ_C)^{\tensor_{E_C}b}\longrightarrow
      (sJ_C)^{\tensor_{E_C}q}\tensor_{E_C}C
$$
of $[f]$ and $[g]$ in $V_p^*$ and $V_q^*$, respectively.
By \cite[Subsection~8.2.2, formula~(8.2)]{ChenLiWang2025}, the two finite-stage cup products are
$$
\begin{aligned}
 f\smile_Cg
 & =({\rm id}^{\tensor(p+q)}\tensor m_C)
 ({\rm id}^{\tensor q}\tensor f\tensor{\rm id}_C)
 ({\rm id}^{\tensor a}\tensor g),\\
 f\smile_{\widetilde C}g
 & =({\rm id}^{\tensor(p+q)}\tensor m_{\widetilde C})
 ({\rm id}^{\tensor q}\tensor f\tensor{\rm id}_C)
 ({\rm id}^{\tensor a}\tensor g).
\end{aligned}
$$
Here we use the common vector space $E_C\oplus J_C$ to identify the domains and codomains of the maps defining the two cup products.
Since $m_C-m_{\widetilde C}=\mu_0$, subtracting the two formulas gives
\begin{equation}\label{eq:cup-difference-composition}
 f\smile_Cg-f\smile_{\widetilde C}g
 =({\rm id}^{\tensor(p+q)}\tensor\mu_0)
 ({\rm id}^{\tensor q}\tensor f\tensor{\rm id}_C)
 ({\rm id}^{\tensor a}\tensor g),
\end{equation}
where all tensor products are over $E_C$.

Put $\vartheta:=s\pi_J:C\to sJ_C$ and
$$
 \overline f:=({\rm id}^{\tensor p}\tensor\vartheta)f,
 \qquad
 \overline g:=({\rm id}^{\tensor q}\tensor\vartheta)g.
$$
Since $\mu$ is binary and has no suspended coefficient factors, specializing \cite[Definition~8.2, formula~(8.3), and Lemma~8.5]{ChenLiWang2025} to the case of the cochains $\mu,f,g$ gives
$$
 \mu\{f,g\}
 =({\rm id}^{\tensor(p+q)}\tensor\mu)
   ({\rm id}^{\tensor q}\tensor\overline f\tensor{\rm id}_{sJ_C})
   ({\rm id}^{\tensor a}\tensor\overline g).
$$
In comparing this formula with \eqref{eq:cup-difference-composition}, the degree-$-1$ map $\vartheta$ passes the map $f$ of degree $|f|$.
By the Koszul rule, this contributes the sign $(-1)^{|\vartheta||f|}=(-1)^{|f|}$.
The remaining copies of $\vartheta$ occur immediately before $\mu$, and $\mu(\vartheta\tensor\vartheta)=\mu_0$.
Taking the resulting classes in $V$, we obtain
$$
[f]\smile_C[g]-[f]\smile_{\widetilde C}[g]
 =(-1)^{|f|}\mu\{[f],[g]\}.
$$
It follows that
$$
\begin{aligned}
 &(D_{C,2}-D_{\widetilde C,2})(s[f]\tensor s[g])\\
 ={}&(-1)^{|f|}s
      \bigl([f]\smile_C[g]-[f]\smile_{\widetilde C}[g]\bigr)\\
 ={}&(-1)^{|f|}s\bigl((-1)^{|f|}\mu\{[f],[g]\}\bigr)\\
 ={}&s\bigl(\mu\{[f],[g]\}\bigr).
\end{aligned}
$$
Thus the required equality holds for $r=2$.

For $r\geq3$, we first prove
$$
 \pi\operatorname{ad}_{s\mu}
 \big|_{(sV)^{\tensor r}}=0.
$$
For homogeneous $[f_1],\ldots,[f_r]\in V$, its value is
$$
\begin{aligned}
&\pi\operatorname{ad}_{s\mu}
  (s[f_1]\tensor\cdots\tensor s[f_r])\\
={}&M_{1,r}(s\mu\tensor s[f_1]\tensor\cdots\tensor s[f_r])\\
&-(-1)^{|s[f_1]|+\cdots+|s[f_r]|}
  M_{r,1}(s[f_1]\tensor\cdots\tensor s[f_r]\tensor s\mu)\\
={}&0.
\end{aligned}
$$
Indeed, $\mu:(sJ_C)^{\tensor_{E_C}2}\to C$, so $\mu\{[f_1],\ldots,[f_r]\}=0$ for $r>2$ and the first term is zero; the second is zero because $M_{r,1}=0$ for $r>1$ in a brace $B_\infty$-algebra.
Thus
$$
 \pi\operatorname{ad}_{s\mu}
 \big|_{(sV)^{\tensor r}}=0
 \qquad(r\geq3).
$$
On the other hand, the brace $B_\infty$ condition gives $D_{C,r}=D_{\widetilde C,r}=0$ for $r\geq3$.
Hence
$$
 D_{C,r}-D_{\widetilde C,r}
 =\pi\operatorname{ad}_{s\mu}
   \big|_{(sV)^{\tensor r}}=0
 \qquad(r\geq3).
$$
Hence the maps of the two coderivations agree for every $r\geq1$.
Since a coderivation on $T^c(sV)$ is determined by its Taylor components,
$$
 D_C=D_{\widetilde C}+\operatorname{ad}_{s\mu}.
$$
This proves the proposition. $\square$

We now twist the morphisms of Lemma~\ref{lem:clw-morphisms}. The opposite
$B_\infty$-structure changes the sign of the twisting parameter.
\begin{Lem}\label{lem:opposite-twist}
$$
 \left(\operatorname{Tw}_\mu\bigl(
   \Cbar_{\sg,R,E_C}^*(\widetilde C)\bigr)\right)^{\opp}
 =
 \operatorname{Tw}_{-\mu}\left(
 \Cbar_{\sg,R,E_C}^*(\widetilde C)^{\opp}\right).
$$
\end{Lem}

{\it Proof.}
Put $V:=\Cbar_{\sg,R,E_C}^*(\widetilde C)$ and write the dg-bialgebra representing $V$ as $(T^c(sV),\Delta,D,M)$.
By the definitions of the twist and the opposite, the dg-bialgebra representing the left-hand side of the equality in Lemma~\ref{lem:opposite-twist} is
$$
 \bigl(T^c(sV),\Delta,
 D+\operatorname{ad}^{M}_{s\mu},M^{\opp}\bigr),
$$
whereas the dg-bialgebra on the right-hand side is
$$
 \bigl(T^c(sV),\Delta,
 D+\operatorname{ad}^{M^{\opp}}_{-s\mu},M^{\opp}\bigr).
$$
Thus the tensor coalgebra, coproduct and multiplication already agree.
It remains to prove
$$
 \operatorname{ad}^{\,M^{\opp}}_{-s\mu}
 =\operatorname{ad}^{\,M}_{s\mu}.
$$
The opposite dg-bialgebra multiplication is
$$
 M^{\opp}(x,y):=(-1)^{|x||y|}M(y,x).
$$
For homogeneous $z$, the definition of the inner derivation and the equality $|s\mu|=1$ give
$$
\begin{aligned}
 \operatorname{ad}^{M^{\opp}}_{-s\mu}(z)
 &=M^{\opp}(-s\mu,z)-(-1)^{|z|}M^{\opp}(z,-s\mu)\\
 &=-(-1)^{|z|}M(z,s\mu)+M(s\mu,z)\\
 &=M(s\mu,z)-(-1)^{|z|}M(z,s\mu)\\
 &=\operatorname{ad}^{M}_{s\mu}(z).
\end{aligned}
$$
Hence $\operatorname{ad}^{M^{\opp}}_{-s\mu}=\operatorname{ad}^{M}_{s\mu}$.
The two displayed dg bialgebras therefore have the same coderivation and are equal.
Hence they represent the same $B_\infty$-algebra.
This proves Lemma~\ref{lem:opposite-twist}. $\square$

\medskip
For each pair $i,j$, consider the finite descending filtration
$$
 e_jJ_Ce_i\supseteq e_jJ_C^2e_i\supseteq
 e_jJ_C^3e_i\supseteq\cdots\supseteq0.
$$
Since $J_C$ is nilpotent, the radical filtration is finite.
Starting with a basis of its last nonzero term and extending the basis successively, choose a $k$-basis $\mathcal B_{ji}$ of $e_jJ_Ce_i$ such that $ \mathcal B_{ji}\cap e_jJ_C^re_i$ is a basis of $e_jJ_C^re_i$ for every $r\geq1$.
Let $Q$ be the radical quiver determined by the bases $\mathcal B_{ji}$ as in Definition~\ref{def:radical-quiver-leavitt}.
We assume that $Q$ has no sinks; equivalently, $J_Ce_i\ne0$ for every vertex $i$.

The algebras $C$ and $\widetilde C$ have the same $E_C$-$E_C$-bimodule $J_C$.
With respect to the bases $\mathcal B_{ji}$ chosen above, they therefore have the same radical quiver $Q$.
Put $\overline{L(Q)}:=L(Q)/E_C$.
The graded algebra $L(Q)$ depends only on $Q$.

The $E_C$-$E_C$-bimodule isomorphism $kQ_1\to J_C$, $\alpha\mapsto j_\alpha$, extends to a $k$-algebra isomorphism
$$
 A_Q=E_C\oplus kQ_1\xrightarrow{\ \sim\ }\widetilde C.
$$
We use this $k$-algebra isomorphism to identify the radical-square-zero algebra $A_Q$ in Lemma~\ref{lem:clw-morphisms} with $\widetilde C$.

Lemma~\ref{lem:clw-morphisms}{\rm(1)} gives
$$
 \Cbar_{\sg,R,E_C}^*(\widetilde C)
 \xleftarrow{\ \kappa\ }
 \Cbar_{\sg,R}^*(Q)
 \xrightarrow{\ \rho\ }
 \Chat^*(L(Q))
 \xrightarrow{\ (\Phi_1,\Phi_2,\ldots)\ }
 \Cbar_{E_C}^*(L(Q))^{\opp}.
$$
Since $\kappa$ is a strict $B_\infty$-isomorphism, these morphisms define a $B_\infty$-morphism
$$
 G:=(\Phi_1,\Phi_2,\ldots)\rho\kappa^{-1}:
 \Cbar_{\sg,R,E_C}^*(\widetilde C)
 \longrightarrow\Cbar_{E_C}^*(L(Q))^{\opp}.
$$
Denote the Taylor components of $G$ by $G_1,G_2,\ldots$.
Since $\kappa$ and $\rho$ are strict, the first Taylor component of $G$ is $G_1=\Phi_1\rho\kappa^{-1}$.
Define
$$
 d_\mu:=G_1(\mu)=\Phi_1\rho\kappa^{-1}(\mu)
 \in\Cbar_{E_C}^{2,1}(L(Q)).
$$
Thus $d_\mu$ is an $E_C$-$E_C$-linear map
$d_\mu:s\overline {L(Q)}\longrightarrow L(Q)$.
Equivalently, it determines a degree-one map on $L(Q)$, vanishing on $E_C$, whose value at $x\in L(Q)$ is $d_\mu(s\overline x)$.
We denote the resulting degree-one map again by $d_\mu$.

Taking $B_\infty$ opposites gives
$$
G^{\opp}:
 \Cbar_{\sg,R,E_C}^*(\widetilde C)^{\opp}
 \longrightarrow\Cbar_{E_C}^*(L(Q)).
$$
Its first Taylor component sends $\mu$ to $d_\mu$, and hence sends $-\mu$ to $-d_\mu$.
Proposition~\ref{prop:source-twist} and Lemma~\ref{lem:opposite-twist} show that $-\mu$ is a Maurer--Cartan element of the source of $G^{\opp}$.
Therefore Lemma~\ref{lem:primitive-twist} gives a $B_\infty$-morphism
\begin{equation}\label{eq:raw-twisted-G}
 \operatorname{Tw}_{-\mu}\left(
 \Cbar_{\sg,R,E_C}^*(\widetilde C)^{\opp}\right)
 \xrightarrow{\ G^{\opp}\ }
 \operatorname{Tw}_{-d_\mu}\left(\Cbar_{E_C}^*(L(Q))\right).
\end{equation}
In particular, $-d_\mu$ is a Maurer--Cartan element of $\Cbar_{E_C}^*(L(Q))$.

We now compute $d_\mu$ on the real arrows.
We read path multiplication from right to left.
We use the $E_C$-$E_C$-bimodule isomorphism $kQ_1\to J_C$, $\alpha\longmapsto j_\alpha$, fixed in Definition~\ref{def:radical-quiver-leavitt}.
We apply the notation preceding Lemma~\ref{lem:chen-wang} to $A=C$.
Under the isomorphism $kQ_1\to J_C$, $\alpha\mapsto j_\alpha$, the structure constants of $\nu$ are the coefficients defined below.
For every path $p=\beta_2\beta_1\in Q_2$, we have
$$
 j_{\beta_2}j_{\beta_1}
 \in e_{t(\beta_2)}J_C^2e_{s(\beta_1)}
 =e_{t(p)}J_C^2e_{s(p)}
 \subseteq e_{t(p)}J_Ce_{s(p)}.
$$
Hence there are unique scalars $\lambda_{p,\alpha}\in k$ such that
$$
 j_{\beta_2}j_{\beta_1}
 =\sum_{\alpha\parallel p}\lambda_{p,\alpha}j_\alpha
 \qquad\text{in }J_C,
$$
where $\alpha\parallel p$ means that $\alpha$ and $p$ have the same source and target.

\begin{Lem}\label{lem:dmu-real}
The cochain $d_\mu$ is a degree-one derivation of $L(Q)$, satisfies $d_\mu^2=0$, and, for every real arrow $\alpha$,
$$
 d_\mu(\alpha)=
 \sum_{\substack{p\in Q_2\\\alpha\parallel p}}
       \lambda_{p,\alpha}p.
$$
\end{Lem}

{\it Proof.}
Put $\delta:=-d_\mu$.
As an element of $\Cbar_{E_C}^*(L(Q))$, it has cohomological degree two.
The Maurer--Cartan equation for $\delta$ is
$$
m_1(\delta)+\delta\{\delta\}=0.
$$
The two summands belong respectively to
$\Cbar_{E_C}^{3,2}(L(Q))$ and $\Cbar_{E_C}^{3,1}(L(Q))$.
Since they have different tensor degrees, both summands vanish.
Hence $m_1(\delta)=0$ and $\delta\{\delta\}=0$.
The first equality is the graded derivation identity
$$
 \delta(xy)=\delta(x)y+(-1)^{|x|}x\delta(y)
$$
for homogeneous $x,y\in L(Q)$.
Since $\delta\in\Cbar_{E_C}^{2,1}(L(Q))$, the second equality gives $\delta\{\delta\}(x)=\delta(\delta(x))=0$.
Thus $d_\mu=-\delta$ is a degree-one derivation and $d_\mu^2=0$.

It remains to compute the value of $d_\mu$ on each real arrow.
By the definition of the coefficients $\lambda_{p,\alpha}$, the cochain $\mu$ is
$$
 \mu=\sum_{p\in Q_2}\ \sum_{\alpha\parallel p}
 \lambda_{p,\alpha}f_{\alpha,p},
$$
where $f_{\alpha,p}$ is the cochain defined in Lemma~\ref{lem:clw-morphisms}{\rm(2)}, corresponding under $kQ_1\simeq J_C$.
Indeed, both sides have value $\sum_{\alpha\parallel p}\lambda_{p,\alpha}j_\alpha$ on the basis tensor $sj_{\beta_2}\tensor_{E_C}sj_{\beta_1}$ corresponding to $p=\beta_2\beta_1$, and both sides are determined by their values on the basis tensors.
It follows from Lemma~\ref{lem:clw-morphisms}{\rm(2)} that
$$
 \rho\kappa^{-1}(\mu)
 =-\sum_{p\in Q_2}\ \sum_{\alpha\parallel p}
 \lambda_{p,\alpha}s^{-1}\alpha^*p.
$$
For every real arrow $\gamma$, Lemma~\ref{lem:clw-morphisms}{\rm(2)} and the equality $d_\mu=\Phi_1\rho\kappa^{-1}(\mu)$ give
$$
\begin{aligned}
 d_\mu(\gamma)
 &=\sum_{p\in Q_2}\ \sum_{\alpha\parallel p}
   \lambda_{p,\alpha}
   \Phi_1(-s^{-1}\alpha^*p)(s\gamma)\\
 &=\sum_{p\in Q_2}\ \sum_{\alpha\parallel p}
   \lambda_{p,\alpha}\delta_{\gamma,\alpha}p\\
 &=\sum_{\substack{p\in Q_2\\\gamma\parallel p}}
   \lambda_{p,\gamma}p.
\end{aligned}
$$
This is the required formula. $\square$

By Lemma~\ref{lem:dmu-real}, the map $-d_\mu:L(Q)\to L(Q)$ is a degree-one derivation satisfying $(-d_\mu)^2=0$.
We write $(L(Q),-d_\mu)$ for the dg algebra whose underlying graded algebra is $L(Q)$ and whose differential is $-d_\mu$.

\begin{Lem}\label{lem:target-twist}
There is an equality of brace $B_\infty$-algebras
$$
 \operatorname{Tw}_{-d_\mu}\left(\Cbar_{E_C}^*(L(Q))\right)
 =
 \Cbar_{E_C}^*((L(Q),-d_\mu)).
$$
\end{Lem}

{\it Proof.}
Put $V:=\Cbar_{E_C}^*(L(Q))$ and write its dg-bialgebra as $(T^c(sV),\Delta,D_0,M)$.
The left-hand side of the equality in Lemma~\ref{lem:target-twist} is represented by
$$
 \bigl(T^c(sV),\Delta,
 D_0+\operatorname{ad}_{-sd_\mu},M\bigr).
$$
The deconcatenation coproduct and the brace formulas depend only on the underlying graded multiplication of $L(Q)$, which is unchanged in $(L(Q),-d_\mu)$.
Hence the right-hand side has the same tensor coalgebra, coproduct and coalgebra multiplication.
Let $D_{-d_\mu}$ denote its coderivation.
It remains to prove
$$
 D_{-d_\mu}=D_0+\operatorname{ad}_{-sd_\mu}.
$$
Since a coderivation on $T^c(sV)$ is determined by its composite with $\pi:T^c(sV)\to sV$, the required equality is equivalent to
$$
 \pi D_{-d_\mu}\big|_{(sV)^{\tensor r}}
 =\pi(D_0+\operatorname{ad}_{-sd_\mu})
   \big|_{(sV)^{\tensor r}}
 \qquad(r\geq1).
$$
Put $\delta:=-d_\mu$.
We first compare the Taylor components for $r=1$.
Let $f:(s\overline {L(Q)})^{\tensor_{E_C}m}\to L(Q)$ have degree $|f|$.
For homogeneous $a_1,\ldots,a_m\in L(Q)$, the change in the Hochschild differential caused by replacing the zero differential of $L(Q)$ with $\delta$ is
$$
\begin{aligned}
 &\bigl((d_{(L(Q),\delta)}-d_{(L(Q),0)})f\bigr)
 (s\overline a_1\tensor_{E_C}\cdots\tensor_{E_C}s\overline a_m)\\
=&
 \delta\bigl(f(s\overline a_1\tensor_{E_C}\cdots
                   \tensor_{E_C}s\overline a_m)\bigr)
 +\sum_{i=1}^{m}(-1)^{\varepsilon_i}
 f(s\overline a_1\tensor_{E_C}\cdots
   \tensor_{E_C}s\overline{\delta(a_i)}\tensor_{E_C}\cdots
   \tensor_{E_C}s\overline a_m),
\end{aligned}
$$
where $\varepsilon_i:=|f|+\sum_{j=1}^{i-1}(|a_j|-1)$, as in the formula for the Hochschild differential in \cite[Subsection~6.1]{ChenLiWang2025}.
By the formula for a brace with one inserted cochain in \cite[(6.1)]{ChenLiWang2025}, the first term is $\delta\{f\}$ and the summation is $-(-1)^{|f|-1}f\{\delta\}$.
Hence their sum is $\delta\{f\}-(-1)^{|f|-1}f\{\delta\}=[\delta,f]=[-d_\mu,f]$.
Consequently,
$$
 s^{-1}(D_{-d_\mu})_1(sf)
 =s^{-1}(D_0)_1(sf)+[-d_\mu,f]
 =s^{-1}\bigl(D_0+\operatorname{ad}_{-sd_\mu}\bigr)_1(sf).
$$
Thus the two composites with $\pi$ agree on $sV$.

Now let $r\geq2$.
The composite of $\operatorname{ad}_{-sd_\mu}$ on homogeneous $f_1,\ldots,f_r\in V$ is
$$
\begin{aligned}
&\pi\operatorname{ad}_{-sd_\mu}
  (sf_1\tensor\cdots\tensor sf_r)\\
={}&M_{1,r}(-sd_\mu\tensor sf_1\tensor\cdots\tensor sf_r)-(-1)^{|sf_1|+\cdots+|sf_r|}
  M_{r,1}(sf_1\tensor\cdots\tensor sf_r\tensor(-sd_\mu))\\
={}&0.
\end{aligned}
$$
Indeed, $d_\mu\in\Cbar_{E_C}^{2,1}(L(Q))$ gives $(-d_\mu)\{f_1,\ldots,f_r\}=0$ for $r\geq2$, so the first term is zero; the second is zero because $M_{r,1}=0$ for $r>1$ in a brace $B_\infty$-algebra.
Thus $\bigl(\operatorname{ad}_{-sd_\mu}\bigr)_r=0$.
Replacing the zero differential of $L(Q)$ with $-d_\mu$ changes only the restriction of $\pi D_0$ to $sV$.
Therefore
$$
 \bigl(\operatorname{ad}_{-sd_\mu}\bigr)_r=0,
 \qquad
 (D_{-d_\mu})_r=(D_0)_r
 \qquad(r\geq2).
$$
Hence the composites of $D_{-d_\mu}$ and $D_0+\operatorname{ad}_{-sd_\mu}$ with $\pi$ agree on every $(sV)^{\tensor r}$ for $r\geq1$.
The two coderivations are therefore equal.
This proves Lemma~\ref{lem:target-twist}. $\square$

\begin{Prop}\label{prop:core-isomorphism}
Let $C=E_C\oplus J_C$ be a split basic finite-dimensional algebra whose radical quiver has no sinks.
Then there is an isomorphism
$$
 \Cbar_{\sg,R,E_C}^*(C)^{\opp}
       \simeq \Cbar_{E_C}^*((L(Q),-d_\mu))
$$
in $\Ho_k(\Binf)$.
It is represented by the twisted $B_\infty$-morphism induced by $G^{\opp}$ in \eqref{eq:raw-twisted-G}.
\end{Prop}

On the source side, Proposition~\ref{prop:source-twist} and Lemma~\ref{lem:opposite-twist} give the equality of $B_\infty$-algebras
$$
 \Cbar_{\sg,R,E_C}^*(C)^{\opp}
 =\operatorname{Tw}_{-\mu}\left(
 \Cbar_{\sg,R,E_C}^*(\widetilde C)^{\opp}\right).
$$
Lemma~\ref{lem:target-twist} similarly identifies the target as a $B_\infty$-algebra.
Combining these identifications with \eqref{eq:raw-twisted-G}, we obtain the $B_\infty$-morphism
\begin{equation}\label{eq:twisted-G}
 \Cbar_{\sg,R,E_C}^*(C)^{\opp}
 =
 \operatorname{Tw}_{-\mu}\left(
 \Cbar_{\sg,R,E_C}^*(\widetilde C)^{\opp}\right)
 \longrightarrow
 \Cbar_{E_C}^*((L(Q),-d_\mu)).
\end{equation}

The morphism in \eqref{eq:twisted-G} factors as
$$
\begin{aligned}
&\operatorname{Tw}_{-\mu}\left(
 \Cbar_{\sg,R,E_C}^*(\widetilde C)^{\opp}\right)
\xra{\ (\kappa^{\opp})^{-1}\ }
\operatorname{Tw}_{(\kappa^{\opp})^{-1}(-\mu)}\left(
 \Cbar_{\sg,R}^*(Q)^{\opp}\right)
\\
&\hspace{15mm}\xra{\ \rho^{\opp}\ }
\operatorname{Tw}_{a}\left(\Chat^*(L(Q))^{\opp}\right)
\xra{\ (\Phi_1,\Phi_2,\ldots)^{\opp}\ }
\operatorname{Tw}_{-d_\mu}\left(\Cbar_{E_C}^*(L(Q))\right),
\end{aligned}
$$
where $a:=\rho^{\opp}(\kappa^{\opp})^{-1}(-\mu)$.
The first and last terms are identified with the source and target of \eqref{eq:twisted-G} by Proposition~\ref{prop:source-twist}, Lemma~\ref{lem:opposite-twist} and Lemma~\ref{lem:target-twist}.

To prove Proposition~\ref{prop:core-isomorphism}, it remains to show that the first Taylor component of the $B_\infty$-morphism in \eqref{eq:twisted-G} is a quasi-isomorphism.
We use the maps $\Phi_1$, $\Psi_1$ and $H$ in Lemma~\ref{lem:clw-homotopy-data}{\rm(1)} and prove a nilpotence bound independent of the tensor degree.

Let $\Phi_1$, $\Psi_1$ and $H$ be the maps in Lemma~\ref{lem:clw-homotopy-data}{\rm(1)}, applied with $E=E_C$.
By Lemma~\ref{lem:clw-homotopy-data}{\rm(2)},
\begin{equation}\label{eq:H-low-tensor-degree}
 H\big|_{\Cbar_{E_C}^{*,0}(L(Q))\oplus
          \Cbar_{E_C}^{*,1}(L(Q))}=0.
\end{equation}

Let $\delta_\mu:=[-d_\mu,-]$ be the change in the Hochschild differential on $\Cbar_{E_C}^*(L(Q))$.
For $n\in\mathbb Z$ and $m\geq0$, we have
$$
 \delta_\mu:\Cbar_{E_C}^{n,m}(L(Q))\longrightarrow
 \Cbar_{E_C}^{n+1,m}(L(Q)).
$$
For $m\geq2$, this is followed by
$$
 H:\Cbar_{E_C}^{n+1,m}(L(Q))\longrightarrow
 \Cbar_{E_C}^{n,m-1}(L(Q)).
$$
Thus the composite $H\delta_\mu:=H\circ\delta_\mu$ satisfies
$$
 H\delta_\mu:
 \Cbar_{E_C}^{n,m}(L(Q))\longrightarrow
 \Cbar_{E_C}^{n,m-1}(L(Q))
 \qquad(m\geq2).
$$
It is zero on tensor degrees zero and one by \eqref{eq:H-low-tensor-degree}.
Consequently, for $0\leq r\leq m$,
$$
 (H\delta_\mu)^r:
 \Cbar_{E_C}^{n,m}(L(Q))\longrightarrow
 \Cbar_{E_C}^{n,m-r}(L(Q)),
$$
and
$$
 (H\delta_\mu)^r\big|_{\Cbar_{E_C}^{n,m}(L(Q))}=0
 \qquad(r\geq m,\ r\geq1).
$$
The vanishing bound in the last display depends on $m$.
We now prove a bound independent of $m$.

\begin{Lem}\label{lem:uniform-nilpotence}
Choose $N\geq1$ with $J_C^N=0$.
On $\Cbar_{E_C}^*(L(Q))$,
$$
             (H\delta_\mu)^{N+1}=0.
$$
\end{Lem}

We prove Lemma~\ref{lem:uniform-nilpotence} through the next three lemmas.

For this $N$, the chosen basis of each $e_jJ_Ce_i$ is adapted to the finite filtration
$$
 e_jJ_Ce_i\supseteq e_jJ_C^2e_i\supseteq\cdots
                  \supseteq e_jJ_C^Ne_i=0.
$$
For every arrow $\alpha$, define its \emph{radical layer} by
$$
 \ell(\alpha):=\max\{r\geq1\mid j_\alpha\in J_C^r\}.
$$
Then
$$
 1\leq\ell(\alpha)\leq N-1,
 \qquad
 j_\alpha\in
 J_C^{\ell(\alpha)}\setminus J_C^{\ell(\alpha)+1}.
$$
Since the basis $\{j_\alpha\mid\alpha\in Q_1\}$ is adapted to the radical filtration, for every $r\geq1$ we have
$$
 J_C^r
 =\operatorname{span}_k
 \{j_\alpha\mid\ell(\alpha)\geq r\}.
$$
\begin{Lem}\label{lem:layer-inequality}
Let $p=\beta_2\beta_1\in Q_2$ and let $\alpha\in Q_1$.
If $\lambda_{p,\alpha}\ne0$, then
$$
 \ell(\alpha)\geq\ell(\beta_2)+\ell(\beta_1).
$$
In particular,
$\ell(\beta_i)<\ell(\alpha),\qquad i=1,2$.
\end{Lem}

{\it Proof.}
Put $r_i:=\ell(\beta_i)$.
The coefficient $\lambda_{p,\alpha}$ occurs in
$$
 j_{\beta_2}j_{\beta_1}
 =\sum_{\gamma\parallel\beta_2\beta_1}
   \lambda_{p,\gamma}j_\gamma.
$$
Since $j_{\beta_i}\in J_C^{r_i}$, we have
$$
 j_{\beta_2}j_{\beta_1}
 \in J_C^{r_2}J_C^{r_1}
 \subseteq J_C^{r_1+r_2}.
$$
Since $\{j_\gamma:\ell(\gamma)\geq r_1+r_2\}$ is a basis of
$J_C^{r_1+r_2}$, the assumption $\lambda_{p,\alpha}\ne0$ gives
$\ell(\alpha)\geq r_1+r_2$.
Since $r_1,r_2\geq1$,
$$
 \ell(\beta_i)=r_i<r_1+r_2\leq\ell(\alpha)
 \qquad(i=1,2).
$$
$\square$

The proof of Lemma~\ref{lem:uniform-nilpotence} uses the following strict decrease.
If a path $\beta_2\beta_1$ occurs with nonzero coefficient in $d_\mu(\alpha)$, then two successive applications of $H\delta_\mu$ replace the arrow $\alpha$ by $\beta_1$:
$$
 \alpha\longmapsto\beta_1,
 \qquad
 \lambda_{\beta_2\beta_1,\alpha}\ne0.
$$
By Lemma~\ref{lem:layer-inequality}, this replacement satisfies $\ell(\beta_1)<\ell(\alpha)$.
We first show how $H$ produces the real arrow to which $\delta_\mu$ is applied.

\begin{Lem}\label{lem:final-real-exposure}
Let $v$ be a nonzero monomial in $L(Q)$.
Let $r\geq2$, let $n\in\mathbb Z$, let $f\in\Cbar_{E_C}^{n,r}(L(Q))$ be homogeneous, and let $x_1,\ldots,x_{r-2}\in s\overline {L(Q)}$ be homogeneous.
If $v=e_i$ for some $i\in Q_0$, then
$$
 H(f)(x_1\tensor_{E_C}\cdots\tensor_{E_C}x_{r-2}
       \tensor_{E_C}s\overline v)=0.
$$
Otherwise, write
$$
 v=\beta_1^*\cdots\beta_p^*\alpha_m\cdots\alpha_1,
 \qquad p+m\geq1,
$$
where all $\alpha_i$ and $\beta_j$ belong to $Q_1$.
There are vertices $i_t\in Q_0$, homogeneous elements $a_t\in L(Q)e_{i_t}$ and $b_t\in e_{i_t}L(Q)$, and signs $\varepsilon_{t,\gamma}\in\{1,-1\}$ such that
$$
\begin{aligned}
&H(f)(x_1\tensor_{E_C}\cdots\tensor_{E_C}x_{r-2}
                 \tensor_{E_C}s\overline v)=
 \sum_{t=1}^{p+m}
 \sum_{\substack{\gamma\in Q_1\\s(\gamma)=i_t}}
 \varepsilon_{t,\gamma}
 f(x_1\tensor_{E_C}\cdots\tensor_{E_C}x_{r-2}
  \tensor_{E_C}s\overline{a_t\gamma^*}
  \tensor_{E_C}s\overline\gamma)b_t,
\end{aligned}
$$
where terms containing $s\overline{a_t\gamma^*}=0$ are omitted.
Thus, in every nonzero summand, the last argument of $f$ is $s\overline\gamma$ for some real arrow $\gamma\in Q_1$.
\end{Lem}

{\it Proof.}
Let $D$ be the graded $E_C$-derivation in Lemma~\ref{lem:clw-homotopy-data}{\rm(2)}.
If $v=e_i$ for some $i\in Q_0$, then $s\overline v=0$ in $s\overline {L(Q)}$, and the asserted vanishing follows.
For the monomial $v=\beta_1^*\cdots\beta_p^*\alpha_m\cdots\alpha_1$, repeated application of the graded Leibniz identity to the $p+m$ factors of $v$ gives
$$
 D(v)=\sum_{t=1}^{p+m}\varepsilon_t
 a_t\tensor s1_k\tensor b_t,
$$
where $\varepsilon_t\in\{1,-1\}$, and, for each summand, there is a vertex $i_t\in Q_0$ such that $a_t\in L(Q)e_{i_t}$ and $b_t\in e_{i_t}L(Q)$.
The elements $a_t$ and $b_t$ are monomials and hence homogeneous.
By the definition of $\pi$,
$\pi(1\tensor_{E_C}s\overline v\tensor_{E_C}1)=D(v)$.
Applying $\iota$ gives
$$
 \iota D(v)
 =-\sum_{t=1}^{p+m}\varepsilon_t
 \sum_{\substack{\gamma\in Q_1\\s(\gamma)=i_t}}
 a_t\gamma^*\tensor_{E_C}s\overline\gamma\tensor_{E_C}b_t.
$$
Applying $((s\circ q)\tensor{\rm id}\tensor{\rm id})$ to the last equality gives
$$
 \overline{\iota\pi}
 (1\tensor_{E_C}s\overline v\tensor_{E_C}1)
 =-\sum_{t=1}^{p+m}\varepsilon_t
 \sum_{\substack{\gamma\in Q_1\\s(\gamma)=i_t}}
 s\overline{a_t\gamma^*}\tensor_{E_C}s\overline\gamma
 \tensor_{E_C}b_t.
$$
The formula for $H$ in Lemma~\ref{lem:clw-homotopy-data}{\rm(2)} now gives the asserted formula, with the corresponding Koszul signs $\varepsilon_{t,\gamma}$. $\square$

\begin{Lem}\label{lem:continuing-term}
Let $r\geq1$, let $n\in\mathbb Z$, and let $g\in\operatorname{im}(H)\cap\Cbar_{E_C}^{n,r}(L(Q))$ be homogeneous.
Let $x_1,\ldots,x_{r-1}\in s\overline {L(Q)}$ be homogeneous, and let $\alpha\in Q_1$.
There is a Koszul sign $\varepsilon\in\{1,-1\}$ such that
$$
 (\delta_\mu g)(x_1\tensor_{E_C}\cdots\tensor_{E_C}x_{r-1}
                 \tensor_{E_C}s\overline\alpha)
 =\varepsilon\,g(x_1\tensor_{E_C}\cdots\tensor_{E_C}x_{r-1}
                 \tensor_{E_C}s\overline{(-d_\mu(\alpha))}).
$$
If $r\geq2$ and $p=\beta_2\beta_1\in Q_2$, then, for homogeneous $z_1,\ldots,z_{r-2}\in s\overline {L(Q)}$, there are Koszul signs $\varepsilon',\varepsilon''\in\{1,-1\}$ such that
$$
\begin{aligned}
 &(H\delta_\mu g)
 (z_1\tensor_{E_C}\cdots\tensor_{E_C}z_{r-2}
 \tensor_{E_C}s\overline p)\\
 &\qquad=\varepsilon'
 (\delta_\mu g)(z_1\tensor_{E_C}\cdots\tensor_{E_C}z_{r-2}
 \tensor_{E_C}s\overline{\beta_2}
 \tensor_{E_C}s\overline{\beta_1})\\
 &\qquad=\varepsilon''
 g(z_1\tensor_{E_C}\cdots\tensor_{E_C}z_{r-2}
 \tensor_{E_C}s\overline{\beta_2}
 \tensor_{E_C}s\overline{(-d_\mu(\beta_1))}).
\end{aligned}
$$
\end{Lem}

{\it Proof.}
Since $g\in\operatorname{im}(H)$, Lemma~\ref{lem:clw-homotopy-data}{\rm(2)} gives
$$
 g(y_1\tensor_{E_C}\cdots\tensor_{E_C}y_{r-1}
   \tensor_{E_C}s\overline\alpha)=0
$$
for all homogeneous $y_1,\ldots,y_{r-1}\in s\overline {L(Q)}$.
Put $\delta:=-d_\mu$, and let $\delta_s$ be the induced differential on $s\overline {L(Q)}$.
Expanding the change in the Hochschild differential caused by $\delta$ gives
$$
\begin{aligned}
&(\delta_\mu g)(x_1\tensor_{E_C}\cdots\tensor_{E_C}x_{r-1}
 \tensor_{E_C}s\overline\alpha)\\
&\quad=\delta\bigl(g(x_1\tensor_{E_C}\cdots\tensor_{E_C}x_{r-1}
 \tensor_{E_C}s\overline\alpha)\bigr)\\
&\qquad-(-1)^{|g|}\sum_{i=1}^{r-1}\varepsilon_i
 g(x_1\tensor_{E_C}\cdots\tensor_{E_C}\delta_s(x_i)
 \tensor_{E_C}\cdots\tensor_{E_C}x_{r-1}
 \tensor_{E_C}s\overline\alpha)\\
&\qquad-(-1)^{|g|}\varepsilon_r
 g(x_1\tensor_{E_C}\cdots\tensor_{E_C}x_{r-1}
 \tensor_{E_C}\delta_s(s\overline\alpha)),
\end{aligned}
$$
where the $\varepsilon_i$ are the Koszul signs.
The first term is zero, and every summand in the sum is zero by
$$
g(y_1\tensor_{E_C}\cdots\tensor_{E_C}y_{r-1}
\tensor_{E_C}s\overline\alpha)=0.
$$
Since $\delta_s(s\overline\alpha)=\pm s\overline{(-d_\mu(\alpha))}$, the last term is the only term which can be nonzero.
Absorbing its sign proves the first assertion.

Suppose that $r\geq2$ and let $p=\beta_2\beta_1\in Q_2$.
By the definition of $\pi$ and the formula for $D$ in Lemma~\ref{lem:clw-homotopy-data}{\rm(2)},
$\pi(1\tensor_{E_C}s\overline p\tensor_{E_C}1)=D(p)$,
where
$$
 D(p)
 =-\beta_2\tensor s1_k\tensor\beta_1
  +p\tensor s1_k\tensor e_{s(\beta_1)}.
$$
Applying $\iota$ gives
$$
\begin{aligned}
 \iota D(p)
 =\sum_{s(\gamma)=s(\beta_2)}
 \beta_2\gamma^*\tensor_{E_C}s\overline\gamma\tensor_{E_C}\beta_1
 -\sum_{s(\gamma)=s(\beta_1)}
 p\gamma^*\tensor_{E_C}s\overline\gamma
   \tensor_{E_C}e_{s(\beta_1)}.
\end{aligned}
$$
Using $p\gamma^*=\beta_2(\beta_1\gamma^*)$, the relations $\beta_2\gamma^*=\delta_{\beta_2,\gamma}e_{t(\beta_2)}$ and $\beta_1\gamma^*=\delta_{\beta_1,\gamma}e_{t(\beta_1)}$ reduce this to
$$
 e_{t(\beta_2)}\tensor_{E_C}s\overline{\beta_2}\tensor_{E_C}\beta_1
 -\beta_2\tensor_{E_C}s\overline{\beta_1}\tensor_{E_C}e_{s(\beta_1)}.
$$
Applying $((s\circ q)\tensor{\rm id}\tensor{\rm id})$ annihilates the first summand because $\overline{e_{t(\beta_2)}}=0$.
Therefore, for a sign $\varepsilon_0\in\{1,-1\}$,
$$
 \overline{\iota\pi}
 (1\tensor_{E_C}s\overline p\tensor_{E_C}1)
 =\varepsilon_0s\overline{\beta_2}\tensor_{E_C}
  s\overline{\beta_1}\tensor_{E_C}e_{s(\beta_1)}.
$$
Substitution into the formula for $H(\delta_\mu g)$ gives
$$
 (H\delta_\mu g)(z_1\tensor_{E_C}\cdots\tensor_{E_C}z_{r-2}
 \tensor_{E_C}s\overline p)
 =\varepsilon'
 (\delta_\mu g)(z_1\tensor_{E_C}\cdots\tensor_{E_C}z_{r-2}
 \tensor_{E_C}s\overline{\beta_2}
 \tensor_{E_C}s\overline{\beta_1}).
$$
The first assertion, applied with $\alpha=\beta_1$, gives the final equality in the lemma. $\square$

\medskip
{\it Proof of Lemma~\ref{lem:uniform-nilpotence}.}
If $N=1$, then $J_C=0$, so $\mu=0$, $d_\mu=0$ and $\delta_\mu=0$, and the assertion follows.
Suppose that $N\geq2$, and put $T:=H\delta_\mu$.
The image of $T$ is contained in $\operatorname{im}(H)$.
It is therefore enough to prove
\begin{equation}\label{eq:T-on-image-H}
 T^N(w)=0\qquad\bigl(w\in\operatorname{im}(H)\bigr),
\end{equation}
because then $T^{N+1}(f)=T^N(T(f))=0$ for every cochain $f$.

We first prove \eqref{eq:T-on-image-H} on each tensor degree.
Fix $m\geq1$ and $n\in\mathbb Z$, and let $w\in\operatorname{im}(H)\cap\Cbar_{E_C}^{n,m+1}(L(Q))$.
Let $x_1,\ldots,x_{m-1}\in s\overline {L(Q)}$ be homogeneous, and let $v$ be a nonzero monomial in $L(Q)$.
If $v=e_i$ for some $i\in Q_0$, then $s\overline v=0$.
Otherwise, write
$$
 v=\beta_1^*\cdots\beta_u^*\alpha_q\cdots\alpha_1,
 \qquad u+q\geq1.
$$
Applying Lemma~\ref{lem:final-real-exposure} to $\delta_\mu w$ and then the first assertion of Lemma~\ref{lem:continuing-term} gives
\begin{align}
 &(Tw)
 (x_1\tensor_{E_C}\cdots\tensor_{E_C}x_{m-1}
 \tensor_{E_C}s\overline v)
 \nonumber\\
 &\qquad=
 \sum_{t=1}^{u+q}
 \sum_{\substack{\gamma\in Q_1\\s(\gamma)=i_t}}
 \varepsilon''_{t,\gamma}\,
 w(x_1\tensor_{E_C}\cdots\tensor_{E_C}x_{m-1}
 \tensor_{E_C}s\overline{a_t\gamma^*}
 \tensor_{E_C}s\overline{(-d_\mu(\gamma))})b_t,         \label{eq:one-nilpotence-step}
\end{align}
where $\varepsilon''_{t,\gamma}\in\{1,-1\}$ are the corresponding Koszul signs.
The monomials used above span $L(Q)$ by \cite[Section~4]{ChenLiWang2025}, and their images span $\overline {L(Q)}=L(Q)/E_C$.
We therefore expand every occurrence of $d_\mu$ and follow the resulting monomial summands before collecting equal terms.
These expansions are finite because $Q$ is finite.

For $r\geq1$, the cochain $T^{r-1}w$ belongs to $\operatorname{im}(H)$.
Applying Lemma~\ref{lem:continuing-term} with
$g=T^{r-1}w$ gives the following recursion, for every real arrow
$\alpha$ and every list $\mathbf z$ of homogeneous arguments of
the required length:
\begin{equation}\label{eq:arrow-decrease}
\begin{aligned}
 &(T^rw)(\mathbf z,s\overline{-d_\mu(\alpha)})\\
 &\quad=
 \sum_{\beta_2,\beta_1}
 \varepsilon_{\beta_2,\beta_1}
 \lambda_{\beta_2\beta_1,\alpha}
 (T^{r-1}w)
   (\mathbf z,s\overline{\beta_2},s\overline{-d_\mu(\beta_1)}),
\end{aligned}
\end{equation}
where the $\varepsilon_{\beta_2,\beta_1}$ are Koszul signs.
If $H$ is applied to a cochain of tensor degree zero or one,
\eqref{eq:H-low-tensor-degree} gives zero.
Every nonzero coefficient in \eqref{eq:arrow-decrease} satisfies
$\ell(\beta_1)<\ell(\alpha)$ by Lemma~\ref{lem:layer-inequality}.

To expand $T^r(w)$, first apply
\eqref{eq:one-nilpotence-step} with $w$ replaced by $T^{r-1}w$.
Each summand has final argument $s\overline{-d_\mu(\gamma)}$.
Apply \eqref{eq:arrow-decrease} successively $r-1$ times.
A nonzero summand must then contain arrows
$$
\eta_0=\gamma,\eta_1,\ldots,\eta_{r-1},
 \qquad
 \ell(\eta_0)>\ell(\eta_1)>\cdots>\ell(\eta_{r-1}).
$$
This is an expansion from the outermost copy of $T$ to the
innermost one. At each step the entire cochain $T^jw$ lies in
$\operatorname{im}(H)$, as required by
Lemma~\ref{lem:continuing-term}. No assertion about an individual
summand belonging to $\operatorname{im}(H)$ is used.

In \eqref{eq:one-nilpotence-step}, $s\overline{a_t\gamma^*}$ precedes $s\overline{(-d_\mu(\gamma))}$, while $b_t$ is multiplied on the right after evaluating $w$.
Hence neither factor changes the final real arrow to which Lemma~\ref{lem:continuing-term} is applied.

Every radical layer belongs to $\{1,\ldots,N-1\}$.
A nonzero summand of $T^N(w)$ would therefore determine $N-1$ strict decreases
$$
 \ell(\eta_0)>\ell(\eta_1)>\cdots>\ell(\eta_{N-1})
$$
among only $N-1$ possible integers.
Such a chain does not exist.
Consequently, every summand in the expansion of $T^N(w)$ vanishes before equal summands are collected; possible coincidences or cancellations among different summands therefore play no role.
Thus $T^N(w)=0$.
For tensor degrees zero and one, the same equality follows from \eqref{eq:H-low-tensor-degree}.
We have therefore proved $T^N(w_m)=0$ for every $m\geq0$ and every $w_m\in\operatorname{im}(H)\cap\Cbar_{E_C}^{n,m}(L(Q))$.

It remains to pass from the individual tensor degrees to
$$
 \Cbar_{E_C}^n(L(Q))
 \cong\prod_{m\geq0}\Cbar_{E_C}^{n,m}(L(Q)).
$$
The map $\delta_\mu$ preserves tensor degree and $H$ lowers it by one.
Consequently, if $f=(f_m)_{m\geq0}$, then $(Tf)_m=T(f_{m+1})$.
Put $Tf=w=(w_m)_{m\geq0}$.
Then $w_m\in\operatorname{im}(H)\cap\Cbar_{E_C}^{n,m}(L(Q))$.
For every $m\geq0$, iterating $(Tf)_m=T(f_{m+1})$ and using \eqref{eq:T-on-image-H} give
$$
 \bigl(T^{N+1}(f)\bigr)_m
 =\bigl(T^N(w)\bigr)_m=T^N(w_{m+N})=0.
$$
For fixed $m$, the formula involves only $w_{m+N}$; hence no interchange of infinite sums is involved.
Hence $T^{N+1}(f)=0$.
Since $n$ and $f$ were arbitrary, $(H\delta_\mu)^{N+1}=0$ on $\Cbar_{E_C}^*(L(Q))$. $\square$

Lemma~\ref{lem:uniform-nilpotence} gives
$$
 ({\rm id}+H\delta_\mu)^{-1}
   =\sum_{r=0}^{N}(-H\delta_\mu)^r.
$$
To prove that the first Taylor component of the twisted $B_\infty$-morphism in \eqref{eq:twisted-G} is a quasi-isomorphism, we use the following lemma.

\begin{Lem}\label{lem:quotient-perturbation}
Let $(X,d_X)$ and $(Y,d_Y)$ be complexes.
Suppose that there are chain maps and a degree-$-1$ map
$$
\xymatrix{
(X,d_X)\ar@<0.55ex>[r]^-i&
(Y,d_Y)\ar@<0.55ex>[l]^-p
}
\qquad
h\in\Hom_k^{-1}(Y,Y)
$$
such that
$$
 p i={\rm id}_X,\qquad {\rm id}_Y-i p=d_Yh+hd_Y,
 \qquad hi=0,\qquad ph=0.
$$
Let $\epsilon:X\to X$ and $\delta:Y\to Y$ have degree one and satisfy
$$
              (d_X+\epsilon)^2=0,\qquad
              (d_Y+\delta)^2=0.
$$
Suppose that $i:(X,d_X+\epsilon)\to(Y,d_Y+\delta)$ is a chain map and that $h\delta$ is nilpotent.
Then $i:(X,d_X+\epsilon)\to(Y,d_Y+\delta)$ is a quasi-isomorphism.
\end{Lem}

{\it Proof.}
It is enough to prove that $(Y,d_Y+\delta)/i(X)$ is contractible.
Put $Z:=Y/i(X)$, $K:=\ker(p)$ and $q:={\rm id}_Y-ip$.
Since $i$ is a chain map for $d_X$ and $d_Y$, the differential $d_Y$ induces a differential $\overline d_Y$ on $Z$.
Since $p$ is a chain map, $d_Y$ preserves $K$; put $d_K:=d_Y|_K$.
The equality $pi={\rm id}_X$ gives a graded direct sum $Y=i(X)\oplus K$.
Moreover, $q(Y)\subseteq K$, $q|_K={\rm id}_K$, and $\ker(q)=i(X)$.
Thus $q$ induces an isomorphism of complexes
$$
 \widetilde q:(Z,\overline d_Y)\xra{\ \sim\ }(K,d_K),
 \qquad \widetilde q(y+i(X))=q(y).
$$
The hypotheses give $h(K)\subseteq K$ and ${\rm id}_K=d_Kh+hd_K$.

Since $i$ is a chain map both before and after adding $\epsilon$ and $\delta$, one has $\delta i=i\epsilon$.
Hence $\delta$ induces a degree-one map $\overline\delta$ on $Z$, and the quotient of $(Y,d_Y+\delta)$ by $i(X)$ is
$$
 (Z,\overline d_Y+\overline\delta),
 \qquad
 \overline\delta(y+i(X))=\delta(y)+i(X).
$$
Under the isomorphism $\widetilde q$, the quotient differential becomes
$$
 \widetilde q(\overline d_Y+\overline\delta)\widetilde q^{-1}
      =d_K+\delta_K,
 \qquad \delta_K:=q\delta|_K.
$$
Consequently, $d_K+\delta_K$ is a differential on $K$, although $d_Y+\delta$ need not preserve $K$.

We now construct a contracting homotopy for $(K,d_K+\delta_K)$.
Put $\Theta:={\rm id}_K+\delta_Kh$.
We first prove that $\Theta$ is invertible.
Since $hi=0$, one has $hq=h({\rm id}_Y-ip)=h-hip=h$.
The definition $\delta_K=q\delta|_K$ and repeated use of $hq=h$ give $(\delta_Kh)^r=q\delta(h\delta)^{r-1}h|_K$ for $r\geq1$.
Choose $N\geq1$ such that $(h\delta)^N=0$.
The formula for $(\delta_Kh)^r$ gives $(\delta_Kh)^{N+1}=0$.
Consequently, $\Theta$ is invertible, with $\Theta^{-1}=\sum_{r=0}^{N}(-\delta_Kh)^r$.
Define $\overline h:=h\Theta^{-1}$.
We claim that $\overline h$ is a contracting homotopy.
For the verification, first establish $(d_K+\delta_K)\Theta=\Theta d_K$.
Since $(d_K+\delta_K)^2=0$, one has $d_K\delta_K+\delta_Kd_K+\delta_K^2=0$.
Using $d_K\delta_K+\delta_Kd_K+\delta_K^2=0$ and ${\rm id}_K=d_Kh+hd_K$, we compute
$$
\begin{aligned}
 &(d_K+\delta_K)\Theta-\Theta d_K\\
 &=(d_K+\delta_K)({\rm id}_K+\delta_Kh)
      -({\rm id}_K+\delta_Kh)d_K\\
 &=\delta_K+d_K\delta_Kh+\delta_K^2h-\delta_Khd_K\\
 &=\delta_K-\delta_Kd_Kh-\delta_Khd_K\\
 &=\delta_K({\rm id}_K-d_Kh-hd_K)=0.
\end{aligned}
$$
Thus $(d_K+\delta_K)\Theta=\Theta d_K$.
Since $\Theta$ is invertible, $\Theta^{-1}(d_K+\delta_K)=d_K\Theta^{-1}$.
Using $\overline h=h\Theta^{-1}$ and ${\rm id}_K=d_Kh+hd_K$, we now compute
$$
\begin{aligned}
 (d_K+\delta_K)\overline h
   +\overline h(d_K+\delta_K)
 &=(d_K+\delta_K)h\Theta^{-1}
   +h\Theta^{-1}(d_K+\delta_K)\\
 &=(d_K+\delta_K)h\Theta^{-1}+hd_K\Theta^{-1}\\
 &=\bigl((d_K+\delta_K)h+hd_K\bigr)\Theta^{-1}\\
 &=\bigl(d_Kh+hd_K+\delta_Kh\bigr)\Theta^{-1}\\
 &=({\rm id}_K+\delta_Kh)\Theta^{-1}
   =\Theta\Theta^{-1}={\rm id}_K.
\end{aligned}
$$
Thus $\overline h$ is a contracting homotopy of $(K,d_K+\delta_K)$.
Since $\widetilde q$ identifies $(K,d_K+\delta_K)$ with $(Z,\overline d_Y+\overline\delta)$, the quotient complex is contractible, and $i$ is a quasi-isomorphism. $\square$

\begin{Lem}\label{lem:twisted-quasi-isomorphism}
Let $f=(f_1,f_2,\ldots):X\to Y$ be a $B_\infty$-morphism and let $a\in X$ be a Maurer--Cartan element.
Put $b:=f_1(a)$.
Suppose that the underlying complexes $(X,d_X)$ and $(Y,d_Y)$ admit chain maps
$$
\xymatrix{
(X,d_X)\ar@<0.55ex>[r]^-{f_1}&
(Y,d_Y)\ar@<0.55ex>[l]^-p
}
\qquad
h\in\Hom_k^{-1}(Y,Y)
$$
such that
$$
 pf_1={\rm id}_X,\qquad {\rm id}_Y-f_1p=d_Yh+hd_Y,
 \qquad hf_1=0,\qquad ph=0.
$$
Using Definition~\ref{def:twist}, write the underlying complexes of the twists as $(X,d_X+\epsilon_a)$ and $(Y,d_Y+\epsilon_b)$, respectively.
If $h\epsilon_b$ is nilpotent, then the induced morphism
$$
 \operatorname{Tw}_a(X)\longrightarrow\operatorname{Tw}_b(Y)
$$
is a $B_\infty$-quasi-isomorphism.
\end{Lem}

{\it Proof.}
Let $F:T^c(sX)\to T^c(sY)$ represent $f$.
Lemma~\ref{lem:primitive-twist} gives $F(D_X+\operatorname{ad}_{sa})=(D_Y+\operatorname{ad}_{sb})F$.
Taking the first Taylor component gives $sf_1(d_X+\epsilon_a)=s(d_Y+\epsilon_b)f_1$, so $f_1:(X,d_X+\epsilon_a)\to(Y,d_Y+\epsilon_b)$ is a chain map.
Lemma~\ref{lem:quotient-perturbation}, applied with $i:=f_1$, $\epsilon:=\epsilon_a$ and $\delta:=\epsilon_b$, proves that $f_1$ is a quasi-isomorphism.
Hence the twisted morphism is a $B_\infty$-quasi-isomorphism. $\square$

\medskip
{\it Proof of Proposition~\ref{prop:core-isomorphism}.}
Recall that $G^{\opp}=(\Phi_1,\Phi_2,\ldots)^{\opp}\rho^{\opp}(\kappa^{\opp})^{-1}$.
Put $a:=\rho^{\opp}(\kappa^{\opp})^{-1}(-\mu)$.
The morphism in \eqref{eq:raw-twisted-G} has the factorization displayed after \eqref{eq:twisted-G}.
The opposite of a strict $B_\infty$-isomorphism is strict, and its inverse is strict.
Moreover, Lemma~\ref{lem:primitive-twist} applies the same coalgebra map before and after twisting.
Hence a strict $B_\infty$-isomorphism $f:V\to W$ induces a strict $B_\infty$-isomorphism $\operatorname{Tw}_c(V)\to\operatorname{Tw}_{f_1(c)}(W)$ whose inverse is induced by $f^{-1}$.
Consequently, the twisted factors induced by $(\kappa^{\opp})^{-1}$ and $\rho^{\opp}$ are strict $B_\infty$-isomorphisms.
It remains to prove that the twisted morphism induced by $(\Phi_1,\Phi_2,\ldots)^{\opp}$ is a $B_\infty$-quasi-isomorphism.
Taking $B_\infty$ opposites leaves the first Taylor component and the differentials of the underlying complexes unchanged.
By Definition~\ref{def:B-infinity-morphism}, it is enough to prove that $\Phi_1$ is a quasi-isomorphism between the underlying complexes after twisting.

In Lemma~\ref{lem:twisted-quasi-isomorphism}, take $X:=\Chat^*(L(Q))^{\opp}$, $Y:=\Cbar_{E_C}^*(L(Q))$, $f:=(\Phi_1,\Phi_2,\ldots)^{\opp}$, $f_1:=\Phi_1$, $p:=\Psi_1$ and $h:=H$.
We first verify the four equalities involving $f_1,p$ and $h$.
Lemma~\ref{lem:clw-homotopy-data}{\rm(1)} gives
$$
 \Psi_1\Phi_1={\rm id}_{\Chat^*(L(Q))^{\opp}},
 \qquad
 {\rm id}_{\Cbar_{E_C}^*(L(Q))}-\Phi_1\Psi_1=d_0H+Hd_0.
$$
The same lemma gives $\im(\Phi_1)\subseteq \Cbar_{E_C}^{*,0}(L(Q))\oplus\Cbar_{E_C}^{*,1}(L(Q))$, so \eqref{eq:H-low-tensor-degree} yields $H\Phi_1=0$.

It remains to prove $\Psi_1H=0$.
Lemma~\ref{lem:clw-homotopy-data}{\rm(1)} gives $\Psi_1H(\varphi)=0$ when $H(\varphi)$ has tensor degree at least two.
If $H(\varphi)$ has tensor degree one, the formula for $\Psi_1$ gives
$$
 \Psi_1H(\varphi)
 =-\sum_{\alpha\in Q_1}s^{-1}\alpha^*H(\varphi)(s\overline\alpha)=0,
$$
because $H(\varphi)(s\overline\alpha)=0$ for every $\alpha\in Q_1$.
Finally, $H(\varphi)=0$ when $\varphi$ has tensor degree zero or one.
Hence $\Psi_1H=0$.

The Maurer--Cartan element $a$ satisfies $b:=f_1(a)=-d_\mu$.
The equality $b=-d_\mu$ follows from $d_\mu=\Phi_1\rho\kappa^{-1}(\mu)$ and from the fact that taking opposites does not change the first Taylor components of $\kappa$ and $\rho$.
Write the differentials of the two twisted complexes as $d_{\Chat}+\epsilon_a$ and $d_0+\epsilon_b$.
In the target complex, $\epsilon_b=\delta_\mu=[-d_\mu,-]$.
Finally, Lemma~\ref{lem:uniform-nilpotence} gives $(H\delta_\mu)^{N+1}=0$.
Thus all the hypotheses of Lemma~\ref{lem:twisted-quasi-isomorphism} hold.
The twisted morphism induced by $(\Phi_1,\Phi_2,\ldots)^{\opp}$ is therefore a $B_\infty$-quasi-isomorphism.
Hence the composite morphism in \eqref{eq:twisted-G} is a $B_\infty$-quasi-isomorphism and becomes an isomorphism in $\Ho_k(\Binf)$. $\square$

\medskip
The proof of Proposition~\ref{prop:core-isomorphism} constructs the required $B_\infty$-quasi-isomorphism.
We now identify the dg algebra occurring in its target with the dg Leavitt algebra of Chen--Wang.

Here $L(Q)$ carries the Chen--Li--Wang grading fixed in Definition~\ref{def:radical-quiver-leavitt}, whereas $L_{\mathrm{CW}}(Q)$ carries the Chen--Wang grading fixed in the preliminaries.

To identify the target with a model of the dg singularity category,
we compare the differential $d_\mu$ with the Chen--Wang differential.
\begin{Lem}\label{lem:signed-star}
There is a degree-zero graded-algebra isomorphism
$$
\omega_Q:L(Q)^{\op}\longrightarrow L_{\mathrm{CW}}(Q),\qquad \omega_Q(e_i):=e_i,\qquad \omega_Q(\alpha):=-\alpha^*,\qquad \omega_Q(\alpha^*):=\alpha.
$$
\end{Lem}

{\it Proof.}
Recall that the multiplication in $L(Q)^{\op}$ is $x\cdot_{\op}y=(-1)^{|x||y|}yx$ for homogeneous $x,y\in L(Q)$.
Hence the assignments in Lemma~\ref{lem:signed-star} define a homomorphism from $L(Q)^{\op}$ precisely when $\omega_Q(xy)=(-1)^{|x||y|}\omega_Q(y)\omega_Q(x)$.
We verify this by checking the defining relations of $L(Q)$.
For every $\alpha\in Q_1$, one has $|\omega_Q(\alpha)|_{\mathrm{CW}}=|-\alpha^*|_{\mathrm{CW}}=1=|\alpha|$ and $|\omega_Q(\alpha^*)|_{\mathrm{CW}}=|\alpha|_{\mathrm{CW}}=-1=|\alpha^*|$.
Thus $\omega_Q$ preserves degrees.
It fixes the vertex relations.
For a real arrow $\alpha$, the source and target relations are preserved because
$$
\begin{aligned}
 \omega_Q(e_{t(\alpha)}\alpha)
 &=\omega_Q(\alpha)\omega_Q(e_{t(\alpha)})
   =-\alpha^*e_{t(\alpha)}=-\alpha^*=\omega_Q(\alpha),\\
 \omega_Q(\alpha e_{s(\alpha)})
 &=\omega_Q(e_{s(\alpha)})\omega_Q(\alpha)
   =-e_{s(\alpha)}\alpha^*=-\alpha^*=\omega_Q(\alpha).
\end{aligned}
$$
The same calculation with $\alpha^*$ in place of $\alpha$ verifies the source and target relations for the ghost arrow.
Since every real or ghost arrow has odd degree, reversing two such generators in the graded opposite contributes the sign $-1$.
Thus, for arrows $\alpha$ and $\beta$ with the same source, the multiplication rule gives
$$
 \omega_Q(\alpha\beta^*)
 =-\omega_Q(\beta^*)\omega_Q(\alpha)
 =\beta\alpha^*
 =\delta_{\alpha,\beta}e_{t(\alpha)}
$$
and, for every vertex $i$,
$$
\begin{aligned}
 \omega_Q\left(\sum_{s(\alpha)=i}\alpha^*\alpha\right)
 &=-\sum_{s(\alpha)=i}
     \omega_Q(\alpha)\omega_Q(\alpha^*)\\
 &=\sum_{s(\alpha)=i}\alpha^*\alpha
 =e_i=\omega_Q(e_i).
\end{aligned}
$$
Hence $\omega_Q$ preserves the two families of relations displayed in Definition~\ref{def:radical-quiver-leavitt} and defines a graded-algebra homomorphism.
The assignments $e_i\mapsto e_i$, $\alpha\mapsto\alpha^*$ and $\alpha^*\mapsto-\alpha$ preserve the same relations and define a homomorphism in the opposite direction.
The two composites fix every vertex, real arrow and ghost arrow, and hence are identity maps.
Thus $\omega_Q$ is a graded-algebra isomorphism. $\square$

Let $\zeta_{\mathrm{CW}}:L_{\mathrm{CW}}(Q)\to L_{\mathrm{CW}}(Q)$ be the automorphism defined by $\zeta_{\mathrm{CW}}(x):=(-1)^{|x|_{\mathrm{CW}}}x$ for homogeneous $x\in L_{\mathrm{CW}}(Q)$, and put $\omega_Q^-:=\zeta_{\mathrm{CW}}\omega_Q$.
Thus $\omega_Q^-(\alpha)=\alpha^*$ and $\omega_Q^-(\alpha^*)=-\alpha$.

\begin{Lem}\label{lem:dmu-ghost}
For every $\alpha\in Q_1$, one has
$$
\begin{aligned}
 d_\mu(\alpha^*)
 &=\sum_{\substack{\beta,\gamma\in Q_1\\
                   \beta\alpha\in Q_2\\
                   \gamma\parallel\beta\alpha}}
      \lambda_{\beta\alpha,\gamma}\gamma^*\beta.
\end{aligned}
$$
\end{Lem}

{\it Proof.}
Fix $\alpha\in Q_1$ and put $i:=s(\alpha)$.
For every $\eta\in Q_1$ with $s(\eta)=i$, differentiate $\eta\alpha^*=\delta_{\eta,\alpha}e_{t(\alpha)}$.
Since $|\eta|=1$ and $d_\mu(e_j)=0$, we obtain $\eta d_\mu(\alpha^*)=d_\mu(\eta)\alpha^*$.
Multiplying on the left by $\eta^*$ and summing over all arrows starting at $i$, we obtain
$$d_\mu(\alpha^*)=\sum_{s(\eta)=i}\eta^*d_\mu(\eta)\alpha^*=\sum_{\substack{\eta,\beta_2,\beta_1\in Q_1\\ \eta\parallel\beta_2\beta_1}}\lambda_{\beta_2\beta_1,\eta}\eta^*\beta_2\beta_1\alpha^*.$$
Here the first equality follows from $e_i=\sum_{s(\eta)=i}\eta^*\eta$.
In the last sum, $s(\beta_1)=i=s(\alpha)$, and hence $\beta_1\alpha^*=\delta_{\beta_1,\alpha}e_{t(\alpha)}$.
Only the terms with $\beta_1=\alpha$ remain.
Renaming $\eta$ as $\gamma$ and $\beta_2$ as $\beta$ proves the formula in Lemma~\ref{lem:dmu-ghost}. $\square$

The next proposition first identifies the differential corresponding to $d_\mu$ and then obtains the differential corresponding to $-d_\mu$ by applying the automorphism $\zeta_{\mathrm{CW}}$.

\begin{Prop}\label{prop:differential-transport}
The graded-algebra isomorphism $\omega_Q$ satisfies $\omega_Qd_\mu=\partial_\nu\omega_Q$ and hence is a dg-algebra isomorphism $\omega_Q:(L(Q),d_\mu)^{\op}\to(L_{\mathrm{CW}}(Q),\partial_\nu)$.
The automorphism $\zeta_{\mathrm{CW}}$ satisfies $\zeta_{\mathrm{CW}}\partial_\nu=-\partial_\nu\zeta_{\mathrm{CW}}$.
Consequently, $\omega_Q^-=\zeta_{\mathrm{CW}}\omega_Q$ satisfies $\omega_Q^-(-d_\mu)=\partial_\nu\omega_Q^-$ and is a dg-algebra isomorphism $\omega_Q^-:(L(Q),-d_\mu)^{\op}\to(L_{\mathrm{CW}}(Q),\partial_\nu)$.
\end{Prop}

{\it Proof.}
Lemma~\ref{lem:signed-star} shows that $\omega_Q$ is a graded-algebra isomorphism.
We first prove $\omega_Qd_\mu=\partial_\nu\omega_Q$ by checking the equality on the vertices and on the real and ghost arrows.
For $\beta_2\beta_1\in Q_2$, the definition of $\omega_Q$ and the multiplication in the graded opposite algebra give $\omega_Q(\beta_2\beta_1)=-\beta_1^*\beta_2^*$ and $\omega_Q(\gamma^*\beta)=\beta^*\gamma$.
It follows from Lemmas~\ref{lem:dmu-real} and~\ref{lem:dmu-ghost} and the formulas for $\partial_\nu$ in Lemma~\ref{lem:chen-wang}{\rm(1)} that
$$
\begin{aligned}
 \omega_Qd_\mu(\alpha)
 &=-\sum_{\substack{\beta_2\beta_1\in Q_2\\
                     \alpha\parallel\beta_2\beta_1}}
      \lambda_{\beta_2\beta_1,\alpha}\beta_1^*\beta_2^*
   =\partial_\nu(-\alpha^*)=\partial_\nu\omega_Q(\alpha),\\
 \omega_Qd_\mu(\alpha^*)
 &=\sum_{\beta,\gamma}
      \lambda_{\beta\alpha,\gamma}\beta^*\gamma
   =\partial_\nu(\alpha)=\partial_\nu\omega_Q(\alpha^*).
\end{aligned}
$$
Both maps vanish on the vertices.
Since $L(Q)$ is generated by its vertices and its real and ghost arrows, the derivation rules give $\omega_Qd_\mu=\partial_\nu\omega_Q$.
Thus $\omega_Q$ is the first asserted dg-algebra isomorphism.
For homogeneous $y\in L_{\mathrm{CW}}(Q)$, the differential $\partial_\nu$ has degree one, and therefore $\zeta_{\mathrm{CW}}(\partial_\nu y)=(-1)^{|y|_{\mathrm{CW}}+1}\partial_\nu y=-\partial_\nu((-1)^{|y|_{\mathrm{CW}}}y)=-\partial_\nu(\zeta_{\mathrm{CW}} y)$.
Since $\omega_Q^-=\zeta_{\mathrm{CW}}\omega_Q$, we obtain $\omega_Q^-(-d_\mu)=\zeta_{\mathrm{CW}}\omega_Q(-d_\mu)=-\zeta_{\mathrm{CW}}\partial_\nu\omega_Q=\partial_\nu\zeta_{\mathrm{CW}}\omega_Q=\partial_\nu\omega_Q^-$.
Since $\zeta_{\mathrm{CW}}$ is a graded-algebra automorphism, $\omega_Q^-$ is the second asserted dg-algebra isomorphism. $\square$

By \cite[Lemma~8.12]{ChenLiWang2025}, choose an isomorphism
$$
 \sigma_{C,E_C}:\Cbar_{\sg,R}^*(C)\longrightarrow
 \Cbar_{\sg,R,E_C}^*(C)
 \quad\text{in }\Ho_k(\Binf).
$$
For any such choice, the following diagram connects the singular
Hochschild cochain complex of $C^{\op}$ with the Hochschild
cochain complex of $(L(Q),-d_\mu)$.
All arrows are read in $\Ho_k(\Binf)$; the dotted arrow $\Upsilon_C$ denotes their composite, with the indicated inverses.
The two vertical equality signs denote equalities of $B_\infty$-algebra structures under the identifications of their underlying graded vector spaces specified in Proposition~\ref{prop:source-twist}, Lemma~\ref{lem:opposite-twist} and Lemma~\ref{lem:target-twist}.
\begin{equation}\label{eq:singular-leavitt-morphisms}
\xymatrix@C=1.5pc@R=1.4pc{
\Cbar_{\sg,L}^*(C^{\op})
 \ar@{.>}[dddd]_-{\Upsilon_C}
 \ar[rr]^-{\text{Lemma~\ref{lem:left-right-singular}}}
&&\Cbar_{\sg,R}^*(C)^{\opp}
&&\Cbar_{\sg,R,E_C}^*(C)^{\opp}
 \ar[ll]_-{(\sigma_{C,E_C}^{\opp})^{-1}}
 \ar@{=}[d]_-{\substack{\text{Proposition~\ref{prop:source-twist}}\\
                         \text{Lemma~\ref{lem:opposite-twist}}}}
\\
&&&&
\operatorname{Tw}_{-\mu}\bigl(
 \Cbar_{\sg,R,E_C}^*(\widetilde C)^{\opp}\bigr)
\\
&&&&
\operatorname{Tw}_{(\kappa^{\opp})^{-1}(-\mu)}
 \bigl(\Cbar_{\sg,R}^*(Q)^{\opp}\bigr)
 \ar[u]_-{\kappa^{\opp}}
 \ar[d]^-{\rho^{\opp}}
\\
&&\operatorname{Tw}_{-d_\mu}\bigl(\Cbar_{E_C}^*(L(Q))\bigr)
 \ar@{=}[d]_-{\text{Lemma~\ref{lem:target-twist}}}
&&\operatorname{Tw}_{a}\bigl(\Chat^*(L(Q))^{\opp}\bigr)
 \ar[ll]_-{(\Phi_1,\Phi_2,\ldots)^{\opp}}
\\
C^*((L(Q),-d_\mu))
&&\Cbar_{E_C}^*((L(Q),-d_\mu))
 \ar[ll]_-{\text{Lemma~\ref{lem:relative-hochschild-inclusion}}}
&&
}
\end{equation}
The top right arrow is the inverse in $\Ho_k(\Binf)$ of the
full opposite of the chosen isomorphism $\sigma_{C,E_C}$.
Here $a=\rho^{\opp}(\kappa^{\opp})^{-1}(-\mu)$.
Lemma~\ref{lem:left-right-singular} and \cite[Lemma~8.12]{ChenLiWang2025} show that the two solid arrows in the first row represent isomorphisms in $\Ho_k(\Binf)$.
Proposition~\ref{prop:source-twist} and Lemma~\ref{lem:opposite-twist} give the first vertical identification.
Proposition~\ref{prop:core-isomorphism} shows that the morphism induced by $(\kappa^{\opp})^{-1}$, $\rho^{\opp}$ and $(\Phi_1,\Phi_2,\ldots)^{\opp}$ after twisting is an isomorphism in $\Ho_k(\Binf)$.
Lemma~\ref{lem:target-twist} gives the second vertical identification.
Finally, $E_C$ is semisimple and $d_\mu$ vanishes on the vertex idempotents, so Lemma~\ref{lem:relative-hochschild-inclusion} shows that the bottom solid arrow represents an isomorphism in $\Ho_k(\Binf)$.
Consequently, the solid zigzag represents an isomorphism
$$
 \Upsilon_C:\Cbar_{\sg,L}^*(C^{\op})
 \longrightarrow C^*((L(Q),-d_\mu))
$$
in $\Ho_k(\Binf)$; this is the isomorphism denoted by the dotted arrow.

The dg algebra isomorphism
$$
 \zeta:(L(Q),-d_\mu)\longrightarrow(L(Q),d_\mu),\qquad
 \zeta(x):=(-1)^{|x|}x\quad(x\in L(Q)\text{ homogeneous})
$$
induces a strict $B_\infty$-isomorphism
$\zeta_*:C^*((L(Q),-d_\mu))\to C^*((L(Q),d_\mu))$ by conjugation.

\begin{Koro}\label{core:existence}
Let $C$ be a basic split finite-dimensional algebra whose radical
quiver has no sinks. There is an isomorphism
$$
 \Cbar_{\sg,L}^*(C^{\op})\simeq C^*(\Sdg(C))
 \qquad\text{in }\Ho_k(\Binf).
$$
\end{Koro}
{\it Proof.}
The map $\Upsilon_C$ in \eqref{eq:singular-leavitt-morphisms}, followed
by $\zeta_*$, gives an isomorphism with $C^*((L(Q),d_\mu))$.
The remaining isomorphisms are
$$
\begin{aligned}
 C^*((L(Q),d_\mu))
 &\xleftarrow{\ \sim\ }
 C^*\bigl(\perdg((L(Q),d_\mu)^{\op})\bigr)\\
 &\xrightarrow{\ \sim\ }
 C^*\bigl(\perdg(L_{\mathrm{CW}}(Q),\partial_\nu)\bigr)
 \xrightarrow{\ \sim\ } C^*(\Sdg(C)).
\end{aligned}
$$
The first isomorphism is given by Lemma~\ref{lem:perfect-restriction}.
The second is induced by the dg algebra isomorphism $\omega_Q$
of Proposition~\ref{prop:differential-transport}.
The last isomorphism follows from
Lemmas~\ref{lem:hochschild-quasi-equivalence} and~\ref{lem:chen-wang}{\rm(2)}.
Composing these isomorphisms proves the assertion.
$\square$

\begin{Koro}\label{core:weak-general}
The weak form of Keller's conjecture holds for every split
finite-dimensional algebra.
\end{Koro}
{\it Proof.}
Let $A$ be a split finite-dimensional algebra.
Choose a basic split algebra $B$ Morita equivalent to $A$.
By \cite[Theorem~14.4(3)]{ChenLiWang2025}, the weak form of
Keller's conjecture holds for $A$ if and only if it holds for $B$.
Put $J:=\rad(B)$. If $B$ is not semisimple and its radical quiver
has a sink, let $e\in B$ be the corresponding primitive idempotent.
Then $Je=0$. For $f:=1-e$, we have
$$
 B\simeq\begin{bmatrix}ke&eJf\\0&fBf\end{bmatrix},
 \qquad \rad(fBf)=fJf.
$$
Hence $B$ is a one-point coextension of $fBf$, whose radical
quiver is obtained by deleting that sink. By
\cite[Theorem~14.4(1)]{ChenLiWang2025}, the weak form of Keller's
conjecture holds for $B$ if and only if it holds for $fBf$.
Repeating the deletion,
we reach a basic split algebra $C$ which is either semisimple or
has a radical quiver without sinks. Thus the weak form of Keller's
conjecture holds for $A$ if and only if it holds for $C$.
In the semisimple case, both
$\Cbar_{\sg,L}^*(C^{\op})$ and $C^*(\Sdg(C))$ are acyclic.
Otherwise, apply Corollary~\ref{core:existence}.
Hence the weak form of Keller's conjecture holds for $C$, and
therefore also for $A$.
$\square$

\subsection{Compatibility with Keller's canonical map}\label{proof:canonical}
For a split finite-dimensional algebra $A$, let
$q_A:\mathbf D^b_{\dg}(A\text{-mod})\to\Sdg(A)$ be the dg quotient functor.
The functor $\mathcal E_A$ of \cite[Section~1]{Keller2019}
sends $X\in\mathbf D((A^{\op})^e)$ to the dg
$\mathbf D^b_{\dg}(A\text{-mod})$-bimodule
$$
 \mathcal E_A(X)(U,V):=
 \mathbf R\Hom_A(U,V\otimes_{A^{\op}}^{\mathbf L}X)
 \qquad(U,V\in\mathbf D^b_{\dg}(A\text{-mod})).
$$
Write $\overline{\mathcal E}_A:\mathbf D((A^{\op})^e)
\to\mathbf D(\Sdg(A)^e)$ for the composite of $\mathcal E_A$ with derived
extension along $q_A$ in both variables. By
\cite[Section~1]{Keller2019}, this composite vanishes on perfect
$A^{\op}$-bimodules and therefore acts on singular roofs. Its
diagonal identification is
$$
 \epsilon_A:\overline{\mathcal E}_A(A^{\op})
                  \xrightarrow{\sim}I_{\Sdg(A)}.
$$
For $\alpha\in\mathrm{HH}_{\sg}^n(A^{\op},A^{\op})$, Keller's map is
\begin{equation}\label{canonical:definition}
 \eta_A^n(\alpha):=
 \epsilon_A[n]\,\overline{\mathcal E}_A(\alpha)\,\epsilon_A^{-1}
 \in\Hom_{\mathbf D(\Sdg(A)^e)}
        (I_{\Sdg(A)},I_{\Sdg(A)}[n]).
\end{equation}
In particular, for a closed singular cochain $f$ at form stage $p$,
\begin{equation}\label{canonical:roof}
 \eta_A^n(\mathrm{can}_{A,k}[f])
 =\epsilon_A[n]\,
   \overline{\mathcal E}_A(t_p)[n]^{-1}
   \overline{\mathcal E}_A([f])\,\epsilon_A^{-1}.
\end{equation}
The same formula holds for $\mathrm{can}_{A,E}$ and the relative
bar resolution. We shall construct $F_A$ and prove the commutativity of
$$
\begin{CD}
 H^*(\Cbar_{\sg,L}^*(A^{\op})) @>{H^*(F_A)}>> H^*(C^*(\Sdg(A)))\\
 @V{\mathrm{can}_{A,k}}V{\cong}V @VV{=}V\\
 \mathrm{HH}_{\sg}^*(A^{\op},A^{\op}) @>{\eta_A}>>
             \mathrm{HH}^*(\Sdg(A),\Sdg(A)).
\end{CD}
$$
Thus $H^*(F_A)=\eta_A\mathrm{can}_{A,k}$; after the canonical
identification of the source, this is written $H^*(F_A)=\eta_A$.

Let $B=E\oplus J$ be a nonzero basic split algebra, where
$J=\rad(B)$ and $E=\bigoplus_{i=1}^rke_i$.
We first construct an isomorphism
$\sigma_{B,E}:\Cbar_{\sg,R}^*(B)\to\Cbar_{\sg,R,E}^*(B)$ in
$\Ho_k(\Binf)$ such that, under Lemma~\ref{src:syzygy-stabilization},
$$
 H^n(\sigma_{B,E})=\id_{\mathrm{HH}_{\sg}^n(B,B)}
 \qquad(n\in\mathbb Z).
$$
Define algebra homomorphisms
$$
 \pi_E:B\longrightarrow E,\quad
 \sum_{i=1}^r a_i e_i+x\longmapsto\sum_{i=1}^r a_i e_i,
 \qquad
 \chi:B\longrightarrow k,\quad
 \sum_{i=1}^r a_i e_i+x\longmapsto a_1
$$
for $a_i\in k$ and $x\in J$. Put
$$
 \Lambda:=\begin{pmatrix}B&B\\0&B\end{pmatrix},\qquad
 f:=\begin{pmatrix}1&0\\0&0\end{pmatrix},\qquad
 g:=1-f,\qquad E_{\Lambda}:=kf\oplus Eg.
$$
The algebra homomorphism
$$
 \Lambda\longrightarrow E_{\Lambda},\qquad
 \begin{pmatrix}a&x\\0&b\end{pmatrix}
 \longmapsto\chi(a)f+\pi_E(b)g
$$
for $a,x,b\in B$ is a retraction of the inclusion $E_{\Lambda}\subseteq\Lambda$.
Thus the relative singular complexes of Definition~\ref{def:relative-singular}
are defined for $(\Lambda,E_{\Lambda})$.

Put $E_f:=k1_B$ and $E_g:=E$. The algebra isomorphisms onto the
two diagonal subalgebras are
$$
 \iota_f:B\xrightarrow{\sim}f\Lambda f,\quad
 a\longmapsto\begin{pmatrix}a&0\\0&0\end{pmatrix},
 \qquad
 \iota_g:B\xrightarrow{\sim}g\Lambda g,\quad
 a\longmapsto\begin{pmatrix}0&0\\0&a\end{pmatrix}.
$$
Since $g\Lambda f=0$, a composable bar word with both endpoints in
$h\in\{f,g\}$ lies entirely in $h\Lambda h$. Hence the maps $\iota_h$
induce identifications
$$
 h(\Lambda/E_{\Lambda})h\simeq B/E_h,\qquad
 h\Omega_{\mathrm{nc},R,E_{\Lambda}}^p(\Lambda)h\simeq \Omega_{\mathrm{nc},R,E_h}^p(B)\qquad(p\geq0).
$$
The coefficient identifications use the bar cokernels in
Definition~\ref{def:relative-singular}.
For $p\geq0$ and $h\in\{f,g\}$, define the restriction map
$$
\begin{aligned}
 r_{h,p}:\Cbar_{E_{\Lambda}}^*(\Lambda,\Omega_{\mathrm{nc},R,E_{\Lambda}}^p(\Lambda))
   &\longrightarrow \Cbar_{E_h}^*(B,\Omega_{\mathrm{nc},R,E_h}^p(B)),\\
 (r_{h,p}F)(s\bar a_1\tensor\cdots\tensor s\bar a_m)
   &:=hF(s\overline{\iota_h(a_1)}\tensor\cdots
                    \tensor s\overline{\iota_h(a_m)})h .
\end{aligned}
$$
In this formula $F$ has degree $n$, $m=n+p\geq0$, and
$a_1,\ldots,a_m\in B$; the value on the right is read in
$\Omega_{\mathrm{nc},R,E_h}^p(B)$ using the displayed identification.
For $m=0$, the formula means $(r_{h,p}F)(1_B):=hF(h)h$.
The bar differential and transition maps restrict to each diagonal:
$$
 r_{h,p}d=dr_{h,p},\qquad
 r_{h,p+1}\theta_{p,R,E_{\Lambda}}=\theta_{p,R,E_h}r_{h,p}.
$$
We therefore obtain maps of complexes
$$
 r_k:=\varinjlim_p r_{f,p}:
 \Cbar_{\sg,R,E_{\Lambda}}^*(\Lambda)\longrightarrow\Cbar_{\sg,R}^*(B),
 \qquad
 r_E:=\varinjlim_p r_{g,p}:
 \Cbar_{\sg,R,E_{\Lambda}}^*(\Lambda)\longrightarrow\Cbar_{\sg,R,E}^*(B).
$$

An absolute--relative $B_\infty$ quasi-isomorphism is already
given in \cite[Lemma~8.12]{ChenLiWang2025}.
The following proposition specifies the maps $r_k,r_E$ and their
action on singular Hochschild cohomology. Its proof uses the
upper triangular algebra $\Lambda$ and the restriction-map argument of
\cite[Section~4.5]{Keller2003} and
\cite[Sections~9.3--9.4]{ChenLiWang2025}, with relative subalgebras $k$ and $E$ at the two diagonal entries.

\begin{Prop}\label{src:mixed}
Let $B=E\oplus J$ be a nonzero basic split algebra, where
$J=\rad(B)$ and $E=\bigoplus_{i=1}^rke_i$. Put
$\Lambda:=\left(\begin{smallmatrix}B&B\\0&B\end{smallmatrix}\right)$
and $E_{\Lambda}:=kf\oplus Eg$, where $f,g$ are the diagonal idempotents of $\Lambda$.
The maps $r_k,r_E$ are strict $B_\infty$ quasi-isomorphisms.
The isomorphism
$$
 \sigma_{B,E}:=r_Er_k^{-1}:
 \Cbar_{\sg,R}^*(B)\longrightarrow\Cbar_{\sg,R,E}^*(B)
 \quad\hbox{in }\Ho_k(\Binf)
$$
induces the identity on $\mathrm{HH}_{\sg}^*(B,B)$ under the
identifications of Lemma~\ref{src:syzygy-stabilization}.
\end{Prop}

{\it Proof.}
Put $\bar B_k:=B/k1$ and $\bar B_E:=B/E$.
For $a\in\{k,E\}$, let $U_a$ be the complex obtained
from $\overline{\operatorname{Bar}}_a(B)$ by placing $B$ in degree $1$, with
$d^0:\overline{\operatorname{Bar}}_a(B)^0\to B$ given by multiplication.
Define $Z:=U_k\tensor_B B[1]\tensor_B U_E$.
The tensor differential uses the cohomological sign rule. Both factors
are bounded above and termwise projective on either single side, so
they are K-flat on those sides. Tensoring either acyclic
factor therefore gives an acyclic complex. Its highest term is
$Z^1=B$, and its nonpositive terms form a projective resolution of $B$.
For $p\geq0$, expanding the tensor product gives
$$
\begin{aligned}
 Z^{-p}\simeq{}&\overline{\operatorname{Bar}}_k(B)^{-p}\oplus \overline{\operatorname{Bar}}_E(B)^{-p}\\
 &\oplus\bigoplus_{i+j=p-1}
 B\tensor(s\bar B_k)^{\tensor i}\tensor sB
 \tensor_E(s\bar B_E)^{\tensor_E j}\tensor_E B.
\end{aligned}
$$
Each of the first two summands is projective over $B^e$.
For the relative term this follows by decomposing the middle
$E$-bimodule into its $(e_u,e_w)$ components: the induced summands
are copies of $Be_u\tensor e_wB$. The same decomposition on the
right of each summand in the last direct sum, together with its free left factor,
proves projectivity of every summand in the last direct sum.

The map $Z^0\to B$ is the sum of the two multiplications.
On $B\otimes sB\otimes_E B\subseteq Z^{-1}$, the differential is
$$
a\tensor sx\tensor b\longmapsto
       (-a\tensor xb,\ ax\tensor_E b).
$$
The right multiplication map crosses the degree $-1$ factor $B[1]$ and hence
has the minus sign. In a term with $i$ left bars, the entire right
bar differential has coefficient $(-1)^{i+1}$; the left bar
differential has its usual coefficients. These statements specify
the differential in every degree, including the maps at the ends.

Equivalently, apply $f(-)g$ to the relative bar complex of
$\Lambda$, including its map to $\Lambda$.
Since $g\Lambda f=0$, every nonzero tensor has exactly one factor
in $f\Lambda g$. If that factor is the right or left outer
coefficient, one obtains $\overline{\operatorname{Bar}}_k(B)$ or
$\overline{\operatorname{Bar}}_E(B)$, respectively. If it is an
internal factor, one obtains the last direct sum. With $i$ factors
to its left, the right bar differential has sign $(-1)^{i+1}$.
Thus the identification with $Z$ commutes with differentials.

For $p\geq0$, define the $B$-bimodule
$$
 \Omega_Z^p(B):=\operatorname{Coker}(d_Z^{-p-1}:Z^{-p-1}\longrightarrow Z^{-p}).
$$
The terms in this formula are regarded as unshifted $B$-bimodules.
There are exact sequences
\begin{equation}\label{src:syzygy}
 0\longrightarrow\Omega^{p+1}_Z(B)
 \longrightarrow Z^{-p}\longrightarrow\Omega^p_Z(B)
 \longrightarrow0,
\end{equation}
where the middle term in this formula is regarded as an unshifted
module. Starting with $\Omega^0_Z(B)=B$, induction splits
these sequences separately as left modules and as right modules.
Thus the syzygies are projective on either single side; they are
not asserted to be projective over $B^e$.

Put $W_p:=\Omega^p_Z(B)[p]$. Its right-form coordinates are
$$
W_p=(s\bar B_k)^{\tensor p}\tensor B\ \oplus
 \bigoplus_{i+j=p-1}(s\bar B_k)^{\tensor i}\tensor sB
                  \tensor_E(s\bar B_E)^{\tensor_E j}\tensor_E B.
$$
The left $B$-action is the action transported from the syzygy, not
naive multiplication of the first quotient factor.

Inside $\overline{\operatorname{Bar}}_{E_{\Lambda}}(\Lambda)$ take the subcomplex
$$
 \mathcal U:=(\Lambda f\tensor_B \overline{\operatorname{Bar}}_k(B)\tensor_B f\Lambda)
       \oplus(\Lambda g\tensor_B \overline{\operatorname{Bar}}_E(B)\tensor_B g\Lambda).
$$
In the first summand all internal factors lie in $f\Lambda f$;
in the second they lie in $g\Lambda g$. Adjacent multiplication
preserves these conditions, so $\mathcal U$ is a subcomplex.
Set $S:=\overline{\operatorname{Bar}}_k(B)\tensor_B B\tensor_B \overline{\operatorname{Bar}}_E(B)$, a projective bimodule
resolution of $B$. Taking the quotient by $\mathcal U$ gives
\begin{equation}\label{src:quotient}
 \overline{\operatorname{Bar}}_{E_{\Lambda}}(\Lambda)/\mathcal U\simeq \Lambda f\tensor_B S[1]\tensor_B g\Lambda.
\end{equation}
The identification on the component with $i$ left bars has coefficient
$(-1)^i$, obtained by moving the suspension to the factor $f\Lambda g$.
The shifted differential on $S[1]$ is minus that on $S$.
For a left differential $i$ decreases by one, and the quotient of
the two identification signs is $-1$. For a right differential
$i$ is unchanged, leaving the required coefficient $(-1)^{i+1}$.
Terms moving the factor $f\Lambda g$ into an outer coefficient
belong to $\mathcal U$ and vanish in the quotient. These checks prove \eqref{src:quotient} in all degrees.

Applying bimodule Hom into the shifted $p$-forms of $\Lambda$ gives
an exact sequence of complexes
\begin{equation}\label{src:kernelSES}
 0\longrightarrow\Hom_{B^e}(S[1],W_p)
 \longrightarrow \Cbar_{E_{\Lambda}}^*(\Lambda,\Omega_{\mathrm{nc},R,E_{\Lambda}}^p(\Lambda))
 \xrightarrow{(r_{f,p},r_{g,p})}
 \Cbar^*(B,\Omega_{\mathrm{nc},R}^p(B))\oplus \Cbar_E^*(B,\Omega_{\mathrm{nc},R,E}^p(B))
 \longrightarrow0.
\end{equation}
Surjectivity is degreewise: extend a cochain on $\mathcal U$ by
zero on tensors having an internal factor in $f\Lambda g$. It is not claimed that this extension
commutes with the differential or the singular transition.

A degree $n$ map
$F:S\to W_p$ has components $F_{i,j}$ with $i+j=n+p$, on words
$$
 (sa_1,\ldots,sa_i; x; sb_1,\ldots,sb_j),
 \qquad a_h\in\bar B_k,\quad x\in B,\quad b_h\in\bar B_E.
$$
It is right $E$-linear in the middle coordinates, and is extended
$B$-bilinearly in the two external factors. Use arbitrary representatives
when expanding the bar differential; the complete sums below are
well defined on normalized classes.

For $i+j=n+p+1$, write $(dF)_{i,j}=(-1)^{n+1}(\mathcal L+\mathcal R)$.
Here, if $i>0$,
$$
\begin{aligned}
 \mathcal L:={}&a_1F_{i-1,j}(a_2,\ldots,a_i;x;b_1,\ldots,b_j)\\
 &+\sum_{h=1}^{i-1}(-1)^h
 F_{i-1,j}(a_1,\ldots,a_ha_{h+1},\ldots,a_i;x;b_1,\ldots,b_j)\\
 &+(-1)^iF_{i-1,j}(a_1,\ldots,a_{i-1};a_ix;b_1,\ldots,b_j),
\end{aligned}
$$
and $\mathcal L=0$ if $i=0$. Suspensions on arguments in these
expanded formulas are understood. If $j>0$,
$$
\begin{aligned}
 \mathcal R:={}&(-1)^iF_{i,j-1}(a_1,\ldots,a_i;xb_1;b_2,\ldots,b_j)\\
 &+\sum_{h=1}^{j-1}(-1)^{i+h}
 F_{i,j-1}(a_1,\ldots,a_i;x;b_1,\ldots,b_hb_{h+1},\ldots,b_j)\\
 &+(-1)^{i+j}F_{i,j-1}(a_1,\ldots,a_i;x;b_1,\ldots,b_{j-1})b_j,
\end{aligned}
$$
and $\mathcal R=0$ if $j=0$. The coefficients of $\mathcal R$
come from the tensor differential $d_{\overline{\operatorname{Bar}}_k(B)}+(-1)^i d_{\overline{\operatorname{Bar}}_E(B)}$.
This lists both end terms, both sets of internal terms, and the two
terms adjacent to the middle factor $x$.

For the following lifting calculation, regard $Z^{-p}$ as a
complex concentrated in degree $-p$, and let
$q_p:Z^{-p}\to W_p$ be the quotient map.
The map $j_{Z,p}:W_p\to Z^{-p}$, $w\mapsto1\tensor w$, is a
right $B$-linear section of $q_p$.
For $F\in\Hom_{B^e}^n(S,W_p)$, define
$\widehat F\in\Hom_{B^e}^n(S,Z^{-p})$ by
$$
 \widehat F(a_0\tensor u\tensor b_0)
 :=a_0\,j_{Z,p}\bigl(F(1\tensor u\tensor1)\bigr)\,b_0,
$$
where $a_0,b_0\in B$ and
$$
 u\in(s\bar B_k)^{\tensor i}\tensor B
       \tensor_E(s\bar B_E)^{\tensor_E j},\qquad i+j=n+p.
$$
The right $B$-linearity of $j_{Z,p}$ makes the formula balanced
over $E$ at the right external factor. Thus $\widehat F$ is a
well-defined homogeneous $B$-bimodule map of degree $n$ and
$q_p\widehat F=F$.
This is the cocycle-lifting procedure used in the proof of
\cite[Lemma~9.13]{ChenLiWang2025}.
Every term of $d\widehat F-\widehat{dF}$ cancels except the first term
of $\mathcal L$. Indeed internal evaluations have identical
$j_{Z,p}F$ values and coefficients; the last term of $\mathcal R$
cancels by right $B$-linearity. Put $z:=F_{i-1,j}(a_2,\ldots,a_i;x;b_1,\ldots,b_j)$ for $i>0$.
Then
\begin{equation}\label{src:defect}
 (d\widehat F-\widehat{dF})_{i,j}
 =\begin{cases}(-1)^{n+1}(a_1j_{Z,p}z-j_{Z,p}(a_1z)),&i>0,\\
 0,&i=0.
 \end{cases}
\end{equation}

Let $J_p:W_{p+1}[-1]\to Z^{-p}$ denote the syzygy injection.
The definition of the bar differential and of the transported left
action gives, with the degree shifted back on the left,
$$
J_p(sa\tensor z)=aj_{Z,p}z-j_{Z,p}(az).
$$
To verify it, apply the bar differential to $1\tensor sa\tensor z$.
The first face gives $a\tensor z$. The remaining alternating faces
give the negative of $1\tensor az$, precisely the right-form formula
for the left action. The calculation applies also when an internal factor lies in
$f\Lambda g$, since the differential is induced by the relative
bar differential of $\Lambda$.

The transition induced by $\theta=\id\tensor-$ in the coordinates
of $\Hom(S,W_p)$ is
\begin{equation}\label{src:vartheta}
 (\vartheta_pF)_{i,j}
 =\begin{cases}(-1)^nsa_1\tensor
    F_{i-1,j}(a_2,\ldots,a_i;x;b_1,\ldots,b_j),&i>0,\\
 0,&i=0.
 \end{cases}
\end{equation}
Indeed, in the kernel before removing the shift $S[1]$, the degree
is $n+1$, so the tensor transition contributes $(-1)^{n+1}$.
The two identifications in \eqref{src:quotient} differ by
$(-1)^i/(-1)^{i-1}=-1$. Their product is $(-1)^n$.
For $i=0$ the removed first input lies in $f\Lambda g$;
all remaining internal factors lie in $g\Lambda g$, where a
cochain in the kernel of $(r_{f,p},r_{g,p})$ vanishes.

For closed $F$, equations \eqref{src:defect}--\eqref{src:vartheta} give
\begin{equation}\label{src:correct-sign}
                 d\widehat F=-J_p\vartheta_pF.
\end{equation}
The minus sign in this equation is essential. In the cone of $J_p$
the pair $(\widehat F,\vartheta_pF)$ is a closed lift of $F$:
its first differential component is zero by
\eqref{src:correct-sign}, and its second is zero since
$\vartheta_pF$ is closed. Projection to the second component is
$\vartheta_pF$. Thus $\vartheta_p$ represents the specified connecting
map, with exactly the cone convention of \eqref{pre:cone}.
This supplies the transition identification, including $i=0$, rather
than only isomorphisms between the finite-stage cohomology groups.

As a lowest-degree check, take $p=n=0$ and let
$F_{0,0}(x)=x$ be the multiplication map $S\to B$.
On the input $(sa;x)$ the lifted differential in the first summand
of $Z^0$ is $-a\tensor x+1\tensor ax$.
The transition is $sa\tensor x$, whose syzygy injection is
$a\tensor x-1\tensor ax$. This verifies the negative sign in
\eqref{src:correct-sign} before any higher-length cancellation is used.

Apply Lemma~\ref{src:syzygy-stabilization} to $B^e$, $M=X=B$, and
\eqref{src:syzygy}. The complex $S$ is a bounded-above resolution by
finitely generated projective $B^e$-modules, so
$H^n\Hom_{B^e}(S,W_p) =\Hom_{\mathbf D^b(B^e\text{-}\mathrm{mod})}(B,W_p[n])$.
The calculation \eqref{src:correct-sign} identifies the transition on
these groups with postcomposition by the same connecting map
$\partial_p[n]$. Therefore
$$
\varinjlim_p H^n\Hom_{B^e}(S,W_p)
       \xrightarrow{\sim}
 \Hom_{\mathbf D_{\sg}(B^e)}(B,B[n]),
 \qquad [f]\longmapsto\mathsf q(t_p[n])^{-1}\mathsf q(f).
$$

The shift in the kernel is also explicit. For a homogeneous
$F:S\to W_p$ of degree $r$, define
$$
 \widetilde F:S[1]\to W_p,\qquad
 \widetilde F(sx):=(-1)^rF(x)\qquad(x\in S\text{ homogeneous}).
$$
It has degree $r+1$ and gives the chain isomorphism
$\Hom(S,W_p)[-1]\simeq\Hom(S[1],W_p)$: using
$d_{S[1]}(sx)=-s(d_Sx)$ gives
$$
 (d\widetilde F)(sx)=(-1)^r d_{W_p}F(x)-F(d_Sx)
       =-\widetilde{dF}(sx).
$$
The factor $(-1)^r$ is unchanged by the degree-zero transition
$\vartheta_p$. By Lemma~\ref{src:syzygy-stabilization} and exactness of filtered
colimits, we obtain
$$
 H^n\!\left(\varinjlim_p\Hom_{B^e}(S[1],W_p)\right)
       =\mathrm{HH}_{\sg}^{n-1}(B,B).
$$

To identify the boundary map, use the bimodule exact sequence
\begin{equation}\label{src:triangle}
 0\to \Lambda f\tensor_B B\tensor_B g\Lambda
 \xrightarrow{\iota_\Lambda}
 (\Lambda f\tensor_B f\Lambda)\oplus(\Lambda g\tensor_B g\Lambda)
 \xrightarrow{\mu+\mu}\Lambda\to0,
\end{equation}
where $\iota_\Lambda(a\tensor x\tensor b):=(-a\tensor xb,ax\tensor b)$.
Applying $f(-)f$ or $g(-)g$, the last map is the identity.
Applying $f(-)g$, the sequence is $0\to B\xrightarrow{(-1,1)}B\oplus B
\xrightarrow{(1,1)}B\to0$. This proves exactness and fixes the sign.
The inclusions
$\overline{\operatorname{Bar}}_k(B)\tensor_B B\to Z^{\leq0}$ and
$B\tensor_B \overline{\operatorname{Bar}}_E(B)\to Z^{\leq0}$ commute with the maps to $B$.
They induce maps $a_p,b_p$ on syzygies. For the absolute bar
resolution, the relative bar resolution and $Z^{\leq0}$, respectively,
denote the maps $t_p$ of Lemma~\ref{src:syzygy-stabilization} by
$t_{k,p},t_{E,p},t_{Z,p}$. Naturality gives
$$
 a_p[p]t_{k,p}=t_{Z,p},\qquad
 b_p[p]t_{E,p}=t_{Z,p},
$$
Indeed, the inclusions give morphisms of short exact syzygy
sequences, and hence commute with the shifted connecting maps.
Let
$$
 u_k:=\id\otimes\id_B\otimes\varepsilon_E:S\longrightarrow
       \overline{\operatorname{Bar}}_k(B),\qquad
 u_E:=\varepsilon_k\otimes\id_B\otimes\id:S\longrightarrow
       \overline{\operatorname{Bar}}_E(B),
$$
where $\varepsilon_k,\varepsilon_E$ are the multiplication maps
of the two bar resolutions and the remaining regular $B$ factors
are contracted by multiplication. Both maps are chain maps
inducing $\id_B$.

Let $u$ and $v$ be closed degree-$n$ cochains in the two terms on
the right of \eqref{src:kernelSES}. Extend $(u,v)$ by zero on
words with an internal factor in $f\Lambda g$, obtaining a
homogeneous cochain $\widetilde F$ of degree $n$ in the middle term.
With the cone convention \eqref{pre:cone}, the connecting class
is represented by $-d\widetilde F$ in the kernel.
Put $m:=n+p$. On a mixed word with $i$ left bars and $j$ right
bars, where $i+j=m\geq0$, only the faces moving the mixed factor
into an outer coefficient contribute. The right outer face,
which occurs when $j=0$, has coefficient
$-(-1)^{n+1}(-1)^{m+1}=(-1)^{p+1}$ in $-d\widetilde F$.
The left outer face, which occurs when $i=0$, has coefficient
$-(-1)^{n+1}=(-1)^n$.
Passing from the mixed quotient to $S[1]$ contributes $(-1)^i$,
and the isomorphism $\Hom(S,W_p)[-1]\simeq\Hom(S[1],W_p)$
contributes $(-1)^n$. Thus the two coefficients become
$$
 (-1)^{n+m+p+1}=-1\quad(j=0),\qquad
 (-1)^{n+n}=1\quad(i=0).
$$
The two evaluations are respectively $a_p[p]u\,u_k$ and
$b_p[p]v\,u_E$. Consequently the finite-stage connecting map,
with its target identified with $H^n\Hom_{B^e}(S,W_p)$, is
$$
 ([u],[v])\longmapsto
 \bigl[-a_p[p]u\,u_k+b_p[p]v\,u_E\bigr].
$$
When $m=0$ both outer faces occur and give the same formula;
when $m<0$ the source cochain spaces vanish.
Since $u_k,u_E$ induce $\id_B$, the two syzygy identities give
$$
 t_{Z,p}[n]^{-1}a_p[p+n][u]=t_{k,p}[n]^{-1}[u],\qquad
 t_{Z,p}[n]^{-1}b_p[p+n][v]=t_{E,p}[n]^{-1}[v]
$$
in $\mathbf D_{\sg}(B^e)$. Hence the colimit connecting map is
$(u,v)\mapsto-u+v$ under Lemma~\ref{src:syzygy-stabilization}.

Put $\mathscr M:=\varinjlim_p \Cbar_{E_{\Lambda}}^*(\Lambda,\Omega_{\mathrm{nc},R,E_{\Lambda}}^p(\Lambda))$ and
$\mathsf H^n:=\mathrm{HH}_{\sg}^n(B,B)$. Exactness of filtered colimits
and \eqref{src:kernelSES} give the exact sequence
$$
 \cdots\longrightarrow\mathsf H^{n-1}\oplus\mathsf H^{n-1}
 \xrightarrow{(-1,1)}\mathsf H^{n-1}
 \longrightarrow H^n(\mathscr M)
 \xrightarrow{(H(r_k),H(r_E))}\mathsf H^n\oplus\mathsf H^n
 \xrightarrow{(-1,1)}\mathsf H^n\longrightarrow\cdots.
$$
Since each map $(-1,1)$ is surjective,
$$
 (H(r_k),H(r_E)):H^n(\mathscr M)
 \xrightarrow{\sim}\{(u,u):u\in\mathsf H^n\}.
$$

Both $H(r_k)$ and $H(r_E)$ are therefore isomorphisms, and
$H(r_E)H(r_k)^{-1}=\id$ under Lemma~\ref{src:syzygy-stabilization}.
It remains to verify strictness. For $a=f$ or $g$, a composable
word with both endpoints in $a$ stays in the $a$ corner, since
$g\Lambda f=0$. The same holds for each evaluation in a cup product or
a brace operation. Consequently, for either corner restriction $r$,
$$
\begin{aligned}
 rd&=dr,&r\theta&=\theta r,\\
 r(u\smile v)&=r(u)\smile r(v),&
 r(u\{v_1,\ldots,v_q\})&=
       r(u)\{r(v_1),\ldots,r(v_q)\}.
\end{aligned}
$$
All identities preserve the tensor signs, since the corner
idempotents have degree zero. The relative structures are those of
Definition~\ref{def:relative-singular}. Thus $r_k,r_E$ are strict
$B_\infty$ quasi-isomorphisms with the asserted cohomology maps.
$\square$

For the rest of the proof, take $\sigma_{B,E}:=r_Er_k^{-1}$ as in
Proposition~\ref{src:mixed}. By Lemma~\ref{lem:left-right-singular}, define
\begin{equation}\label{ops:lambda}
 \lambda_{B^{\op}}:=\tau_E^{-1}\sigma_{B,E}^{\opp}\tau_k:
 \Cbar_{\sg,L}^*(B^{\op})\longrightarrow\Cbar_{\sg,L,E}^*(B^{\op})
 \quad\text{in }\Ho_k(\Binf).
\end{equation}
By \eqref{pre:reversal-syzygy}, the first components of
$\tau_k$ and $\tau_E$ transport a roof $t_p[n]^{-1}[f]$
by applying the same opposite-bimodule functor to its numerator
and denominator. Proposition~\ref{src:mixed} identifies the two
right-form roofs. Applying the inverse reversal gives
\begin{equation}\label{src:left-marking}
 \mathrm{can}_{B,E}H^*(\lambda_{B^{\op}})=\mathrm{can}_{B,k}.
\end{equation}

Let $C=E_C\oplus J_C$ be the basic split algebra of
Subsection~\ref{proof:isomorphism}. Retain its quiver $Q$, basis
$\{j_\alpha:\alpha\in Q_1\}$ and coefficients
$\lambda_{\beta_2\beta_1,\alpha}$.
In the construction of $\Upsilon_C$, use the isomorphism
$\sigma_{C,E_C}$ of Proposition~\ref{src:mixed}.
For an arrow $v\in Q_1$, $j_v^{\op}$ denotes the corresponding
element of $C^{\op}$.
Put $O_i:=\{v\in Q_1:s(v)=i\}$ for $i\in Q_0$.
Set $\lambda_{ba,v}:=0$ when $ba\notin Q_2$ or $v$ is not parallel to $ba$.
Products of radical elements of $C^{\op}$ are therefore read as
$$
 j_a^{\op}j_b^{\op}=(j_bj_a)^{\op}
              =\sum_{v\in Q_1}\lambda_{ba,v}j_v^{\op}
 \qquad(a,b\in Q_1).
$$
For a left $C$-module $U$, tensor products over $C^{\op}$ use
$u c^{\op}:=cu$ for $u\in U$ and $c\in C$.

The relative bar resolution of the left $C$-module $E_C$ has terms
$$
P^{-n}:=(J_C^{\op})^{\otimes_{E_C} n}\otimes_{E_C} {C^{\op}}\quad(n\geq0),\qquad P^i:=0\quad(i>0).
$$
For $a_1,\ldots,a_n\in J_C^{\op}$ and $b\in C^{\op}$, write
$(a_1,\ldots,a_n;b):=a_1\otimes\cdots\otimes a_n\otimes b$.
The complex is equipped
with the quotient map $P^0=C^{\op}\to E_C$ and differential
\begin{align}
 d_P(a_1,\ldots,a_n;b)
 :={}&\sum_{r=1}^{n-1}(-1)^{n-r}
 (a_1,\ldots,a_ra_{r+1},\ldots,a_n;b)\notag\\
 &+(a_1,\ldots,a_{n-1};a_nb).\label{gen:dP}
\end{align}
Products in this formula are expanded linearly. In particular internal
products are not suppressed. Multiplying the suspended bar coordinate
of length $n$ by $(-1)^{n(n+1)/2}$ gives this differential. Thus $P\to E_C$
is a quasi-isomorphism. Each term is a finitely generated projective left $C$-module and
$P$ is bounded above, hence K-projective.

Give $kQ_1$ degree one and put $D:=T_{E_C}(kQ_1)$.
Define a graded left $D$-action on $P$ by letting the vertex $e_i$
project onto words starting at $i$ and letting $v\in Q_1$ act by
$$
 v(a_1,\ldots,a_n;b):=
 (-1)^{n-1}\delta_{j_v^{\op},a_1}(a_2,\ldots,a_n;b),
 \qquad v(P^0):=0.
$$
Here $a_l\in\{j_\alpha^{\op}:\alpha\in Q_1\}$ and $b\in C^{\op}$.
Then $e_{t(v)}ve_{s(v)}=v$ as operators on $P$, and the action
commutes with the right $C^{\op}$-action.
The next lemma identifies the resulting tensor dg algebra with
$\operatorname{End}_C(P)$ up to quasi-isomorphism.

\begin{Lem}\label{gen:actual}
Let $C=E_C\oplus J_C$ be basic split, and let $P$ be the complex
\eqref{gen:dP}. The derivation of $D=T_{E_C}(kQ_1)$ given by
$$
 d_D(e_i):=0,\qquad
 d_D(v):=\sum_{a,b\in Q_1}\lambda_{ba,v}ba
 \quad(i\in Q_0,\ v\in Q_1)
$$
satisfies $d_D^2=0$. The arrow action defines a dg algebra
quasi-isomorphism $\iota_P:D\to\operatorname{End}_C(P)$.
\end{Lem}
{\it Proof.}
For $U=(a_1,\ldots,a_n;b)$, where $a_l=j_{\alpha_l}^{\op}$,
$\alpha_l\in Q_1$, $b\in C^{\op}$ and $n\geq2$,
the faces after the first letter cancel in $d_Pv +v d_P$.
The remaining term is
$$
 (d_Pv +v d_P)(U)
 =-\lambda_{\alpha_2\alpha_1,v}(a_3,\ldots,a_n;b)
 =\left(\sum_{a,b}\lambda_{ba,v} b a \right)(U).
$$
Both sides vanish on words of length at most one. Thus the action
$\iota_P:D\to\operatorname{End}_{C}(P)$ commutes with differentials.
Let $\epsilon:P\to E_C$ be the map induced by $C^{\op}\to E_C$,
and let $\overline b$ denote the image of $b\in C^{\op}$ in $E_C$.
On words of length $r$,
$$
 \epsilon v_r \cdots v_1 
 (a_1,\ldots,a_r;b)
 =(-1)^{r(r-1)/2}
   \left(\prod_{j=1}^r\delta_{j_{v_j}^{\op},a_j}\right)\overline b.
$$
Hence $\epsilon_*\iota_P:D\to\Hom_{C}(P,E_C)$ is an isomorphism of
graded vector spaces, and $\iota_P$ is injective. Consequently
$$
 \iota_P d_{D}^2=d_{\operatorname{End}}^2\iota_P=0
       \quad\Longrightarrow\quad d_{D}^2=0.
$$
Since $P$ is K-projective and $\epsilon$ is a quasi-isomorphism,
the commutative diagram
$$
 \begin{CD}
 D @>{\iota_P}>> \Hom_{C}(P,P)\\
 @V{\epsilon_*\iota_P}V{\cong}V @VV{\epsilon_*}V\\
 \Hom_{C}(P,E_C) @= \Hom_{C}(P,E_C)
 \end{CD}
$$
has a chain isomorphism on the left and a quasi-isomorphism on the
right. Thus $H^*(\iota_P)=H^*(\epsilon_*)^{-1}H^*(\epsilon_*\iota_P)$
is an isomorphism.

$\square$

Since $E_C$ is the direct sum of the simple left $C$-modules,
$\mathbf D^b(C\text{-mod})=\operatorname{thick}(E_C) =\operatorname{thick}(P)$.
By \cite[Section~4.2]{Keller1994} and Lemma~\ref{gen:actual},
$$
\mathbf R\Hom_C(P,-):\mathbf D^b(C\text{-mod})
       \xrightarrow{\sim}\operatorname{per}(D^{\op}).
$$

Until the proof of Theorem~\ref{tw:marked-core} is complete, assume that
$O_i\ne\varnothing$ for every vertex $i\in Q_0$.

Let $M_i:=\Hom_{C}(P,e_i C^{\op})$, with right $D$ action by precomposition.
\begin{Lem}\label{loc:localization}
Let $C=E_C\oplus J_C$ be basic split with radical quiver $Q$
without sinks, and let $D\to(L(Q),d_\mu)$ be the dg algebra
homomorphism induced by the real arrows.
The modules $M_i$ are compact in $\mathbf D(D^{\op})$ and satisfy
$M_i\otimes_D^{\mathbf L}L(Q)=0$ in $\mathbf D(L(Q)^{\op})$.
Multiplication is an isomorphism
$$
 L(Q)\otimes_D^{\mathbf L}L(Q)\xrightarrow{\ \mathrm{mult}\ }L(Q)
 \quad\text{in }\mathbf D(L(Q)^e).
$$
Moreover,
$$
 \ker(-\otimes_D^{\mathbf L}L(Q))=\operatorname{Loc}\{M_i:i\in Q_0\},
$$
and extension of scalars induces a triangle equivalence
$$
 \mathbf D(D^{\op})/\operatorname{Loc}\{M_i:i\in Q_0\}
 \xrightarrow{\ \sim\ }\mathbf D(L(Q)^{\op}).
$$
\end{Lem}
{\it Proof.}
Define $h_i,f_v\in\Hom_{C}^0(P,e_i C^{\op})$, for $v\in O_i$, by
$h_i(b):=-e_ib$ and $f_v(b):=j_v^{\op}b$ on $P^0={C^{\op}}$, and by zero on
$P^{-n}$ for $n>0$. Evaluation on the basis tensors of $P$ gives the graded right
$D$-module decomposition
$$
 M_i=h_ie_i D\oplus
          \bigoplus_{v\in O_i}f_v e_{t(v)}D.
$$
For $a\in J_C^{\op}$ and $b\in C^{\op}$, evaluation of the Hom
differential on $a\otimes b\in P^{-1}$ gives
\begin{equation}\label{loc:projective-d}
 dh_i=\sum_{v\in O_i}f_vv ,\qquad
 df_v=-\sum_{a,w}\lambda_{av,w}f_wa .
\end{equation}
These equations determine the differential in every degree by right
Leibniz. In particular the second term is retained when $(J_C^{\op})^2\ne0$.
Let $M_i'$ be the submodule generated by the $f_v$. Its generators,
ordered by decreasing radical layer, give a finite semifree filtration:
every term of $df_v$ has $\ell(w)>\ell(v)$.
The map
$$
 a_i:e_i D[-1]\longrightarrow M_i',
 \qquad s^{-1}z\longmapsto\sum_v f_vv z\quad(z\in e_iD)
$$
is a chain map. Indeed the differential of its column is
$-\sum\lambda_{av,w} f_wa v +
\sum\lambda_{ba,w} f_wb a =0$; the column has degree one,
which supplies the source shift sign. With our cone convention,
$M_i=\operatorname{Cone}(a_i)$.
Thus $M_i$ is finite semifree and hence compact in
$\mathbf D(D^{\op})$.

By Definition~\ref{def:radical-quiver-leavitt} and
Lemma~\ref{lem:dmu-ghost}, the real arrows induce a dg algebra map
$D\longrightarrow(L(Q),d_\mu)$, where $L(Q)$ is the algebra fixed
in Subsection~\ref{proof:isomorphism}. For $v,w\in Q_1$ with
$s(v)=s(w)=i$,
\begin{equation}\label{loc:inverse}
 vw^*=\delta_{vw}e_{t(v)},\qquad
 \sum_{s(v)=i}v^*v=e_i.
\end{equation}
Moreover,
\begin{equation}\label{loc:inverse-d}
 d_\mu(v^*)=\sum_{s(w)=i}w^*d_D(w)v^*
           =\sum_{b,w\in Q_1}\lambda_{bv,w}w^*b.
\end{equation}
Thus the coefficient computation uses $(L(Q),d_\mu)$ throughout.

The inverse of $a_i\otimes_{D} L(Q)$ sends
$f_vl$ to $s^{-1}(v^* l)$. It is a chain map by
\eqref{loc:projective-d} and \eqref{loc:inverse-d}.
Equivalently, the degree-minus-one right $L(Q)$-linear map
$h_{M_i}:M_i\otimes_{D}L(Q)\to M_i\otimes_{D}L(Q)$ defined by
$$
 h_{M_i}(h_il):=0,\qquad h_{M_i}(f_vl):=h_iv^* l\qquad(l\in L(Q))
$$
satisfies $dh_{M_i}+h_{M_i}d=\id$.
On $h_i$, this identity is the second inverse relation.
On $f_v$, its two extra terms are
$h_i d_\mu(v^*)$ and $-h_i\sum_{b,w}\lambda_{bv,w}w^*b$,
which cancel. The finite semifree filtration means ordinary tensor
here computes derived tensor.

Put $\Omega_{D}:=\ker(D\otimes_{E_C} D\to D)$ and
$\partial c:=c\otimes1-1\otimes c$ for $c\in D$. The degree-zero universal
derivation identifies $\Omega_{D}=D\otimes_{E_C} kQ_1\otimes_{E_C} D$.
For $u_v:=\partial v $ one has
$$
 d_\Omega u_v=\sum_{a,b}\lambda_{ba,v}(u_b a +b u_a).
$$
Ordering $u_v$ by increasing radical layer gives a finite semifree
bimodule filtration. Thus
$K_{D}:=\operatorname{Cone}(\Omega_{D}\to D\otimes_{E_C} D)$ is finite
semifree, with degree-zero generators $su_v$ satisfying
\begin{equation}\label{loc:small-d}
 d_Ksu_v=v \otimes1-1\otimes v +
             \sum_{a,b}\lambda_{ba,v}(b (su_a)-(su_b) a ).
\end{equation}
Here $a(su)=(-1)^{|a|}s(au)$ for homogeneous $a\in D$ and
$u\in\Omega_D$ fixes the shift sign.

Let $m_D:D\otimes_{E_C}D\to D$ be multiplication. Define a
right dg $D$-module map
$$
 r_D:D\otimes_{E_C}D\longrightarrow\Omega_D,
 \qquad r_D(a\otimes b):=a\otimes b-1\otimes ab
 \quad(a,b\in D).
$$
For $y\in(D\otimes_{E_C}D)^n$ and $x\in\Omega_D^{n+1}$, put
$h(y,sx):=(0,sr_D(y))\in K_D^{n-1}$.
Since $r_Dd=dr_D$ and $r_D(x)=x$ for $x\in\Omega_D$,
$$
 (dh+hd)(y,sx)=(r_D(y),sx)=(y,sx)-(1\otimes m_D(y),0).
$$
Thus $K_D\to D$ is a homotopy equivalence of right dg $D$-modules
and, in particular, a quasi-isomorphism of dg bimodules.

The same resolution computes derived extension. The successive
filtration quotients of $K_{D}\otimes_{D} L(Q)$, as left $D$ modules,
are shifts and vertex summands of $D\otimes_{E_C} L(Q)$. They are K-flat:
tensoring an acyclic right $D$ module with such a quotient reduces
to tensor over semisimple $E_C$, where every complex is K-flat.
Thus $L(Q)\otimes_{D} K_{D}\otimes_{D} L(Q)$ computes
$L(Q)\otimes_{D}^{\mathbf L}L(Q)$.

The universal derivation on $L(Q)$ is freely specified on the $v $;
its remaining values are forced by
$$
 \partial v^* =-\sum_{w\in O_i}w^* (\partial w )v^* .
$$
Differentiating the inverse relations proves existence and uniqueness.
Thus $\Omega_{L(Q)}=L(Q)\otimes_{E_C}kQ_1\otimes_{E_C}L(Q)$
as a graded bimodule. Since $\partial d=d\partial$, the identification
also respects differentials. We obtain
$$
 K_{L(Q)}:=L(Q)\otimes_{D} K_{D}\otimes_{D} L(Q)
       =\operatorname{Cone}(\Omega_{L(Q)}\to L(Q)\otimes_{E_C} L(Q)).
$$
The same right $L(Q)$-linear contraction proves that
$K_{L(Q)}\to L(Q)$ is a quasi-isomorphism. We have proved, for the
specified multiplication,
\begin{equation}\label{loc:epimorphism}
 L(Q)\otimes_{D}^{\mathbf L}L(Q)\xrightarrow{\ \sim\ }L(Q).
\end{equation}
We next identify the kernel of extension of scalars.

For $I=v_p\cdots v_1\in Q_p$, put
$I^*:=v_1^*\cdots v_p^*$, with $e_i^*:=e_i$ for $p=0$.
In bar coordinates the tuple $(v_1,\ldots,v_p)$ denotes
$j_{v_1}^{\op}\otimes\cdots\otimes j_{v_p}^{\op}$.
The path $I$ in a bar tensor denotes the same ordered tensor of radical elements.
Define graded right $D$-modules $G_p$ and degree-zero maps
$\theta_p:G_p\to G_{p+1}$ by
$$
 G_p:=\bigoplus_{|I|=p}q_Ie_{t(I)}D,
 \quad |q_I|=-p,\qquad
 \theta_p(q_Ic):=\sum_{u\in O_{t(I)}}q_{I,u}u c
 \quad(c\in e_{t(I)}D).
$$
We also write $q_{v_1,\ldots,v_p}:=q_I$ and
$q_{I,u}:=q_{uI}$ for $u\in O_{t(I)}$.
The empty words give $G_0=D$. The usual normal form for the
presentation \eqref{loc:inverse} gives
$\varinjlim G_p=L(Q)$, with $q_Ic\mapsto I^* c$.
Indeed real--inverse adjacent pairs reduce by the first relation;
the remaining relations insert $\sum u^* u $ at the boundary between
the inverse and real words, precisely the displayed transitions.
For the assertion that these boundary relations generate the entire
relation ideal, apply \cite[Lemma~2.5 and Theorem~2.6]{ChenWang2024} with opposite multiplication
to obtain this right-module convention. Each $\theta_p$ is injective:
for fixed $I$, at least
one outgoing $u$ exists, and left multiplication by its tensor letter
$u $ in $D$ is injective on the appropriate corner. Different $I$
have disjoint output components.

Differentiate $I^*$ and insert
$e_i=\sum_{u\in O_i}u^*u$ at the endpoint of each shortened word.
The result lies in $G_p$, so $G_p$ is a dg submodule of $L(Q)$. Explicitly,
for $p\geq1$,
\begin{align}
 d_{G_p}q_I={}&\sum_{r=1}^{p-1}(-1)^{r-1}
  \sum_{w,u}\lambda_{v_{r+1}v_r,w}
  q_{v_1,\ldots,v_{r-1},w,v_{r+2},\ldots,v_p,u}u \notag\\
 &+(-1)^{p-1}\sum_{b,w}\lambda_{bv_p,w}
  q_{v_1,\ldots,v_{p-1},w}b .\label{loc:stage-d}
\end{align}
In the first sum $u$ starts at the original terminal vertex.
Injectivity into $L(Q)$ gives $d_{G_p}^2=0$ and the chain property of
$\theta_p$. Every generator term strictly increases
$\mathrm{wt}(I)=\sum_r\ell(v_r)$, so decreasing weight gives a finite
semifree filtration of $G_p$.

Filter $\operatorname{Cone}(\theta_p)$ by the prefix weight
$\mathrm{wt}(I)$, assigned to both $sq_I$ and $q_{I,v}$.
The source differential increases it. Merging inside the target prefix,
including its last letter with $v$, increases it as well: after inserting the inverse relation,
the new prefix contains the old final letter $v$.
Only the map $G_p[1]\to G_{p+1}$ in the cone differential and
the differential of the last inverse letter preserve the weight.
The associated graded complex is therefore a direct sum of
$M_{t(I)}[p+1]$, with the summand indexed by $I$ identified by
$$
 s^{p+1}h_{t(I)}\longmapsto(-1)^{p+1}sq_I,
 \qquad s^{p+1}f_v\longmapsto q_{I,v}.
$$
The shifted $df_v$ coefficient is $(-1)^p\lambda_{bv,w}$, exactly
the last-letter term of \eqref{loc:stage-d}. This is a finite dg
filtration, so it proves
$\operatorname{Cone}(\theta_p)\in\operatorname{thick}\{M_i\}$.

Set $\mathcal T:=\operatorname{Loc}\{M_i\}\subset\mathbf D(D^{\op})$.
Every cone of $D\to G_p$ belongs to $\mathcal T$. Filtered colimits
of vector spaces are exact, so the homotopy colimit of this sequence
is its ordinary colimit $L(Q)$. Hence $\operatorname{Cone}(D\to L(Q))$
belongs to $\mathcal T$. Moreover, tensor--Hom adjunction and
$M_i\otimes_D^{\mathbf L}L(Q)=0$ give
$\Hom_{\mathbf D(D^{\op})}(T,L(Q)[n])=0$ for all $T\in\mathcal T$
and $n\in\mathbb Z$. We call an object satisfying these vanishing
conditions $\mathcal T$-local. Thus $D\to L(Q)$ is a
$\mathcal T$-localization: its cone belongs to $\mathcal T$ and
its target is $\mathcal T$-local.

The $M_i$ are compact, so this localization preserves coproducts.
By \eqref{loc:epimorphism}, the functor
$X\mapsto\operatorname{Res}_{D^{\op}}(X\otimes_D^{\mathbf L}L(Q))$,
with its natural map from $X$, is another exact localization,
preserves coproducts, annihilates $\mathcal T$ and has
$\mathcal T$-local values. Here $\operatorname{Res}_{D^{\op}}$
is restriction of right modules along $D\to L(Q)$.
The universal property of $\mathcal T$-localization induces a
natural map from the $\mathcal T$-localization of $X$ to
$\operatorname{Res}_{D^{\op}}(X\otimes_D^{\mathbf L}L(Q))$.
It is an isomorphism for $X=D$. The objects for which it is an
isomorphism form a localizing subcategory containing $D$, hence
all of $\mathbf D(D^{\op})$. Therefore
$$
 \ker(-\otimes_{D}^{\mathbf L}L(Q))=\operatorname{Loc}\{M_i\},\qquad
 \mathbf D(D^{\op})/\operatorname{Loc}\{M_i\}\simeq\mathbf D(L(Q)^{\op}).
$$
$\square$

We now realize the equivalence of Lemma~\ref{loc:localization}
on the dg categories used for the coefficient calculation.

\begin{Lem}\label{loc:specified-dg-quotient}
Let $C=E_C\oplus J_C$ be basic split and assume that its radical
quiver $Q$ has no sinks. Let $D\to L(Q)$ be the dg algebra map
\eqref{loc:inverse}--\eqref{loc:inverse-d}, and let
$M_i=\Hom_{C}(P,e_i C^{\op})$.
The extension functor $-\otimes_{D} L(Q)$ on K-projective perfect
modules factors through a dg quotient model of
$\operatorname{per}(D^{\op})/\operatorname{thick}\{M_i\}$, and the resulting
dg functor is a Morita equivalence with $\perdg(L(Q)^{\op})$.
Under the generator identifications its extension on coefficient
bimodules is $N\mapsto L(Q)\otimes_{D}^{\mathbf L}N\otimes_{D}^{\mathbf L}L(Q)$,
with the unit identified by the multiplication
\eqref{loc:epimorphism}.
\end{Lem}
{\it Proof.}
Let $\mathcal P_{D}$ be a small full dg category of K-projective
perfect right $D$-modules, containing $D$ and the modules $M_i$,
closed under shifts and cones, and containing representatives of
all perfect objects, including retracts. Thus
$H^0(\mathcal P_{D})=\operatorname{per}(D^{\op})$ up to equivalence.
Choose $\mathcal P_{L(Q)}$ similarly, including all modules $U\otimes_{D} L(Q)$
for $U\in\mathcal P_{D}$. These modules are K-projective: for every
acyclic right $L(Q)$-module $V$, the adjunction
$$
 \Hom_{L(Q)^{\op}}(U\otimes_{D} L(Q),V)=\Hom_{D^{\op}}(U,V|_{D})
$$
makes the Hom complex acyclic. They are perfect because derived
extension is exact, preserves retracts, and sends $D$ to $L(Q)$.
Ordinary extension on these representatives therefore defines the
dg functor
$$
 \mathcal F:\mathcal P_{D}\longrightarrow\mathcal P_{L(Q)},
 \qquad U\longmapsto U\otimes_{D} L(Q),
$$
representing the specified derived extension, including its maps.

Let $\mathcal N$ be the full dg subcategory on the objects whose
classes belong to $\operatorname{thick}\{M_i\}$.
The derived extension $-\otimes_D^{\mathbf L}L(Q)$ annihilates
$\mathcal T$, so $\mathcal F(N)$ is acyclic for every
$N\in\mathcal N$. It is also K-projective, so
$\Hom_{L(Q)^{\op}}(\mathcal F(N),\mathcal F(N))$ is acyclic. In particular choose a degree
$-1$ endomorphism $h_N$ with $d(h_N)=\id_{\mathcal F(N)}$.
For $N=M_i$, take $h_N:=h_{M_i}$.

Over the field $k$, the Drinfeld quotient
$\mathcal P_D/\mathcal N$ is obtained by freely adjoining
to each $N\in\mathcal N$ an endomorphism $\epsilon_N$ of degree
$-1$ with $d\epsilon_N=\id_N$; see
\cite[Section 4.4, Theorem 4.8]{Keller2006}.
Sending $\epsilon_N$ to $h_N$ extends $\mathcal F$ to a dg functor
$\overline{\mathcal F}:\mathcal P_D/\mathcal N\to\mathcal P_{L(Q)}$.
There are no relations requiring compatible choices of the different
$h_N$: the freely adjoined generators and their differentials
specify the extension. Write $q:\mathcal P_D\to\mathcal P_D/\mathcal N$ for
the quotient functor; then $\overline{\mathcal F}q=\mathcal F$ on objects and morphisms.
Both $\mathcal P_D$ and $\mathcal N$ are closed under shifts and
cones. By \cite[Lemma~2.7]{ChenLiWang2025}, the dg quotient
$\mathcal P_D/\mathcal N$ is pretriangulated, and $q$ induces a
triangle equivalence
$$
 H^0(\mathcal P_D)/H^0(\mathcal N)
 =\operatorname{per}(D^{\op})/\operatorname{thick}\{M_i\}
 \xrightarrow{\ \sim\ }H^0(\mathcal P_D/\mathcal N).
$$
Under this equivalence, $H^0(\overline{\mathcal F})$ is induced by the specified
derived extension.

By Lemma~\ref{loc:localization} and
\cite[Theorem~4.11]{Keller2006},
$H^0(\overline{\mathcal F})$ is fully faithful.
For $U,V\in\mathcal P_D/\mathcal N$ and $n\in\mathbb Z$, pretriangulatedness gives
$$
\begin{aligned}
 H^n\bigl((\mathcal P_D/\mathcal N)(U,V)\bigr)
 &\simeq\Hom_{H^0(\mathcal P_D/\mathcal N)}(U,V[n])\\
 &\xrightarrow{\ \sim\ }
 \Hom_{H^0(\mathcal P_{L(Q)})}
   (\overline{\mathcal F}U,(\overline{\mathcal F}V)[n])\\
 &\simeq H^n\mathcal P_{L(Q)}
   (\overline{\mathcal F}U,\overline{\mathcal F}V).
\end{aligned}
$$
The composite is $H^n(\overline{\mathcal F}_{U,V})$.
Hence $\overline{\mathcal F}$ is quasi-fully faithful. Moreover,
$\overline{\mathcal F}(qD)=L(Q)$ and
$$
 \operatorname{thick}_{H^0(\mathcal P_{L(Q)})}
   \bigl(\operatorname{Im}H^0(\overline{\mathcal F})\bigr)
 =\operatorname{thick}(L(Q))=\operatorname{per}(L(Q)^{\op}).
$$
Derived extension is therefore fully faithful on representable
modules and their localizing closure. Its image contains every
representable module up to shifts, cones and retracts.
Since representable modules generate the derived categories,
$\overline{\mathcal F}$ is a dg Morita equivalence.

The equality $\overline{\mathcal F}q=\mathcal F$ identifies the
restricted functor with $U\mapsto U\otimes_D L(Q)$.
At the generators $D$ and $L(Q)$, the bimodules for extension in
the two variables are ${}_{L(Q)} L(Q)_D$ and ${}_D L(Q)_{L(Q)}$.
Thus, for $N,N'\in\mathbf D(D^e)$ and a morphism $f:N\to N'$, derived
extension on bimodules is
$$
\begin{aligned}
 N&\longmapsto L(Q)\otimes_D^{\mathbf L}N\otimes_D^{\mathbf L}L(Q),\\
 f&\longmapsto
 1_{L(Q)}\otimes_D^{\mathbf L}f\otimes_D^{\mathbf L}1_{L(Q)}.
\end{aligned}
$$
These identifications follow from composition of the two tensor
functors, and are natural in $N$. For the diagonal bimodule $D$,
the induced map to the target diagonal is
$$
 L(Q)\otimes_D^{\mathbf L}D\otimes_D^{\mathbf L}L(Q)
 \simeq L(Q)\otimes_D^{\mathbf L}L(Q)
 \xrightarrow{\ \mathrm{mult}\ }L(Q),
$$
which is the isomorphism of \eqref{loc:epimorphism}.
The formula on morphisms also determines the image of any roof
whose denominator becomes invertible, by inverting the image of
its denominator.
$\square$

By Lemma~\ref{gen:actual}, the bimodule $P$ and the dg algebra
quasi-isomorphism $D\to\operatorname{End}_C(P)$ induce the triangle equivalence
$$
 \mathbf R\Hom_C(P,-):\mathbf D^b(C\text{-mod})
 \xrightarrow{\ \sim\ }\operatorname{per}(D^{\op}).
$$
The corresponding dg quasi-functor from
$\mathbf D^b_{\dg}(C\text{-mod})$ to $\mathcal P_D$ sends $U$
to a K-projective representative of $\mathbf R\Hom_C(P,U)$
\cite[Section~4.2]{Keller2006}.
For $U,V\in\mathbf D^b(C\text{-mod})$ and $n\in\mathbb Z$, its
map on degree-$n$ Hom cohomology is the isomorphism
$$
 \Hom_{\mathbf D^b(C\text{-mod})}(U,V[n])
 \xrightarrow{\ \sim\ }
 \Hom_{\mathbf D(D^{\op})}
  (\mathbf R\Hom_C(P,U),\mathbf R\Hom_C(P,V)[n]).
$$
It is essentially surjective on $H^0$ by the triangle equivalence.
Since $\mathbf R\Hom_C(P,e_iC^{\op})=M_i$, the equivalence sends
$$
 \operatorname{per}(C)=\operatorname{thick}\{e_iC^{\op}\}
 \quad\text{onto}\quad
 \operatorname{thick}\{M_i\}\simeq H^0(\mathcal N).
$$
The dg quasi-functor therefore induces an equivalence of the dg
quotients, identifying $\Sdg(C)$ with $\mathcal P_D/\mathcal N$ up to
quasi-equivalence. Composing with $\overline{\mathcal F}$ gives
a dg Morita equivalence with $\mathcal P_{L(Q)}$.
By \cite[Sections~3.2--3.3 and~4.6]{Keller2003} and
Lemma~\ref{lem:perfect-restriction}, the resulting isomorphisms are
$$
\begin{aligned}
 C^*((L(Q),d_\mu))
 &\xleftarrow{\ \sim\ }C^*(\mathcal P_{L(Q)})\\
 &\simeq C^*(\mathcal P_D/\mathcal N)
 \simeq C^*(\Sdg(C))
 \quad\text{in }\Ho_k(\Binf).
\end{aligned}
$$
Denote their composite, with the inverse of the first arrow, by
\begin{equation}\label{tar:Delta}
 \Delta_\mu:C^*((L(Q),d_\mu))\xrightarrow{\ \sim\ }C^*(\Sdg(C))
                    \quad\text{in }\Ho_k(\Binf).
\end{equation}
Define
$$
F_C:=\Delta_\mu\zeta_*\Upsilon_C:
 \Cbar_{\sg,L}^*(C^{\op})\longrightarrow C^*(\Sdg(C))
 \quad\hbox{in }\Ho_k(\Binf).
$$
Each factor is an isomorphism. It remains to prove
$$
 H^*(F_C)=\eta_C\mathrm{can}_{C,k}.
$$
By \eqref{canonical:roof}, we must compute the images of the
numerator and the syzygy denominator under Keller's coefficient
functor. We first compute the inverse denominator and then the numerator.
We retain the basic no-sink assumptions and the complex $P$ of \eqref{gen:dP}.
Define functors $N:\mathbf D((C^{\op})^e)\to\mathbf D(D^e)$ and
$\mathcal K:\mathbf D((C^{\op})^e)\to\mathbf D(L(Q)^e)$ by
$$
 N_X:=\mathbf R\Hom_{C}(P,P\otimes_{C^{\op}}^{\mathbf L}X),\qquad
 \mathcal K(X):=L(Q)\otimes_{D}^{\mathbf L}N_X\otimes_{D}^{\mathbf L}L(Q).
$$
The left $D$ action comes from the output $P$ and the right action from
precomposition on the input. By Lemma~\ref{loc:specified-dg-quotient},
$\mathcal K$ is the coefficient extension
of the dg functor defining $\Delta_\mu$, with unit
given by \eqref{loc:epimorphism}. We use $\mathcal K$ on bimodules and on their morphisms.
For objects $U,V$ of $\mathbf D^b_{\dg}(C\text{-mod})$,
Keller's corrected coefficient functor has value
$$
 (U,V)\longmapsto\mathbf R\Hom_{C}(U,V\otimes_{C^{\op}}^{\mathbf L}X);
$$
see \cite[Section~1]{Keller2019}. Restriction to $(P,P)$ and then
extension to the dg quotient give
$$
 \mathbf R\Hom_{C}(P,P\otimes_{C^{\op}}^{\mathbf L}X)=N_X,\qquad
 N_{C^{\op}}\simeq\operatorname{End}_{C}(P)\xleftarrow{\sim}D,\qquad
 \mathcal K({C^{\op}})\xrightarrow{\sim}L(Q).
$$
The arrow from $D$ is $\iota_P$ of Lemma~\ref{gen:actual};
the map to $L(Q)$ is induced by \eqref{loc:epimorphism}. These are the
coefficient and diagonal identifications defining $\eta_C$
in the model determined by $\Delta_\mu$.

Put $\Omega_p:=\Omega_{\mathrm{nc},R,E_C}^p({C^{\op}})[-p]$.
In its degree-zero coordinates $(J_C^{\op})^{\otimes_{E_C}p}\otimes_{E_C} {C^{\op}}$, write
$(u_1,\ldots,u_p;b)$ for $u_1\otimes\cdots\otimes u_p\otimes b$,
where $u_i\in J_C^{\op}$ and $b\in C^{\op}$.
We use Definition~\ref{def:relative-singular}, with the suspensions removed. The form sequence is
$$
 0\to\Omega_{p+1}\xrightarrow{j_p}{C^{\op}}\otimes_{E_C}\Omega_p
   \xrightarrow{\mathrm{act}}\Omega_p\to0,
 \qquad j_p(a\otimes z)=a\otimes z-1\otimes az.
$$
Here $a\in J_C^{\op}$ and $z\in\Omega_p$.
Units in balanced expressions mean their matching vertex idempotents.

Put $P^{(p)}:=P\otimes_{C^{\op}}\Omega_p$ and define the complex $T_p$ by
$$
 T_p^i:=\begin{cases}P^i,&i\leq-p,\\0,&i>-p,\end{cases}
 \qquad
 d_{T_p}^i:=\begin{cases}d_P^i,&i<-p,\\0,&i\geq-p.\end{cases}
$$
Let $\Theta_p:D\to G_p$ be the right dg $D$-module map obtained
by composing the first $p$ transitions, so
$\Theta_p(1)=\sum_{|I|=p}q_I I$ and
$\Theta_0=\id_D$.
\begin{Lem}\label{coef:finite-model}
For every $p\geq0$, concatenation identifies $P^{(p)}[p]$ with
$T_p$ as dg $D$--$C^{\op}$-bimodules. The right dg $D$-module
$G_p$ admits a left dg $D$-action and a dg bimodule
quasi-isomorphism
$$
 \iota_p:G_p\longrightarrow\Hom_C(P,T_p)\simeq N_{\Omega_p}[p].
$$
Using $\iota_P:D\xrightarrow{\sim}N_{C^{\op}}$, the map
$N(t_p):N_{C^{\op}}\to N_{\Omega_p}[p]$ is represented by
$(-1)^p\Theta_p:D\to G_p$.
After derived extension on both sides along $D\to L(Q)$, its
inverse is represented by the dg $D$-bimodule map
$$
 \beta_p:G_p\longrightarrow\operatorname{Res}_{D^e}L(Q),
 \qquad q_Ic\longmapsto(-1)^p I^*c
 \quad(I\in Q_p,\ c\in e_{t(I)}D).
$$
Here $\operatorname{Res}_{D^e}$ denotes restriction of bimodules
along $D\to L(Q)$; the target diagonal is identified by the
multiplication in \eqref{loc:epimorphism}.
\end{Lem}
{\it Proof.}
Concatenated word coordinates give
\begin{equation}\label{coef:tail}
 P^{(p)}\simeq T_p[-p],\qquad P^{(p)}[p]\simeq T_p.
\end{equation}
Indeed on a word with $q$ prefix letters and $p$ form letters, $q\geq1$,
the internal face at position $r$ has coefficient $(-1)^{q-r}$ and
the terminal face coefficient $(-1)^p$, by Definition~\ref{def:relative-singular}.
These are $(-1)^p$ times the coefficients in $P$ at length $q+p$.
At $q=0$ both outgoing differentials vanish.
The identifications are $D$--${C^{\op}}$ linear: the prefix deletion on $P$
has sign $(-1)^{q+p-1}$, while the left action after shift $[-p]$
contributes $(-1)^p$, leaving $(-1)^{q-1}$.

Tensoring the form sequence with $P$ gives an exact sequence of
complexes of right $C^{\op}$-modules
$$
 0\longrightarrow P^{(p+1)}
 \xrightarrow{\id_P\otimes j_p}P\otimes_{E_C}\Omega_p
 \xrightarrow{\mathrm{act}}P^{(p)}\longrightarrow0.
$$
Define a degree-zero right $C^{\op}$-linear section of the last map by
$$
\begin{aligned}
 \varsigma_p:P^{(p)}&\longrightarrow P\otimes_{E_C}\Omega_p,\\
 (a_1,\ldots,a_q;b)\otimes_{C^{\op}}z
 &\longmapsto(a_1,\ldots,a_q;1)\otimes_{E_C}bz,
\end{aligned}
$$
for $q\geq0$, $a_1,\ldots,a_q\in J_C^{\op}$,
$b\in C^{\op}$ and $z\in\Omega_p$. The product $bz$ is the
left form action; therefore the formula is balanced over
$C^{\op}$ and $\mathrm{act}\,\varsigma_p=\id$.
It suffices to compute on
$(a_1,\ldots,a_q;1)\otimes z$.
For $q\geq1$, the internal bar faces in
$d\varsigma_p-\varsigma_p d$ cancel, while the two terminal faces give
$$
\begin{aligned}
 &(d\varsigma_p-\varsigma_p d)
       ((a_1,\ldots,a_q;1)\otimes z)\\
 &\quad=(a_1,\ldots,a_{q-1};a_q)\otimes z
       -(a_1,\ldots,a_{q-1};1)\otimes a_qz\\
 &\quad=(\id_P\otimes j_p)
       \bigl((a_1,\ldots,a_{q-1};1)\otimes(a_q\otimes z)\bigr).
\end{aligned}
$$
Here $a_q\otimes z\in\Omega_{p+1}$ denotes its right-form
coordinate. For $q=0$ the difference is zero.
By the cone convention \eqref{pre:cone}, the connecting map to
$P^{(p+1)}[1]$ is represented by
$-(\id_P\otimes j_p)^{-1}(d\varsigma_p-\varsigma_p d)$.
After shifting the exact sequence by $p$,
Lemma~\ref{src:syzygy-stabilization} supplies the factor $(-1)^p$.
For $z=(u_1,\ldots,u_p;b)$, the concatenated word
$(a_1,\ldots,a_q,u_1,\ldots,u_p;b)$ is consequently sent to
$(-1)^{p+1}$ times the same word when $q\geq1$, and to zero
when $q=0$.
Under \eqref{coef:tail}, the resulting map is therefore
$(-1)^{p+1}\operatorname{pr}_{p,p+1}:T_p\to T_{p+1}$, where
$\operatorname{pr}_{p,p+1}$ is the truncation projection. Thus
$$
 a_0^P=\id_P,\qquad
 a_{p+1}^P=(-1)^{p+1}\operatorname{pr}_{p,p+1}a_p^P,
 \qquad \prod_{j=0}^{p-1}(-1)^{j+1}=(-1)^{p(p+1)/2},
$$
Write $\operatorname{pr}_p:P\to T_p$ for the truncation projection.
The image of $t_p$ after tensoring with $P$ is therefore
\begin{equation}\label{coef:actual-denominator}
 a_p^P=(-1)^{p(p+1)/2}\operatorname{pr}_p:
          P\longrightarrow T_p\simeq P^{(p)}[p].
\end{equation}
This calculation uses the usual tensor shift
$P\otimes X[p]\to(P\otimes X)[p]$, with sign $(-1)^{p|u|}$ on a
homogeneous prefix $u$; that sign must not be inserted a second time.

Use $G_p$ with differential \eqref{loc:stage-d}. For $I\in Q_p$
define a degree $-p$ right ${C^{\op}}$-linear operator
$$
 j_I(a_1,\ldots,a_q;b):=
     (-1)^{pq}(I,a_1,\ldots,a_q;b):P\longrightarrow T_p.
$$
Here $a_i\in J_C^{\op}$ and $b\in C^{\op}$. For $c\in e_{t(I)}D$,
the assignment $q_Ic\mapsto j_I\circ\iota_P(c)$ defines
\begin{equation}\label{coef:finite-map}
 \iota_p:G_p\longrightarrow\Hom_{C}(P,T_p).
\end{equation}
To check its differential, compute $d_Tj_I-(-1)^p j_Id_P$.
Faces internal to the input and the final coefficient face cancel.
A merge in the inserted prefix at position $r<p$ has coefficient
$(-1)^{pq+p+q-r}$. The corresponding operator
$j_{v_1,\ldots,v_{r-1},w,v_{r+2},\ldots,v_p,u}u$,
for a term $\lambda_{v_{r+1}v_r,w}j_w^{\op}$ of the merged product,
has coefficient
$(-1)^{pq+q-p-1}$; their ratio is $(-1)^{r-1}$.
A merge of the last inserted letter with the input first letter has
ratio $(-1)^{p-1}$. These are exactly the two sums in
\eqref{loc:stage-d}. At $q=0$ the top outgoing differential and all
the $u $ on the input are zero. Thus \eqref{coef:finite-map} is a
chain map in all degrees.

The bar-cokernel presentations identify $\Omega_p$, as a left
module, with ${C^{\op}}\otimes_{E_C}(J_C^{\op})^{\otimes_{E_C}p}$. It is finite projective.
Put $E^{(p)}:=E_C\otimes_{C^{\op}}\Omega_p$. Then $P^{(p)}\to E^{(p)}$ is a quasi-isomorphism.
Postcomposing \eqref{coef:finite-map} with
$\Hom_{C}(P,T_p)\to\Hom_{C}(P,E^{(p)}[p])$ gives a graded vector-space
isomorphism: $q_I\mathbf a $ reads a basis input word $\mathbf a$ and returns the basis
class of $(I;1)$ with an invertible sign. All other input lengths
are killed by projection to the top. K-projectivity of $P$ now proves
that $\iota_p$ is a quasi-isomorphism; the same matrix-unit description
proves its injectivity.

The left $D$-action on $G_p$ is
\begin{equation}\label{coef:left-action}
 v q_{v_1,\ldots,v_p}=
 \delta_{v,v_1}\sum_{u\in O_{t(I)}}
                      q_{v_2,\ldots,v_p,u}u \quad(p\geq1).
\end{equation}
For $p=0$ take the regular bimodule. On a positive-length input, the
two sides of left linearity for $\iota_p$ have exponents
$pq+q+p-1$ and $pq+q-p-1$, differing by $2p$. On length zero both
vanish by truncation. Thus $\iota_p$ preserves the left $D$-action. Its injectivity
and the dg Leibniz identity in $\Hom_C(P,T_p)$ imply the dg
Leibniz identity for \eqref{coef:left-action}. We have a dg bimodule quasi-isomorphism
$G_p\simeq N_{\Omega_p}[p]$.

On length $q\geq p$, the deletion $I $ has exponent
$pq-p(p+1)/2$ and the following insertion has exponent $p(q-p)$.
Consequently
$$
 \iota_p\Theta_p=(-1)^{p(p-1)/2}\operatorname{pr}_p.
$$
Together with \eqref{coef:actual-denominator}, this equality shows that the
actual denominator in $G_p$ is $(-1)^p\Theta_p$.

The map $S_p:G_p\to L(Q)$, $q_Ic\mapsto I^* c$, is a dg right map by
construction. It is left linear by \eqref{coef:left-action} and the
inverse relations. Since $G_p$ is finite semifree on the right,
$G_p\otimes_{D} L(Q)$ computes its derived extension, and $S_p$ extends
to a chain isomorphism with inverse
$$
 l\longmapsto\sum_I q_I\otimes I l.
$$
Indeed, writing $\overline S_p:G_p\otimes_{D} L(Q)\to L(Q)$ for the
extension and $v_p(l):=\sum_Iq_I\otimes I l$ for its proposed
inverse, the two inverse relations give, for $l\in L(Q)$ and $K\in Q_p$,
$$
 \begin{aligned}
 \overline S_pv_p(l)&=\sum_I I^* I l=l,\\
 v_p\overline S_p(q_K\otimes l)
     &=\sum_Iq_I\otimes I K^* l=q_K\otimes l.
 \end{aligned}
$$
Here $I K^*=\delta_{IK}e_{t(I)}$ for $I,K\in Q_p$.
Extending also on the left uses the specified multiplication
\eqref{loc:epimorphism}. The inverse of the transported denominator is
therefore represented by the following dg $D$-bimodule map, where
$\operatorname{Res}_{D^e}L(Q)$ denotes restriction along $D\to L(Q)$:
\begin{equation}\label{coef:inverse-denominator}
 \beta_p:G_p\longrightarrow\operatorname{Res}_{D^e}L(Q),
             \qquad q_Ic\longmapsto(-1)^p I^* c.
\end{equation}
Indeed, for every $c\in D$,
$$
 \beta_p\bigl((-1)^p\Theta_p(c)\bigr)
   =(-1)^{2p}\sum_{|I|=p}I^* I c=c.
$$
After extension, both maps are isomorphisms, so this equality
identifies $\beta_p$ with the inverse of that denominator.
$\square$

Let $a\in\Hom_{D^e}^n(K_D,G_p)$.
By Lemma~\ref{coef:finite-model}, extension of scalars followed
by $\beta_p$ gives a degree-$n$ cochain on $K_{L(Q)}$ with values
in $L(Q)$. For homogeneous $l,l'\in L(Q)$ and $x\in K_D$, its value is
\begin{equation}\label{coef:cochain-extension}
 (l\otimes x\otimes l')\longmapsto
            (-1)^{n|l|}l\beta_p(a(x))l',\qquad K_{L(Q)}=L(Q)\otimes_{D}K_{D}\otimes_{D}L(Q).
\end{equation}
The balancing relations are precisely graded bimodule linearity.
The Leibniz rule verifies compatibility with the Hom differential.
The preceding right extension and multiplication identify
\eqref{coef:cochain-extension} with the two-sided derived coefficient
extension, followed by the inverse of $\mathcal K(t_p)$.

We now compute the numerator in the same coefficient model.
The left and right form coordinates represent the same bar cokernel.
The bimodule isomorphism between these coordinates is
$$
 \chi_p:{C^{\op}}\otimes_{E_C}(J_C^{\op})^{\otimes_{E_C}p}\longrightarrow\Omega_p,
 \qquad b_0\otimes I\longmapsto b_0\cdot(I;1).
$$
For $b_0\in {J_C^{\op}}$, Definition~\ref{def:relative-singular} gives
\begin{align}
 \chi_p(b_0,u_1,\ldots,u_p)
 ={}&(b_0u_1,u_2,\ldots,u_p;1)\notag\\
 &+\sum_{r=1}^{p-1}(-1)^r
 (b_0,u_1,\ldots,u_ru_{r+1},\ldots,u_p;1)\notag\\
 &+(-1)^p(b_0,u_1,\ldots,u_{p-1};u_p).\label{cap:chi}
\end{align}
For $b_0\in E_C$ it is the corresponding endpoint action on $(I;1)$;
for $p=0$ it is the identity. In general the internal terms in
\eqref{cap:chi} cannot be discarded: for ${C^{\op}}=k[x]/(x^3)$,
$\chi_1(x,x)=(x^2;1)-(x;x)$.

Put $P_b:=E_C\otimes_{C^{\op}}\overline{\operatorname{Bar}}_{E_C}({C^{\op}})$.
The change of coordinates used in \eqref{gen:dP} is the chain isomorphism
$$
 \xi:P_b\longrightarrow P,\qquad
 (sa_1\cdots sa_q)\otimes b\longmapsto
       (-1)^{q(q+1)/2}(a_1,\ldots,a_q;b).
$$
In suspended coordinates define
$\Delta_P:P_b\to P_b\otimes_{C^{\op}}\overline{\operatorname{Bar}}_{E_C}({C^{\op}})$ by
$$
 \Delta_P((sa_1\cdots sa_q)\otimes b)
 :=\sum_{l=0}^q((sa_1\cdots sa_l)\otimes1)
          \otimes_{C^{\op}}(1\otimes(sa_{l+1}\cdots sa_q)\otimes b).
$$
It is a chain map: internal faces on each side of a cut agree, and
the two faces at a moving cut cancel. This includes all nonzero
products in ${J_C^{\op}}$. Let $\varepsilon:\overline{\operatorname{Bar}}_{E_C}({C^{\op}})\to {C^{\op}}$
be the multiplication map and put $r:=\id_{P_b}\otimes\varepsilon$.
The complex $P_b$ is bounded above with projective right ${C^{\op}}$-modules
as terms, hence is K-flat. Thus $r$ is a quasi-isomorphism, and
$r\Delta_P=\id_{P_b}$. In these coordinates $v $
is minus first-letter deletion. In $\Delta_P$, deletion of a prefix of
length $l\geq1$ matches the cut $l-1$ after deleting the input first
letter; the $l=0$ term is killed. Thus $\Delta_P$ is $D$-linear.

Let $m,p\geq0$ and let
$f:(J_C^{\op})^{\otimes_{E_C} m}\to {C^{\op}}\otimes_{E_C}(J_C^{\op})^{\otimes_{E_C}p}$ be $E_C$-bilinear.
With the forms placed in degree zero, $f$ is a degree-$m$
coefficient cochain. If $f$ is closed, applying $N$ to the morphism represented by
$f$ gives the roof
$$
 \begin{gathered}
 \Hom_{C}(P_b,P_b)
 \xleftarrow{\ r_*\ }
 \Hom_{C}(P_b,P_b\otimes_{C^{\op}}\overline{\operatorname{Bar}}_{E_C}({C^{\op}}))\\
 \xrightarrow{\ (\id\otimes\chi_pf)_*\ }
 \Hom_{C}(P_b,P_b\otimes_{C^{\op}}\Omega_p)[m].
 \end{gathered}
$$
The map $r_*$ is a quasi-isomorphism because $P_b$ is K-projective.
Its section is $(\Delta_P)_*$, so the roof is represented by
$(\id\otimes\chi_pf)\Delta_P$. The same formula defines a degree-$m$
map for every homogeneous cochain $f$; closedness is required only
when interpreting the formula as a morphism in the derived category.

For $U=(a_1,\ldots,a_q)$ with $a_i\in {J_C^{\op}}$ and $b\in {C^{\op}}$, put
$c_m:=(-1)^{m(m+1)/2}$. In the coordinates of $P$ its cap is
\begin{equation}\label{cap:terminal}
 \operatorname{cap}_f(U;b):=c_m\,(a_1,\ldots,a_{q-m})\otimes
       \chi_p f(a_{q-m+1},\ldots,a_q)\,b\quad(q\geq m),
\end{equation}
and zero for $q<m$. Indeed the tensor sign is $(-1)^{m(q-m)}$;
the input and prefix changes of coordinates contribute
$(-1)^{q(q+1)/2+(q-m)(q-m+1)/2}$. Their product is $c_m$.

Let $\delta_m f:=(-1)^{m+1}\delta_{\rm us}f$, where
$$
\begin{aligned}
 (\delta_{\rm us}f)(a_1,\ldots,a_{m+1}):={}&
 a_1f(a_2,\ldots,a_{m+1})\\
 &+\sum_{j=1}^m(-1)^j f(a_1,\ldots,a_ja_{j+1},\ldots,a_{m+1})\\
 &+(-1)^{m+1}f(a_1,\ldots,a_m)\cdot a_{m+1}.
\end{aligned}
$$
The map $\delta_m$ is the Hom differential for the unshifted coefficient module
and our suspended bar convention. The final action is the true
left-form right action, not just multiplication at its last letter.
Explicitly, for $p\geq1$, $b_0\in C^{\op}$ and
$a,b_1,\ldots,b_p\in J_C^{\op}$, the unshifted left coordinates give
$$
\begin{aligned}
 (b_0,b_1,\ldots,b_p)\cdot a
 ={}&\sum_{j=0}^{p-1}(-1)^{p-j}
 (b_0,\ldots,b_jb_{j+1},\ldots,b_p,a)\\
 &+(b_0,b_1,\ldots,b_{p-1},b_pa).
\end{aligned}
$$
In the $j=0$ summand $b_0b_1$ is the first coefficient; all other
displayed slots after that coefficient are radical letters.
For $a\in E_C$, the action is
$(b_0,b_1,\ldots,b_p)\cdot a=(b_0,b_1,\ldots,b_pa)$.
For $p=0$ it is $b_0a$. These formulas determine the action for
every $a\in C^{\op}$ by linearity. The formula for $a\in J_C^{\op}$
follows by eliminating
the final outer coefficient in the bar cokernel, and its image under
$\chi_p$ is multiplication by $a$ on the right coefficient.

Decompose $f=f_E+f_J$ according to its first coefficient in $E_C\oplus {J_C^{\op}}$.
For $0\leq l\leq q-m$, define $R_l^f(U;b)$ as follows. Expand
$$
 f_J(a_{l+1},\ldots,a_{l+m})
   =\sum_{b_0,\ldots,b_p}
       \gamma_{b_0,\ldots,b_p}\,b_0\otimes\cdots\otimes b_p,
$$
where $b_0,\ldots,b_p$ run through
$\{j_\alpha^{\op}:\alpha\in Q_1\}$. Set, in the concatenated coordinates of \eqref{coef:tail},
$$
 R_l^f(U;b):=\sum\gamma_{b_0,\ldots,b_p}
 (a_1,\ldots,a_l,b_0,\ldots,b_p,a_{l+m+1},\ldots,a_q;b).
$$
Define the degree-$(m-1)$ map $H_f:P\to P^{(p)}$ by
$$
 H_f(U;b):=\sum_{l=0}^{q-m}h_{m,p}(q,l)R_l^f(U;b),
$$
where $h_{m,p}(q,l):=c_m(-1)^{(m+p+1)(q-m-l)}$ for
$0\leq l\leq q-m$.
Only composable basis tensors occur. For $m=0$, evaluate $f_J$
on the vertex idempotent between $a_l$ and $a_{l+1}$, including
$l=0,q$. For $q<m$ the sum is zero.
For $\mathbf a=a_m\cdots a_1\in Q_m$, write
$f(\mathbf a):=f(j_{a_1}^{\op}\otimes\cdots\otimes j_{a_m}^{\op})$.
Write
$f_E(\mathbf a)=\sum_I c_{I,\mathbf a}(e_{s(I)}\otimes I)$,
where $I$ runs through basis words of length $p$ and $s(I)$ is
their initial vertex. For $p=0$, $I$ is an empty word at a vertex.
Put
$$
 X_f:=(-1)^m\sum_{I,\mathbf a}c_{I,\mathbf a}j_I\mathbf a :P\to P^{(p)}.
$$
Here the same insertion $j_I$ is read in the unshifted $P^{(p)}$, so its
degree is zero. On length $q$, this operator reads the first $m$
letters, with coefficient $c_m(-1)^{(m+p)(q-m)}$.

\begin{Prop}\label{cap:full-homotopy}
Let $m,p\geq0$ and let
$f\in\Hom_{E_C^e}((J_C^{\op})^{\otimes_{E_C} m},{C^{\op}}\otimes_{E_C}(J_C^{\op})^{\otimes_{E_C}p})$.
Use the right ${C^{\op}}$-module complexes $P,P^{(p)}$ of
\eqref{gen:dP} and \eqref{coef:tail}, with the coefficient forms
in degree zero. The maps $\operatorname{cap}_f,H_f,X_f:P\to P^{(p)}$
are homogeneous right $C^{\op}$-linear maps of degrees
$m,m-1,m$, respectively, and satisfy
$$
 \operatorname{cap}_f=d_{P^{(p)}}H_f-(-1)^{m-1}H_fd_P+H_{\delta_mf}+X_f.
$$
\end{Prop}
{\it Proof.}
First suppose $f=f_J$. Fix an output $(b_0,\ldots,b_p)$ of
$f(a_{l+1},\ldots,a_{l+m})$. Put $r:=m+p+1$, $s:=q-m-l$, and
$h:=c_m(-1)^{rs}$. We list all cancellations; positions in a word
are numbered starting at one.

The terms of $d_{P^{(p)}}H_f+(-1)^mH_fd_P+H_{\delta_mf}$
cancel in pairs. The last two columns give their coefficients.
$$
\begin{array}{c|c|c|c}
 \text{position}&\text{terms}&\text{first coefficient}&\text{second coefficient}\\ \hline
 i<l&d_{P^{(p)}}H_f,(-1)^mH_fd_P&h(-1)^{q-m+1-i}&h(-1)^{m+q-i}\\
 i\geq l+m+1&d_{P^{(p)}}H_f,(-1)^mH_fd_P&h(-1)^{q-p-i}&h(-1)^{m+q-i+r}\\
 \text{right coefficient},\ s>0&d_{P^{(p)}}H_f,(-1)^mH_fd_P&h(-1)^p&h(-1)^{m+r}\\
 \text{input merge }j&(-1)^mH_fd_P,H_{\delta_mf}&c_m(-1)^{1-j+(r+1)s'}&c_m(-1)^{j+(r+1)s'}\\
 \text{left boundary}&d_{P^{(p)}}H_f,H_{\delta_mf}&c_m(-1)^{(r+1)s+1}&c_m(-1)^{(r+1)s}\\
 0\leq j\leq p,\ s>0&d_{P^{(p)}}H_f,H_{\delta_mf}&c_m(-1)^{(r+1)s-j}&c_m(-1)^{(r+1)s-j-1}
\end{array}
$$
In the first row, deleting one input before $a_{l+1}$ changes
$(q,l)$ to $(q-1,l-1)$ and leaves $h$ unchanged.
In the fourth row, $\delta_mf$ has $m+1$ inputs and
$s':=q-m-1-l$. In the last row, $j<p$ is
an output merge and $j=p$ is the right boundary merge. Its second
coefficient uses $(-1)^{p-j}$ from the left-form right action and
$c_{m+1}=(-1)^{m+1}c_m$. Since $r=m+p+1$, the two coefficients
in every row have sum zero.
When $s=0$, there is no term of $H_{\delta_mf}$ with
$m+1$ inputs starting at $a_{l+1}$. The remaining
output merges have coefficients $c_m(-1)^j$, and the final
coefficient product has coefficient $c_m(-1)^p$. These are exactly
all the terms of \eqref{cap:chi} in \eqref{cap:terminal}.
This proves the assertion for $f_J$, including nonzero right
coefficients in ${J_C^{\op}}$.

Next suppose $f=f_E$, so $H_f=0$. Only the left action and the
first term of the right action of $\delta_mf$ have radical first
coefficient. The latter term is
$(-1)^p b_0b_1\otimes(b_2,\ldots,b_p,a)$ for $p>0$;
for $p=0$ it is $b_0a$. For the summand indexed by $l$, put $s:=q-m-l$. Then
these two contributions to $H_{\delta_mf}$ have coefficients
$+c_m(-1)^{(m+p+2)s}$ when $l\geq1$ and
$-c_m(-1)^{(m+p+2)s}$ when $s\geq1$. The contributions cancel for $0<l<q-m$.
The term at $l=q-m$ gives $\operatorname{cap}_f$ and the term at
$l=0$ gives $-X_f$. For $q=m$, $H_{\delta_mf}=0$ and
$\operatorname{cap}_f=X_f$; for $q<m$ all terms vanish.
For $m=0$, evaluation at the intervening vertex idempotents gives
the same cancellations for $0\leq l\leq q$. Linearity completes the proof. No two output parts have been
assumed separately closed.
$\square$

Restore the shift $[p]$ on forms and set $n:=m-p$.
For $a\in\Hom_{C}^r(P,P^{(p)})$, put
$$
 \mathsf s_p(a):=s^pa\in\Hom_{C}^{r-p}(P,P^{(p)}[p]),\qquad
 d\mathsf s_p(a)=(-1)^p\mathsf s_p(da).
$$
A cochain in
$\Hom_{D^e}(K_D,\Hom_C(P,P^{(p)}[p]))$ is determined by its
values on $e_i\otimes e_i$ and $su_v$.
Extend $\mathsf s_p(H_f)$ and $\mathsf s_p(\operatorname{cap}_f)$
to such cochains by assigning their vertex components to
$e_i\otimes e_i$ and zero to every $su_v$. Put
$$
 \mathcal H(f):=(-1)^p\mathsf s_p(H_f),\qquad
 \operatorname{cap}(f):=\mathsf s_p(\operatorname{cap}_f).
$$
The source differential in these coordinates is $d_s=(-1)^p\delta_m$.
The map $\operatorname{cap}$ commutes with the differentials;
its Hom differential has value zero on every $su_v$ because
$\Delta_P$ is $D$-linear. Therefore
$$
 U:=\operatorname{cap}-d\mathcal H-\mathcal H d_s
$$
defines a chain map from the left relative coefficient cochains
at stage $p$ to $\Hom_{D^e}(K_D,\Hom_C(P,P^{(p)}[p]))$, homotopic
to $\operatorname{cap}$. Indeed
$d\operatorname{cap}=\operatorname{cap}d_s$ and
$d^2=d_s^2=0$ give
$$
 \begin{aligned}
 dU-Ud_s
 &=d\operatorname{cap}-\operatorname{cap}d_s-d^2\mathcal H-d\mathcal H d_s
       +d\mathcal H d_s+\mathcal H d_s^2=0.
 \end{aligned}
$$
In particular, if $d_sf=0$, then
$$
 \operatorname{cap}(f)-U(f)=d\mathcal H(f),\qquad [\operatorname{cap}(f)]=[U(f)].
$$
By Proposition~\ref{cap:full-homotopy}, the restriction of
$U(f)$ to $D\otimes_{E_C}D$ is $\mathsf s_p(X_f)$.
The value of $U(f)$ on $su_v$, before shifting, is
$$
 (-dH_f)(su_v)=v H_f-(-1)^{m-1}H_fv .
$$
The quadratic terms of \eqref{loc:small-d} do not contribute here,
because $H_f$ vanishes on the $su_v$ coordinates.
For every noninitial replacement the two prefix deletions cancel:
$h_{m,p}(q,l)=h_{m,p}(q-1,l-1)$ and their signs are
$(-1)^{q-m}$ and $(-1)^{m+q-2}$. Only an initial replacement whose
first output letter is $v$ remains. Its coefficient is
$c_m(-1)^{(m+p+2)(q-m)}$, whereas that of $j_I\mathbf a $ is
$(-1)^{mq-m(m+1)/2+p(q-m)}$. Their ratio is $(-1)^m$.

Writing $f_J(\mathbf a)=\sum_{v,I}r_{v,I;\mathbf a}(j_v^{\op}\otimes I)$, we thus have
a cochain $U(f)\in\Hom_{D^e}^n(K_D,G_p)$:
$$
U(f)(1\otimes1)=(-1)^m\sum c_{I,\mathbf a}q_I\mathbf a ,
 \qquad U(f)(su_v)=(-1)^m\sum r_{v,I;\mathbf a}q_I\mathbf a .
$$
Injectivity of \eqref{coef:finite-map} pulls the chain identity back
to $\Hom_{D^e}(K_{D},G_p)$. In particular all the nonzero internal
differential terms in this finite complex are present.

Apply \eqref{coef:cochain-extension} to $U(f)$, using the inverse
of the same denominator. On
$\Hom_{L(Q)^e}^n(K_{L(Q)},L(Q))$ write
$z_i\in e_iL(Q)^ne_i$ for the value at $e_i\otimes e_i$ and
$z_v\in e_{t(v)}L(Q)^ne_{s(v)}$ for the value at $su_v$.
The corresponding elements $y_i\in e_iL(Q)^{n-1}e_i$ are
$$
 y_i=-\sum_{v\in O_i}v^* z_v,\qquad z_v=-v y_i.
$$
The inverse relations prove these equations are inverse to one another.
We obtain the canonical image of a closed total-degree $n$ cochain:
\begin{equation}\label{cap:canonical-coordinates}
 \left((-1)^n\sum c_{I,\mathbf a}I^* \mathbf a ,\quad
       (-1)^{n-1}s^{-1}\sum r_{v,I;\mathbf a}v^* I^* \mathbf a \right).
\end{equation}
The two coordinates need not be separately closed. Their differential
is obtained from the $d_K$; explicitly,
$$
 z'_i=d_\mu z_i,\qquad
 z'_v=d_\mu z_v-v z_i+(-1)^nz_{t(v)}v 
       -\sum_{a,b}\lambda_{ba,v} b z_a
       +(-1)^n\sum_{a,b}\lambda_{ba,v} z_b a .
$$
Thus no zero-differential small model has been substituted.

The formula passes through the left singular transition: adding one
common terminal letter increases $m,p$ by one and inserts
$\sum_u u^*u$ between $I^*$ and $\mathbf a$. The value and both signs
are unchanged. Since every singular cohomology class has a closed
representative at a finite stage, we obtain
\eqref{cap:canonical-coordinates} for every singular class, including
negative degrees. We next map \eqref{cap:canonical-coordinates} to the relative
bar cochains and identify its image with the first component
of the morphism constructed in Subsection~\ref{proof:isomorphism}.

 Denote the twisted morphism
\eqref{eq:twisted-G} by $G_\mu^{\opp}$, and let
$j:\Cbar_{E_C}^*((L(Q),-d_\mu))\to C^*((L(Q),-d_\mu))$ be the
ordinary relative Hochschild inclusion. Put
$$
 \Upsilon_\mu:=jG_\mu^{\opp}\tau_{E_C}:
 \Cbar_{\sg,L,E_C}^*({C^{\op}})\longrightarrow C^*((L(Q),-d_\mu)).
$$
By \eqref{ops:lambda} and \eqref{eq:singular-leavitt-morphisms},
$$
\Upsilon_C=\Upsilon_\mu\lambda_{C^{\op}}
 \qquad\text{in }\Ho_k(\Binf).
$$
Each factor of $\Upsilon_\mu$ is a $B_\infty$ quasi-isomorphism by
Lemma~\ref{lem:left-right-singular}, Proposition~\ref{prop:core-isomorphism}
and Lemma~\ref{lem:relative-hochschild-inclusion}.

We first identify the bar cochain represented by the two coordinates
in \eqref{cap:canonical-coordinates}. For $\delta=d_\mu$ or
$-d_\mu$, put
$$
 K_{(L(Q),\delta)}:=
 \operatorname{Cone}\bigl(\Omega_{(L(Q),\delta)}
              \longrightarrow (L(Q),\delta)\otimes_{E_C}(L(Q),\delta)\bigr).
$$
The differential on its universal forms is induced by $\delta$.
For $\delta=d_\mu$ this is $K_{L(Q)}$ constructed above.
There is a dg bimodule map commuting with the maps to $(L(Q),\delta)$,
$$
 \pi_\delta:\overline{\operatorname{Bar}}_{E_C}(L(Q),\delta)
       \longrightarrow K_{(L(Q),\delta)},\qquad
 (\pi_\delta)_0=\id,\quad
 (\pi_\delta)_1(1\otimes s\bar a\otimes1)=s\partial a,\quad
 (\pi_\delta)_r=0\quad(r\geq2).
$$
Indeed, the internal differential commutes with $\partial$, and the
external differential in lengths one and two gives respectively
$\partial a=a\otimes1-1\otimes a,\qquad \partial(ab)=(\partial a)b+a\partial b$.
The differential in larger lengths has zero image. Applying
$\Hom_{L(Q)^e}(-,L(Q))$ gives the chain map
$$
 \pi_\delta^*:\Hom_{L(Q)^e}(K_{(L(Q),\delta)},L(Q))
                   \longrightarrow\Cbar_{E_C}^*((L(Q),\delta)).
$$
Under the coordinate identification
$z_v=-v y_i$ for $v\in O_i$ and
$y_i=-\sum_{v\in O_i}v^* z_v$, this map is $\Phi_1$:
it fixes zero cochains and its one-cochain value is the
$E_C$-derivation whose value at $v $ is $-v y_i$.
These values agree with
Lemma~\ref{lem:clw-morphisms}{\rm(2)};
the inverse relations determine the values at all $v^* $.
Thus the map of bar cochains used for the coefficient calculation and
the first component used in Proposition~\ref{prop:core-isomorphism}
are the same graded map.

Let $f$ be a finite-stage left relative cochain with input length
$m$, form length $p$, and total degree $n=m-p$. Write
$$
 f_E(\mathbf a)=\sum_I c_{I,\mathbf a}(e_{s(I)}\otimes I),
 \qquad
 f_J(\mathbf a)=\sum_{v,I}r_{v,I;\mathbf a}(j_v^{\op}\otimes I).
$$
The first component of $\tau_{E_C}$ and the inverse of the path
map $\kappa$ both have sign
$(-1)^{\binom{m+1}{2}-\binom{p+1}{2}}$
\cite[Theorem~10.4]{ChenLiWang2025}. The two signs cancel.
Hence application of $\rho\kappa^{-1}(\tau_{E_C})_1$
gives the coordinates
\begin{equation}\label{tw:unsigned}
 \left(\sum_{I,\mathbf a}c_{I,\mathbf a}I^* \mathbf a ,
 \quad s^{-1}\sum_{v,I,\mathbf a}
                    r_{v,I;\mathbf a}v^* I^* \mathbf a \right).
\end{equation}
The path order follows from
$\rho(\gamma,\beta)=\beta^*\gamma$ and
$\rho(s^{-1}(\gamma,\beta))=s^{-1}\beta^*\gamma$
\cite[Lemma~11.1]{ChenLiWang2025}.
Primitive twisting preserves all Taylor components, so the same
formula computes the first component of $\Upsilon_\mu$ before $\Phi_1$
and the relative inclusion.

Conjugation by $\zeta(x)=(-1)^{|x|}x$ acts on a cochain $g$ with
$r$ inputs and total degree $n$ by $(-1)^{n-r}$. In fact, on
homogeneous inputs $sx_1,\ldots,sx_r$ the sign is
$$
 (-1)^{\,n+\sum_i(|x_i|-1)+\sum_i|x_i|}
       =(-1)^{n-r}.
$$
It therefore multiplies the zero and one coordinates in
\eqref{tw:unsigned} by $(-1)^n$ and $(-1)^{n-1}$.
The resulting coordinates are precisely
\eqref{cap:canonical-coordinates}.

\begin{Theo}\label{tw:marked-core}
Let $C$ be a basic split finite-dimensional algebra whose radical
quiver has no sinks. The morphism $F_C=\Delta_\mu\zeta_*\Upsilon_C$
is an isomorphism in $\Ho_k(\Binf)$ and satisfies
$$
 H^*(F_C)=\eta_C\mathrm{can}_{C,k}.
$$
\end{Theo}

{\it Proof.}
The morphism $\Upsilon_C$ is an isomorphism by
\eqref{eq:singular-leavitt-morphisms}, $\zeta_*$ is induced by a
dg algebra isomorphism, and $\Delta_\mu$ is an isomorphism by
Lemma~\ref{loc:specified-dg-quotient} and \eqref{tar:Delta}.
It remains to compute the induced map.

Use the canonical identifications $\mathrm{can}_{C,k}$ and
$\mathrm{can}_{C,E_C}$ defined in Subsection~\ref{pre:cochains}.
Proposition~\ref{src:mixed} and \eqref{ops:lambda} give
$$
                  \mathrm{can}_{C,E_C}H^*(\lambda_{C^{\op}})=\mathrm{can}_{C,k}.
$$
Let $f$ be a closed cochain of total degree $n$ at form stage $p$.
Then $\mathrm{can}_{C,E_C}[f]=t_p[n]^{-1}[f]$.
By Lemma~\ref{coef:finite-model}, derived extension on both sides
sends $(-1)^p\Theta_p$ to a representative of $\mathcal K(t_p)$,
and sends $\beta_p$ to a representative of its inverse, using
the diagonal identification \eqref{loc:epimorphism}.
Proposition~\ref{cap:full-homotopy}, with the shifts restored, gives
$[\operatorname{cap}(f)]=[U(f)]$ in the coefficient complex over $D$.
Applying the derived extension in \eqref{coef:cochain-extension}
therefore gives
$$
 \mathcal K(t_p[n]^{-1}[f])=[\beta_pU(f)]
 \quad\text{in }H^n\Hom_{L(Q)^e}(K_{L(Q)},L(Q)).
$$
Here $\beta_pU(f)$ denotes the extended cochain given by
\eqref{coef:cochain-extension} with $a=U(f)$.
By Lemma~\ref{loc:specified-dg-quotient}, $\mathcal K$ is the
coefficient functor for the dg functor defining $\Delta_\mu$.
By \eqref{canonical:definition}--\eqref{canonical:roof} and
\eqref{loc:epimorphism}, the first equality below holds.
The second follows from \eqref{tw:unsigned}, conjugation by $\zeta$,
and the map $\pi_{d_\mu}$ commuting with the maps to $(L(Q),d_\mu)$:
$$
\begin{aligned}
 H^*(\Delta_\mu)^{-1}\eta_C\mathrm{can}_{C,E_C}([f])
 &=[\beta_pU(f)]\\
 &=H^*(\zeta_*\Upsilon_\mu)([f]).
\end{aligned}
$$
Since cohomology commutes with filtered colimits of vector spaces,
every class has a closed representative at a finite stage.
Hence the calculation covers every
cohomology class. Finally,
$$
\begin{aligned}
 H^*(F_C)
 &=H^*(\Delta_\mu)H^*(\zeta_*\Upsilon_\mu)H^*(\lambda_{C^{\op}})\\
 &=\eta_C\mathrm{can}_{C,E_C}H^*(\lambda_{C^{\op}})
 =\eta_C\mathrm{can}_{C,k}.
\end{aligned}
$$
Under $\mathrm{can}_{C,k}$, this is $H^*(F_C)=\eta_C$.
$\square$

We next consider a basic split algebra $B$ whose radical quiver may have sinks.

\begin{Prop}\label{red:general-sink}
Let $B=E_B\oplus J_B$ be basic split. Let $h$ be a primitive
vertex idempotent with $J_Bh=0$, and put $e:=1-h$ and $B':=eBe\ne0$.
There are isomorphisms
$$
 \Gamma_h:\Cbar_{\sg,L}^*(B^{\op})\longrightarrow\Cbar_{\sg,L}^*((B')^{\op}),
 \qquad \beta_h:C^*(\Sdg(B))\longrightarrow C^*(\Sdg(B'))
$$
in $\Ho_k(\Binf)$ such that
$$
 H^*(\beta_h)\eta_B\mathrm{can}_{B,k}
       =\eta_{B'}\mathrm{can}_{B',k}H^*(\Gamma_h).
$$
\end{Prop}
{\it Proof.}
Since $Bh=kh$ and $eBh=0$,
$$
 B=\begin{pmatrix}B'&0\\hBe&k\end{pmatrix},\qquad
 B^{\op}=\begin{pmatrix}(B')^{\op}&eB^{\op}h\\0&k\end{pmatrix}.
$$
By \cite[Lemma~9.4]{ChenLiWang2025}, corner restriction is a
strict $B_\infty$-isomorphism
$$
 r_h:\Cbar_{\sg,L,E_B}^*(B^{\op})\xrightarrow{\sim}
            \Cbar_{\sg,L,eE_B}^*((B')^{\op}).
$$
Put $\Gamma_h:=\lambda_{(B')^{\op}}^{-1}r_h\lambda_{B^{\op}}$.
By \eqref{src:left-marking}, its cohomology map is the corner map.

Let $T_h:B\text{-mod}\to B'\text{-mod}$, $U\mapsto eU$, and let
$I_h:B'\text{-mod}\to B\text{-mod}$ be restriction along $B\to B'$.
For $U\in B\text{-mod}$ and a primitive vertex $e_i\ne h$, the sequences
$$
 0\longrightarrow hU\longrightarrow U\longrightarrow I_h(eU)\longrightarrow0,
 \qquad
 0\longrightarrow hBe_i\longrightarrow Be_i\longrightarrow I_h(B'e_i)\longrightarrow0
$$
are exact. The modules $hU$ and $hBe_i$ are sums of copies of $Bh$.
Moreover, $T_h(Bh)=0$ and $T_h(Be_i)=B'e_i$. Hence $T_h$ and $I_h$
preserve perfect complexes and induce inverse dg singularity equivalences.
Denote the induced Hochschild isomorphism by $\beta_h$.

We prove the equality in Proposition~\ref{red:general-sink}.
Write $P_B,D_B,G_p(B),P^{(p)}(B),T_p(B)$ for the constructions over $B$,
and use the subscript $B'$ for the same constructions over $B'$.
The finite modules $G_p$ and prefix insertions $j_I$ of the
coefficient calculation also make sense with sinks: each sum over
outgoing letters is allowed to be empty. The formulas for $d_{G_p}$ and
$d_Tj_I-(-1)^pj_Id_P$ remain valid. Indeed, for positive input length its
only expansion of a shortened prefix is
$\sum_u j_{K,u}u $; if its endpoint is a sink there are no
composable positive-length inputs, so both sides vanish. On input
length zero the outgoing differential is removed by the tail
truncation. Thus
$$
 G_p\longrightarrow\Hom_B(P_B,T_p(B))
$$
is the map $\iota_p$ of \eqref{coef:finite-map}, with $C$ replaced
by $B$, and is injective. Put
$\Omega_p(B):=\Omega_{\mathrm{nc},R,E_B}^p(B^{\op})[-p]$.
Postcomposition with the projection to
$\Hom_B(P_B,(E_B\otimes_{B^{\op}}\Omega_p(B))[p])$
is a graded isomorphism by evaluation on basis tensors.
Hence $\iota_p$ is a quasi-isomorphism.
Injectivity transfers the differential-square and left Leibniz
identities from the Hom complex. The radical-layer weight
filtration remains finite and semifree. This argument does not
assert that $D_B$ embeds in a Leavitt algebra when there are sinks.

There are identifications
$$
 eP_B=P_{B'},\qquad eP^{(p)}(B)=P^{(p)}(B'),\qquad eT_p(B)=T_p(B').
$$
All surviving words avoid $h$, so these equalities preserve every
internal multiplication and every shift sign. Prefix deletion
induces $D_B\to D_{B'}$. Since $D_Bh=kh$,
its kernel is $h D_B$, a right dg direct summand.
The sequence $0\to hD_B\to D_B\to D_{B'}\to0$
therefore proves that the specified multiplication
$D_{B'}\otimes_{D_B}^{\mathbf L}D_{B'}\to D_{B'}$ is a quasi-isomorphism.
Right extension of $G_p(B)$ deletes exactly the indices ending at
$h$. Its differential and left action become those of $G_p(B')$.
Using its right semifree filtration and then that multiplication gives
$$
 G_p(B)\otimes_{D_B}^{\mathbf L}D_{B'}\simeq G_p(B'),\qquad
 D_{B'}\otimes_{D_B}^{\mathbf L}G_p(B)
        \otimes_{D_B}^{\mathbf L}D_{B'}\simeq G_p(B').
$$
The identifications preserve the maps $j_I$.
To identify the functor induced by $T_h$, consider the natural transformation
$$
 \mathbf R\Hom_B(P_B,U)\otimes_{D_B}^{\mathbf L}D_{B'}
       \longrightarrow\mathbf R\Hom_{B'}(P_{B'},eU)
$$
which is the identity identification at $U=P_B$. It is therefore an
isomorphism for every $U\in\operatorname{thick}(P_B)
=\mathbf D^b(B\text{-mod})$. Applying the natural transformation
to the maps $j_I$ gives the displayed identifications for $G_p$.

Finally the deconcatenation diagonal restricts at every split
position to the diagonal over $(B')^{\op}$, and the full map $\chi_p$,
including its internal products, restricts to the same map over $(B')^{\op}$.
Thus the cap numerator restricts to that for $r_hf$.
The denominator is the signed tail projection
$(-1)^{p(p+1)/2}\operatorname{pr}_p$; exact restriction preserves
this projection and its sign. The functor $T_h$ also preserves the sequence of triangles
defining $t_p$. Two-sided
extension and the dg quotient therefore transport the numerator,
denominator and its inverse together. By \eqref{src:left-marking}, the maps $\lambda$ preserve
$t_p[n]^{-1}[f]$. The equality in Proposition~\ref{red:general-sink}
follows.
$\square$

\begin{Koro}\label{red:all-basic}
For every basic split finite-dimensional algebra $B$, there is an
isomorphism
$$
 F_B:\Cbar_{\sg,L}^*(B^{\op})\longrightarrow C^*(\Sdg(B))
 \quad\hbox{in }\Ho_k(\Binf)
$$
such that $H^*(F_B)=\eta_B\mathrm{can}_{B,k}$.
\end{Koro}
{\it Proof.}
Set $B_0:=B$. Successive deletion of sinks gives basic split algebras
$$
 B_0,B_1,\ldots,B_r=C,
 \qquad B_{i+1}=e_iB_ie_i,
$$
where $C=0$ or the radical quiver of $C$ has no sinks.
For a nonzero step, denote the morphisms of
Proposition~\ref{red:general-sink} by $\Gamma_i$ and $\beta_i$.
If $C\ne0$, take $F_C$ from Theorem~\ref{tw:marked-core} and define
$$
 F_{B_i}:=\beta_i^{-1}F_{B_{i+1}}\Gamma_i
 \qquad(i=r-1,\ldots,0).
$$
The compatibility in Proposition~\ref{red:general-sink} gives
$$
 \begin{aligned}
 H^*(F_{B_i})
 &=H^*(\beta_i)^{-1}\eta_{B_{i+1}}
       \mathrm{can}_{B_{i+1},k}H^*(\Gamma_i)\\
 &=\eta_{B_i}\mathrm{can}_{B_i,k}.
 \end{aligned}
$$
Induction gives the assertion for $B=B_0$.

If $C=0$, the radical quiver of $B$ is acyclic. Write
$B=E_B\oplus J_B$, where $J_B=\rad(B)$ and $E_B$ is the chosen
split semisimple subalgebra. If the quiver has $r$ vertices, then
$J_B^{\otimes_{E_B}r}=0$: a nonzero tensor would give a path of
length $r$ and repeat a vertex. The relative bar resolution is
bounded, so every simple $B$-module has finite projective dimension
and $\Sdg(B)$ has acyclic Hochschild cochains. The relative singular
cochains vanish in form stages $p\geq r$, and the absolute source
is acyclic by $\lambda_{B^{\op}}$. The zero $B_\infty$-morphism
between the acyclic complexes gives the assertion. The case $B=0$
has the same interpretation.
$\square$

Let $A$ be a finite-dimensional algebra and let $e\in A$ be an
idempotent satisfying $AeA=A$. Put $B:=eAe$. The Morita functors used below are
$$
 \begin{aligned}
 T_M:A\text{-mod}&\longrightarrow B\text{-mod},&U&\longmapsto eU,\\
 G_M:A^{\op}\text{-mod-}A^{\op}&\longrightarrow
             B^{\op}\text{-mod-}B^{\op},&X&\longmapsto eXe.
 \end{aligned}
$$
The multiplication maps
$$
 Ae\otimes_BeA\xrightarrow{\sim}A,\qquad
 eA\otimes_AAe\xrightarrow{\sim}B
$$
are bimodule isomorphisms; see \cite[Section~18]{Lam1999}.
In particular, $G_M(A^{\op})=B^{\op}$.

\begin{Lem}\label{red:morita-source}
There is an isomorphism
$$
 \Gamma_M:\Cbar_{\sg,L}^*(A^{\op})\longrightarrow
                  \Cbar_{\sg,L}^*(B^{\op})\qquad\text{in }\Ho_k(\Binf)
$$
whose induced map, under the singular cohomology identifications, is
$$
 (G_M)_*:\mathrm{HH}_{\sg}^*(A^{\op},A^{\op})\longrightarrow
              \mathrm{HH}_{\sg}^*(B^{\op},B^{\op}).
$$
\end{Lem}
{\it Proof.}
Apply \cite[Theorem~9.6]{ChenLiWang2025} to the algebras $A,B$
and the $A$-$B$-bimodule $Ae$. Tensoring with $Ae$ on the right
and on the left, respectively, induces the maps
$$
 \begin{aligned}
 \alpha_{\sg}^n:\mathrm{HH}_{\sg}^n(A,A)&\longrightarrow
       \Hom_{\mathbf D_{\sg}(A\otimes B^{\op})}(Ae,Ae[n]),\\
 \beta_{\sg}^n:\mathrm{HH}_{\sg}^n(B,B)&\longrightarrow
       \Hom_{\mathbf D_{\sg}(A\otimes B^{\op})}(Ae,Ae[n])
 \end{aligned}
$$
which are isomorphisms. By \cite[Proposition~9.14]{ChenLiWang2025},
the two restriction maps in the resulting zigzag of right singular
$B_\infty$-algebras have cohomology image
$$
 \ker(-\alpha_{\sg}^n,\beta_{\sg}^n)
 =\{(x,(\beta_{\sg}^n)^{-1}\alpha_{\sg}^n(x)):
                                      x\in\mathrm{HH}_{\sg}^n(A,A)\}.
$$
Thus the zigzag induces $(\beta_{\sg}^n)^{-1}\alpha_{\sg}^n$.
Apply the opposite $B_\infty$-construction and
Lemma~\ref{lem:left-right-singular}.
Equation~\eqref{pre:reversal-syzygy} identifies reversal of
$t_p[n]^{-1}[f]$ with $(t_p^{\op})[n]^{-1}[f^{\op}]$.
Since $(eXe)^{\op}=eX^{\op}e$, reversal commutes with the
Morita functor on the numerator, denominator and diagonal
bimodule. The resulting functor is $G_M$, so
$$
 \mathrm{can}_{B,k}H^*(\Gamma_M)=(G_M)_*\mathrm{can}_{A,k}.
$$
$\square$

\begin{Lem}\label{red:morita-marking}
Let $\beta_M:C^*(\Sdg(A))\to C^*(\Sdg(B))$ be the isomorphism
in $\Ho_k(\Binf)$ induced by the dg singularity equivalence of
$T_M$. Then
$$
 H^*(\beta_M)\eta_A\mathrm{can}_{A,k}
       =\eta_B\mathrm{can}_{B,k}H^*(\Gamma_M).
$$
\end{Lem}
{\it Proof.}
The functor $G_M$ preserves perfect bimodules because
$G_M(A^{\op}\otimes_kA^{\op})=eA^{\op}\otimes_kA^{\op}e$
is projective over $(B^{\op})^e$. For $U,V\in\mathbf D(A)$ and
$X\in\mathbf D((A^{\op})^e)$, the multiplication isomorphism
$A^{\op}e\otimes_{B^{\op}}eA^{\op}\simeq A^{\op}$ gives
$$
 T_M(V)\otimes_{B^{\op}}^{\mathbf L}G_M(X)
       \simeq T_M(V\otimes_{A^{\op}}^{\mathbf L}X).
$$
Derived full faithfulness of $T_M$ therefore yields
$$
 \mathbf R\Hom_B(T_M(U),T_M(V)\otimes_{B^{\op}}^{\mathbf L}G_M(X))
       \simeq\mathbf R\Hom_A(U,V\otimes_{A^{\op}}^{\mathbf L}X).
$$
These isomorphisms are natural in $U,V,X$ and preserve composition.
At $X=A^{\op}$ they preserve the diagonal: both contractions of
$A^{\op}e\otimes_{B^{\op}}eA^{\op}\otimes_{A^{\op}}A^{\op}e$
send $p\otimes q\otimes p'$ to $pqp'$, for
$p,p'\in A^{\op}e$ and $q\in eA^{\op}$.
Since $T_M$ preserves perfect modules, the natural isomorphism
between the two derived Hom complexes descends to the dg
quotients. Those Hom complexes are the values of $\mathcal E_A$
and $\mathcal E_B$ defining Keller's map
\cite[Section~1]{Keller2019}. If $f$ is the numerator and $t$
the denominator of a singular roof, its image is
$G_M(t)^{-1}G_M(f)$. Lemma~\ref{red:morita-source} now gives
$H^*(\beta_M)\eta_A\mathrm{can}_{A,k}
       =\eta_B\mathrm{can}_{B,k}H^*(\Gamma_M)$.
$\square$

{\it Proof of Theorem~\ref{thm:main}.}
If $A=0$, both complexes are acyclic and the zero
$B_\infty$-morphism gives the assertion. Assume $A\ne0$.
Choose a complete set of pairwise orthogonal primitive idempotents of $A$.
From this set choose one idempotent for each isomorphism class of
indecomposable projective left $A$-modules, and let $e$ be their sum.
Then $AeA=A$ and $B:=eAe$ is basic split
\cite[Section~18, Example~18.30 and Proposition~18.37]{Lam1999}.
By Corollary~\ref{red:all-basic}, there is an isomorphism
$$
 F_B:\Cbar_{\sg,L}^*(B^{\op})\xrightarrow{\sim}C^*(\Sdg(B)),
 \qquad H^*(F_B)=\eta_B\mathrm{can}_{B,k}.
$$
Put $F_A:=\beta_M^{-1}F_B\Gamma_M$ in $\Ho_k(\Binf)$.
Lemma~\ref{red:morita-marking} gives
$$
 H^*(F_A)
 =H^*(\beta_M)^{-1}\eta_B\mathrm{can}_{B,k}H^*(\Gamma_M)
 =\eta_A\mathrm{can}_{A,k}.
$$
After identifying the source by $\mathrm{can}_{A,k}$, this is
$H^*(F_A)=\eta_A$, as required in Theorem~\ref{thm:main}.
$\square$

\section*{Acknowledgements}

The research was partially supported by the National Natural Science Foundation of China (Grants 12671048 and 12401038).

{\footnotesize

Changchang Xi, School of Mathematical Sciences, Capital Normal
University, 100048 Beijing, P.
R.
China

{\tt Email: xicc@cnu.edu.cn (C.
C.
Xi)}

\medskip
Jinbi Zhang, School of Mathematical Sciences, Anhui University,
230601 Hefei, P.
R.
China

{\tt Email: zhangjb@ahu.edu.cn (J.
B.
Zhang)}
}

\end{document}